\documentclass[a4paper, oneside,11 pt]{article}

\usepackage[latin1]{inputenc}
\usepackage[T1]{fontenc}
\usepackage[english]{babel}

\usepackage{graphicx} 
\usepackage{tikz}
\usetikzlibrary{datavisualization} 
\usepackage{booktabs}

\usepackage{xspace}

\usepackage{tabularx}
\usepackage[absolute,overlay]{textpos}

\usepackage{wrapfig}

\usepackage[all]{xy}
\usepackage{verbatim}
\usepackage{amsmath}
\usepackage{amsthm}
\usepackage{amssymb}
\usepackage{epstopdf}
\usepackage{url}
\usepackage{amsfonts}
\usepackage{todonotes}
\usepackage{fullpage}
\newcommand{\mc}{\mathcal}
\newcommand{\pt}{\partial}

\newcommand{\br}{\mathbb{R}}

\newcommand{\bt}{\mathbb{T}}

\newcommand{\e}{\varepsilon}
\newcommand{\x}{\bar{x}}

\newcommand{\bb}{\mathbb}

\newtheorem{theorem}{Theorem}
\newtheorem{lemma}[theorem]{Lemma}
\newtheorem{cor}[theorem]{Corollary}
\newtheorem{prop}[theorem]{Proposition}
\newtheorem{defi}[theorem]{Definition}
\newtheorem{remark}[theorem]{Remark}

\newtheorem{hyp}[theorem]{Assumption}
\def\be{\begin{equation}}
\def\ee{\end{equation}}
\def\bea{\begin{eqnarray}}
\def\eea{\end{eqnarray}}
            
\newcommand{\NL}{{NL}}
\newcommand{\app}{\mathrm{app}}
\newcommand{\lin}{\mathrm{lin}}             
             
\newcommand{\T}{\mathbb{T}}            
\newcommand{\R}{\mathbb{R}}   
\newcommand{\N}{\mathbb{N}}   

\newcommand{\loc}{{\rm loc}}
\newcommand{\Lip}{{\rm Lip}}

\usepackage{enumerate}

\usepackage{bbm}

\newcommand{\cF}{{\mathcal{F}}}

\numberwithin{theorem}{section}
\numberwithin{equation}{section}

\newcommand{\TT}{\mathbb{T}}
\newcommand{\Z}{\mathbb{Z}}

\usepackage{enumitem}

\newlist{Hlist}{enumerate}{1}
\setlist[Hlist]{label=(H\arabic*)}

\newlist{Hlistprime}{enumerate}{1}
\setlist[Hlistprime]{label=(H\arabic*)'}

\newlist{Hlist2prime}{enumerate}{1}
\setlist[Hlist2prime]{label=(H\arabic*)''}

\usepackage{hyperref}
\usepackage{mathtools}
\mathtoolsset{showonlyrefs}

\newcommand*\di{\mathop{}\!\mathrm{d}}

\title{Stability in Quasineutral Plasmas with Thermalized Electrons}

\author{
Megan Griffin-Pickering 
  \thanks{University of Z\"urich, Institute of Mathematics, Winterthurerstrasse 190, 8057 Z\"urich, Switzerland. Email: \textsf{megan.griffin-pickering@math.uzh.ch}}
  \and
Mikaela Iacobelli
  \thanks{ETH Z\"urich, Department of Mathematics, R\"amistrasse 101, 8092 Z\"urich, Switzerland. Email: \textsf{mikaela.iacobelli@math.ethz.ch}
  }
}

\begin{document}

\maketitle


\begin{abstract}
We study the quasineutral limit for the ionic Vlasov--Poisson system with thermalized electrons (VPME) on the torus in dimensions one to three, for rough solutions with bounded spatial density. Our main result is a quantitative stability theorem showing that quasineutral convergence is robust under exponentially small perturbations of the initial data, as measured in Wasserstein distance: given a regular family of reference solutions for which the quasineutral limit is known to hold, we prove that the same limit remains valid for perturbed solutions on the same time interval.

The proof combines a kinetic-Wasserstein stability framework with a refined analysis of the Poisson--Boltzmann coupling specific to VPME.
A central new ingredient is an improved control of the characteristic flow: we obtain quantitative bounds on the growth of characteristics in the velocity coordinate, with only polynomial deterioration in the Debye length.
This yields new locally-uniform-in-time bounds on the spatial density and provides the key input needed to complete the stability estimates.

These results bring the stability theory for the ionic model in the quasineutral regime close to the known instability threshold and substantially relax the smallness conditions required in earlier works.
As a byproduct, our approach improves the moment assumptions in the global well-posedness theory for bounded-density solutions to VPME on the torus.
\end{abstract}

\section{General Overview} \label{sec:intro}

\subsection{Vlasov--Poisson type systems.}

The Vlasov--Poisson system is the classical kinetic model describing dilute, totally ionized, unmagnetized plasma. In its most common form, the unknown $f$ is the distribution function of the electrons moving in a self-induced electrostatic field, while the ions are assumed to act as a fixed background.
Here, instead, we consider solutions of the \emph{Vlasov--Poisson system for ions}, also known as the \emph{Vlasov--Poisson system with massless or thermalized electrons} (\emph{VPME}):
\begin{equation}
\label{eq:vpme}
(VPME)_\e:= \left\{ \begin{array}{ccc}\pt_t f_\e+v\cdot \nabla_x f_\e+ E_\e\cdot \nabla_v f_\e=0,  \\
E_\e=-\nabla_{x} U_\e, \\
\e^2\Delta_{x} U_\e=e^{U_\e}- \int_{ \br^d} f_\e\, \di v=e^{U_\e}- \rho_\e,\\
f_\e\vert_{t=0}=f_{0,\e}\ge0,\ \  \int_{\bt^d \times \br^d} f_{0,\e} \di x \di v=1.
\end{array} \right.
\end{equation}
At each time $t \geq0$ the phase-space is assumed to be $\bt^d\times\br^d$, where $\TT^d$ denotes the unit flat torus $(\br / \Z)^d$, and $f = f(t,x,v)$ is the probability distribution function of ions with position $x$ and velocity $v$ (since the evolution \eqref{eq:vpme} preserves both non-negativity and the total mass $\int_{\bt^d \times \br^d} f_\e(t,x,v) \di x \di v = \int_{\bt^d \times \br^d}  f_{0,\e}(x,v) \di x \di v$, solutions with probability densities as initial data remain probability densities at all later times).
The notations $\nabla_x$, $\nabla_v$ respectively denote gradients with respect to the $x$ and $v$ variables only; similarly $\Delta_x = \sum_{j=1}^d \partial^2_{x_j}$ denotes the Laplacian in the $x$ variable only.  
The parameter $\varepsilon$ stands for the \textit{Debye length} of the plasma, whose role in the plasma's stability will be clarified later (see Section~\ref{sec:intro_QN} below).
The \emph{electric field} $E_\e = E_\e(t,x)$ is a function of time and position only, and is induced by the distribution of charges (both ions and electrons) within the plasma itself.
In the VPME system, electrons are \emph{thermalized} and therefore distributed according to a Maxwell--Boltzmann law $e^{U_\e}.$
Indeed, as the mass ratio between an electron and a proton is of order $10^{-3}$, the disparity between the relative masses of an electron and an ion justifies the approximation that the electrons are in thermal equilibrium. 
The interested reader is directed to the survey \cite{GPI-WPproceedings} for a more detailed overview of the background to the VPME model, including a formal derivation in the massless electrons limit and a discussion of the progress on rigorous results in this direction, such as \cite{BGNS18}; see also the more recent work \cite{FlynnGuo}.

In the physics literature, the VPME system \eqref{eq:vpme} has appeared in applications to, for example, the formation of
ion-acoustic shocks \cite{Mason71, SCM, BPLT1991} and the expansion of ion plasma into vacuum \cite{Medvedev2011}.
See Gurevich and Pitaevsky \cite{Gurevich-Pitaevsky75} for an introduction to the model \eqref{eq:vpme} from the point of view of astrophysics.

The Vlasov--Poisson system for ions has drawn interest from the mathematical community comparatively more recently than the more well-known Vlasov--Poisson system for electrons \cite{Vlasov}
The critical difference is that in the ion model \eqref{eq:vpme}, the electrostatic potential $U$ satisfies the nonlinear \emph{Poisson--Boltzmann} equation $\e^2 \Delta U = e^U - \rho$, rather than the linear Poisson equation $\e^2 \Delta U = 1 - \rho$. The exponential nonlinearity introduces several mathematical difficulties in the ion case that are not present in the electron case: the lack of an explicit convolution representation for the electric field precludes the use of techniques that rely on such a representation, such as were crucial in establishing global well-posedness in 3D $\bt^3$ \cite{Pfaffelmoser, Batt-Rein}, while many quantitative estimates are subject to exponential divergence of constants due to the form of the nonlinearity.

The effect of this can be seen, for example, in the respective development of the well-posedness theory for the electron and ion models. For the electron model, global existence of weak solutions was proved by Arsenev \cite{Arsenev} in the 70s, while the global well-posedness of classical solutions was obtained in dimension $d=2$ in the 70s by Ukai and Okabe \cite{Ukai-Okabe} and in $\br^3$ around the 90s by Pfaffelmoser \cite{Pfaffelmoser} and Lions and Perthame \cite{Lions-Perthame} using two different methods.
Since then, these results have been extended to the three-dimensional torus $\bt^3$ \cite{Batt-Rein}, and many works have refined the assumptions and techniques, such as \cite{Schaeffer, Loeper, Pallard, ChenChen} (this list is non-exhaustive, see \cite{GPI-WPproceedings} for a more detailed discussion about the well-posedness of Vlasov type systems).
In contrast, the solution theory for the ion model was developed only more recently: weak global solutions in $\br^3$ were obtained in the 90s by Bouchut \cite{Bouchut}, while global well-posedness theory for classical solutions in two and three dimensions was the subject of a series of recent works by the authors \cite{GPI-WP, GPI-WP-R3} ($x \in \bt^3, \br^3$) and Cesbron and the second author \cite{CesbronIac} ($x$ in bounded domains).

The first step in understanding the dynamics of the ion model has been to study the equation with linearized Poisson coupling, which led to the study of the 
 (partially) \emph{linearized VPME} system where the Vlasov equation is coupled with $- \e^2\Delta U_\e + U_\e = \rho_\e - 1$ (note however that even with this linear Poisson coupling, the resulting Vlasov system is still nonlinear and thus not truly `linearized').
This system is closely related to the \emph{screened Vlasov--Poisson} system, up to a difference in the scaling with respect to the Debye length: in the screened Vlasov--Poisson system, the potential $U$ satisfies $- \Delta U_\e + \e^{-2} U_\e = \rho_\e - 1$. Thus for $\e=1$ the screened Vlasov--Poisson and linearized VPME systems coincide. For more information about linearized VPME and screened VP see also \cite{SchaefferScreened, HKR, BedrossianMasmoudiMouhot, TrinhNguyen,  HKNguyenRousset2021, HoferWinter}. 

\subsection{The Quasineutral Limit and Kinetic Euler Equations} \label{sec:intro_QN}
Since plasmas are highly conductive, any developed charge imbalances are readily screened; thus, they can be treated as quasineutral at large scales. Conversely, quasineutrality is no longer verified at small spatial and time scales. The Debye length $\lambda_D$ is the distance over which quasineutrality may break down, and it varies according to the physical characteristics of the plasma.

The Debye length is usually considerably short compared to the typical observation scale -- in fact, this property is considered by some authors to be one of the defining characteristics that distinguishes a \emph{plasma} from a more general ionized gas \cite[Chapter 1, Section 1.4]{Chen}. Therefore, we can define the parameter $\e := \lambda_D/L$ and consider the limit as $\e$ tends to zero. This procedure is known as {\em quasineutral limit}. For the ion model, the limit is formally identified by setting $\e = 0$ in system \eqref{eq:vpme}.
This results in the following equation, known as the \emph{kinetic isothermal Euler} equation (KIsE):
\be \label{eq:KE-iso}
(KIsE) :=
\begin{cases}
\partial_t f + v \cdot \nabla_x f - \nabla_x U \cdot \nabla_v f = 0, \\
U = \log{\rho_f}, \\
f \vert_{t=0} = f_0, \, \, \,  \int_{\bt^d \times \br^d} f_0 (x,v) \di x \di v = 1.
\end{cases}
\ee

Strictly speaking, the relation $U=\log \rho_f$ is meaningful only in regimes where $\rho_f$ remains strictly positive, so \eqref{eq:KE-iso} should be understood in this class. In particular, the well-posedness results invoked later concern sufficiently regular solutions for which the spatial density stays bounded away from zero; see in particular the discussion below and \cite{HKR}.

KIsE is an example of a more general class of \emph{kinetic Euler} systems: a type of Vlasov equation with `very singular' potential that reduces to an (incompressible or compressible) Euler system in the case of \emph{monokinetic} solutions $f(t,x,v) = \rho(t,x) \delta_{v - u(t,x)}$ \cite{GPI-WPproceedings2}. These models arise, for example, as formal quasineutral limits from Vlasov--Poisson systems, but have also appeared in other contexts.
In the seminal paper \cite{Brenier1989}, Brenier considered the \emph{kinetic incompressible Euler} system (KInE) as a kinetic formulation of the incompressible Euler equations. In this case, the force $E=-\nabla_x U$ is implicitly defined through the incompressibility constraint $\rho=1$, and may be considered a Lagrange multiplier associated with this constraint. In particular, (KInE) also arises in the quasineutral limit of the electron Vlasov--Poisson system.

Another example of a kinetic Euler system is the \emph{Vlasov--Dirac--Benney} system (VDB) where the acceleration in the Liouville equation is $U = \rho_f.$
Bardos named this system VDB \cite{Bardos} due to a connection with the Benney equations for water waves through a formulation by Zakharov \cite{Zakharov}.
The VDB equation demonstrates perhaps most clearly the interpretation of kinetic Euler systems as Vlasov equations with very singular potential, since the potential $U$ may formally be written as $U = \delta \ast \rho_f$ \cite{Bardos, Bardos-Nouri}.

The VDB system is formally obtained in the quasineutral limit from the linearized VPME system.
In general, in performing the formal quasineutral limit, we pass from a transport system where the force field is given by a (possibly nonlinear) elliptic equation to a transport-type system coupled to a singular force field. Thus, it is clear that the Cauchy theory for the limit systems and the quasineutral limit are intimately related. It was shown in \cite{HKR} that VDB is locally well-posed in sufficiently high Sobolev regularity, provided that a Penrose-type stability criterion is satisfied (see Section~\ref{sec:PenroseStable} below). As explained in the introduction of \cite{HKR}, their techniques can be extended to prove similar local well-posedness for KIsE \eqref{eq:KE-iso} for strictly positive initial data. For a fuller discussion about these models and their Cauchy theory, see also the survey \cite{HanKwanHDR} and the research papers \cite{Bardos, Bardos-Nouri, Bardos-Besse, Jabin-Nouri}. On the other hand, the loss of derivatives in the limiting system is reflected in the presence of spectral instabilities in the linearized system and consequent ill-posedness of the complete system around any smooth linearly unstable profile \cite{Han-Kwan-Nguyen, Baradat}. Similar strong instability issues are expected to affect KIsE, since its linearized theory around spatially homogeneous profiles is identical to that of VDB.

Before concluding our digression on the limiting systems, let us mention that the VDB system also appears as the semiclassical limit of an infinite dimensional system of coupled nonlinear Schr{\"o}dinger equations \cite{Bardos-Besse, Bardos-Besse2015, Bardos-BesseSC}.
For a discussion about semiclassical limits involving the KIsE model, see \cite{Carles-Nouri, Ferriere}.
See also \cite{ PuelCPDE, PuelM2AN} for combined semiclassical and quasineutral limits.

\begin{paragraph}{Previous results on the quasineutral limit.}

The mathematical study of the quasineutral limit can be traced back to the first pioneering works of Brenier and Grenier \cite{BG, Grenier95} on the electron model, which used an approach based on defect measures and gave a mathematically rigorous description of the `plasma oscillations' which appear in the electron case. Grenier \cite{Grenier99} then showed further that the limit holds in the sense of strong convergence, in one dimension, for smooth `single bump' type profiles.
This structural assumption is critical in understanding the quasineutral limit.

Indeed, as observed by Grenier \cite{Grenier96, Grenier99},
an instability mechanism inherent to the physics, known as the \emph{two-stream} instability, 
presents an obstruction to the quasineutral limit.
It is well-known in plasma physics \cite{Chen} that velocity distributions with multiple sharp peaks, such as a beam injected into a bulk of lower energy plasma, are unstable profiles.
Solutions evolving from initial data that are perturbations of this `double bump' form exhibit phase-space vortices.
This behaviour is observed in both electron \cite{RobertsBerkPRL} and ion \cite{BPLT1991} models.
Mathematically this corresponds to the linearization of the Vlasov equation around this profile having an exponentially growing mode.
The connection between these growing modes and the structure of the distribution was investigated for the electron model by Penrose \cite{Penrose}, who gave a stability criterion that shows in particular that profiles with a single maximum have no exponentially growing modes, while profiles with sufficiently sharp minima, such as certain double-bump profiles, do have exponentially growing modes.
It is then reasonable to expect that these modes are an obstacle to the quasineutral limit due to the connection with a long-time limit.
Indeed, Han-Kwan and Hauray \cite{HKH} used these unstable modes to construct counterexamples to the quasineutral limit in arbitrarily high Sobolev regularity.
For ions, analogous criteria have been obtained in \cite{HKH, HKR}.

Other positive results were obtained for the electron model  in the `cold electron' case where the velocity distribution is a Dirac mass (a kind of `extreme single bump') by Brenier \cite{Brenier2000} and Masmoudi \cite{Masmoudi2001}. Han-Kwan \cite{Han-Kwan2011} obtained the limit for VPME \eqref{eq:vpme} in the corresponding `cold ions' setting.
For general data without structural conditions, a significant result was obtained by Grenier \cite{Grenier96}, showing that the quasineutral limit holds for initial data with uniformly analytic spatial regularity.
In a later breakthrough, Han-Kwan and Rousset \cite{HKR} proved the quasineutral limit from the linearized VPME system to Vlasov--Dirac--Benney in Sobolev regularity, under a Penrose-type structural condition. 
The quasineutral limit has also been studied in the context of magnetised plasmas \cite{Golse-SR2003, PuelSR}.

More recently, a new line of research has investigated the stability of the existing results on the quasineutral limit, when the initial data are perturbed with possibly rough perturbations, vanishing in the limit as $\e$ tends to zero.

\end{paragraph}

\subsection{The quasineutral limit with rough initial data}
\label{sect:QN rough}

The results we have discussed so far that prove the quasineutral limit for Vlasov--Poisson systems apply to initial data with high (analytic/Sobolev) regularity. However, both the ion and electron Vlasov--Poisson systems are globally well-posed (for $d\leq3$) for initial data with much lower regularity: e.g. for $d=3$, $(1 + |v|^k) f_{0, \e} \in L^1_+ \cap L^\infty (\bt^3 \times \br^3)$ for some $k > 6$ is known to suffice for the existence of a unique global solution with (locally-in-time) bounded spatial density $\rho_{f_\e} \in L^\infty_\loc \left ([0, + \infty) ; L^\infty(\bt^3) \right )$ \cite{Loeper, Pallard, GPI-WP}, and in fact the techniques of this article can be used to show that the moment requirement can be further lowered to $k > 3$. This raises the question of what happens to these solutions in the quasineutral regime.

A further motivation for considering such rough perturbations comes from particle approximations: empirical measures are naturally very irregular objects, and Wasserstein distances provide a natural way to compare them with smooth reference profiles. In particular, in dimension $d=1$, `weak-strong' stability estimates (such as those we will prove in this article) can be applied directly to empirical measure solutions of \eqref{eq:vpme}. This gives an additional reason to seek quasineutral stability estimates that remain valid under very weak assumptions on the perturbation.

The investigation of the quasineutral limit for rough data originated from the work of Han-Kwan and the second author in the one-dimensional case \cite{IHK1}. The idea is to consider data $f_{0,\e}$ as perturbations around a distribution $g_{0,\e}$ that satisfies the quasineutral condition, such as uniformly analytic distributions where Grenier's result \cite{Grenier96} can be applied; or, in the ionic case, sufficiently smooth distributions satisfying a uniform Penrose-type stability condition as in \cite{HKR}. The perturbed data $f_{0,\e}$ takes the form:
\be
f_{0,\e} = g_{0,\e} + h_{0,\e},
\ee
where the perturbation $h_{0,\e}$ is such that $f_{0,\e} \in L^\infty$: i.e., if $g_{0,\e}$ is smooth, then $h_{0,\e}$ may be $L^\infty$.
We may then ask: for which perturbations $h_{0,\e}$ is quasineutrality still valid, in the sense that the quasineutral limit holds for solutions with the initial data $f_{0,\e}$? Can we ensure that quasineutrality remains valid by taking $h_{0,\e}$ `sufficiently small'?

To formulate this question more precisely, the
magnitude of the perturbation is measured using Monge-Kantorovich (Wasserstein) distances (see Definition~\ref{def:Wp} below): we suppose that, for some function $\eta: \mathbb{R}_+ \rightarrow \mathbb{R}_+$ (say, non-decreasing, with $\lim_{\e \to 0} \eta(\e) = 0$),
\be \label{est:data-Wp}
W_p(f_{0,\e}, g_{0,\e}) \leq \eta(\e),
\ee
typically with $p$ chosen as $1$ or $2$. The objective is then to identify admissible functions $\eta$ such that the assumption \eqref{est:data-Wp} implies that the quasineutral limit holds for solutions with initial data $f_{0,\e}$. We emphasize that, because Wasserstein distances metrize a weak topology, the functions $f_{0,\e}$ can be very rough even if the functions $g_{0,\e}$ are smooth, and highly oscillatory perturbations are permitted. 

\subsubsection{Electrons}

In \cite{HKH}, Han-Kwan and Hauray demonstrated that the quasineutral limit can fail with $\eta(\epsilon)\sim \epsilon^N$ for any given $N>0$, even with $g_{0,
\e}$ chosen in ``the best possible way'' as a spatially-homogeneous stationary solution $\mu(v)$. In other words, the validity of the quasineutral limit is {\em unstable} (not preserved) under polynomially small perturbations of analytic data. In fact, the result is false for data $f_{0,\e}$ such that $h_{0,\e}$ is polynomially small in an arbitrarily strong Sobolev space, see also Remark~\ref{rmk:almost opt} below.
 
 In contrast to this negative outcome, Han-Kwan and the second author \cite{IHK1} established the validity of the quasineutral limit for the one-dimensional electron Vlasov--Poisson system under the condition 
\be \label{est:data-vp-1d}
W_1(f_{0,\e}, g_{0,\e}) \leq \exp{(- C \e^{-1})}.
\ee
This exponential form is essentially optimal, since the counterexamples from \cite{HKH} show that no analogous statement can be true for any polynomially vanishing $\eta$.

In the higher-dimensional setting $d=2,3$, in \cite{IHK2} it was shown that the quasineutral limit holds under the condition:
\be
W_2(f_{0,\e}, g_{0,\e}) \leq \left [\exp{\exp (C \e^{-\zeta})} \right ]^{-1},
\ee
where $\zeta > 0$ is an exponent that depends on the dimension.
More recently, in \cite{Iac22}, the second author improved upon the previous result and achieved the validity of the quasineutral limit under the essentially optimal condition:
\be
W_2(f_{0,\e}, g_{0,\e}) \leq {\exp (-C \e^{-\zeta})}.
\ee

We summarize these results in Table~\ref{tab:electrons}.

\begin{table}[!h]
\centering

\begin{tabular}{|c|c|c|}
        \hline  & $d=1$ & $d=2,3$  \\
        \hline  Stability&  {$\exp{(-C \e^{-1})}$ \,\,\,{\tiny{ Han-Kwan \& Iacobelli \cite{IHK1}}} } &  {$\left [\exp \exp{(C \e^{-\zeta(d)})} \right ]^{-1}$ \,\,\, {\tiny{Han-Kwan \& Iacobelli  \cite{IHK2}  } } }\\
        &&
$\exp{(-C \e^{-\zeta(d)})}$
\,\,\,
{\tiny{Iacobelli \cite{Iac22}}}
  \\
        \hline  Instability & $\e^N$  \,\,\,{\tiny{Han-Kwan \& Hauray \cite{HKH}}} & $\e^N$  \,\,\,{\tiny{Han-Kwan \& Hauray \cite{HKH}}}  \\
        \hline
\end{tabular}

\caption{Stability properties of the quasineutral limit for {\bf electrons} under rough perturbations.} \label{tab:electrons}
\end{table}

\subsubsection{Ions}
\label{sect:QN-VPME}
As for the electron case, the quasineutral limit for the ionic Vlasov--Poisson system \eqref{eq:vpme} is false for polynomially small perturbations: this was shown by Han-Kwan and Hauray \cite{HKH} for ions with linearized coupling, and we show in Appendix~\ref{app:Instability} that this method can be extended to the nonlinear coupling. Thus at least exponential smallness must be required.
In \cite{IHK1} Han-Kwan and the second author established the validity of the quasineutral limit for the one-dimensional ionic Vlasov--Poisson system \eqref{eq:vpme} under the condition
\be
W_1(f_{0,\e}, g_{0,\e}) \leq \left [\exp{\exp (C \e^{-2})} \right ]^{-1},
\ee
where $g_{0,\e}$ are initial data for which the quasineutral limit holds, for example uniformly analytic distributions as in \cite{Grenier96}, or Sobolev data satisfying a uniform Penrose condition as in \cite{HKR}.

Later, in \cite{GPI-MFQN}, the authors proved the quasineutral limit in dimensions $d=2,3$ under restrictive assumptions on the smallness of the perturbation, expressed as:
\be
\label{eq:4exp}
W_2(f_{0,\e}, g_{0,\e}) \leq \left [\exp \exp \exp \exp{(C \e^{-2})} \right ]^{-1}.
\ee
Finally, in this paper we have succeeded in completing the program by obtaining the validity of the quasineutral limit for ions under an essentially optimal exponential smallness assumption, as for the electron case; see Table~\ref{tab:ions} (and compare Table~\ref{tab:electrons}).

\begin{table}[h]
\centering

\begin{tabularx}{\textwidth}{|c|X|X|}
\hline  & $d=1$ & $d=2,3$  \\
        \hline  Stability&   {$\left [\exp \exp{(C \e^{-2})} \right ]^{-1}$  \newline {\tiny{Han-Kwan \& Iacobelli  \cite{IHK1}}} } &   {$\left [\exp \exp \exp \exp{(C \e^{-2})} \right ]^{-1}$ \newline {\tiny{Griffin-Pickering \& Iacobelli \cite{GPI-MFQN}} }}    \\
        &{ $\exp(- C \e^{-2})$} \newline {\tiny{Theorem~\ref{thm:main} } }&
{$\exp(- C \e^{-\zeta(d)})$} \newline {\tiny{Theorem~\ref{thm:main}  } }
  \\
        \hline  Instability & $\e^N$  \,\,\,{\tiny{Han-Kwan \& Hauray \cite{HKH}}} & $\e^N$  \,\,\,{\tiny{Han-Kwan \& Hauray \cite{HKH}}; Appendix~\ref{app:Instability}}  \\
        \hline
\end{tabularx}

\caption{Stability properties of the quasineutral limit for {\bf ions} under rough perturbations.}
\label{tab:ions}
\end{table}

For more details on the results discussed above, we refer to the survey \cite{GPI-WPproceedings2}.

\subsection{Main Result}

We begin by stating our main result, which establishes the stability of the quasineutral limit under sufficiently small rough perturbations, in the general setting.
Following this, we state corollaries of our result in two specific instances in which the quasineutral limit is known to hold: the analytic setting in the style of Grenier \cite{Grenier96}, and the Penrose stable case considered by Han-Kwan and Rousset \cite{HKR}.

In the following statement, and throughout the paper, $W_1$ denotes the first-order Wasserstein distance (see Definition~\ref{def:Wp} below). 
Our main result holds under the following hypotheses.

\begin{hyp} \label{hyp:main}
Let $1\leq d \leq3$.
Assume that there exists $T_\ast, \e_\ast > 0$ and, for each $\e \in (0, \e_\ast]$, a weak solution $g_\e$ of the Vlasov--Poisson system for ions \eqref{eq:vpme} on the time interval $[0, T_\ast]$ with initial datum a probability measure $g_{0,\e}$, such that the following hypotheses are satisfied.

\begin{Hlist} \item \label{hyp:regular-moments} For each $\e \in (0, \e_\ast]$, the spatial density $\rho[ g_\e ]$ belongs to the space $L^1 \left ( [0, T_\ast] ; L^\infty(\bt^d) \right )$, and there exists $C_0 > 0$ such that
\be \label{hyp:density-uniform}
\sup_{\e \in (0, \e_\ast]} \int_0^{T_\ast} \left \| \rho[ g_\e ] \right \|_{L^\infty(\bt^d)} \leq C_0 .
\ee
In the case $d=2$ or $3$, we assume furthermore that, for some $j_0 > 2$, the velocity moments of order $j_0$ are bounded uniformly in $\e$:
\be
\sup_{\e \in (0, \e_\ast]}  \int_{\bt^d \times \br^d}  |v|^{j_0} \di g_{0,\e} (x,v) \leq C_0 .
\ee

\end{Hlist}
\begin{Hlist}[resume]
\item \label{hyp:regular-convergence}
The solutions $\{ g_{\e} \}_{\e \leq \e_\ast}$ 
converge as $\e$ tends to zero 
to a weak solution $g$ of the kinetic isothermal Euler system \eqref{eq:KE-iso}, in the sense that
\be
\lim_{\e \to 0} \sup_{t \in [0, T_\ast]} W_1 \left (g_\e(t), g(t) \right ) = 0 .
\ee
\end{Hlist}

Consider rough perturbations $\{ f_{0,\e} \}_{0<\e \leq \e_\ast}$ of the initial data satisfying the following hypotheses.
\begin{Hlist}[resume]
\item  \label{hyp:moments}
$\{ f_{0,\e} \}_{\e \leq \e_\ast}$ are either:
\begin{itemize}
\item In the case $d=1$, probability measures with finite first moment in velocity: \\
 for each $\e \in (0,\e_\ast]$, $\int_{\bt \times \br} |v| \di f_{0,\e} ( x, v) < +\infty$; or
\item In the case $d=2$ or $3$, $L^\infty$ probability density functions such that, for some $k_0 > d$,
\be
 \sup_{0 < \e \leq \e_\ast} \| f_{0,\e} (1 + |v|)^{k_0}  \|_{L^1 \cap L^\infty} \leq C_0 .
\ee
\end{itemize}
\end{Hlist}
\end{hyp}

\begin{theorem} \label{thm:main}
Suppose that Assumption~\ref{hyp:main} holds. Then there exists a constant $C_\ast > 0$ (depending on $C_0$, $T_\ast$, $j_0$ and $k_0$) such that, if
\be \label{hyp:rate} 
W_1(f_{0, \e}, g_{0,\e}) \leq \exp{(-C_\ast \e^{-\zeta})}, \qquad 
\zeta = \begin{cases}
 2 & d=1 \\
11 & d=2 \\
62 & d=3, k_0 \geq 13/4 \\
14 +  \frac{12}{k_0 - 3}  & d=3, k_0 < 13/4,
\end{cases}
\ee
then
\be
\lim_{\e \to 0} \sup_{t \leq T_\ast} W_1(f_\e(t), g(t)) = 0,
\ee
where, for each $\e \in (0, \e_\ast)$, $f_\e$ is ($d=1$) any weak solution or ($d=2,3$) the unique global bounded density solution of the $(VPME)_\e$ system \eqref{eq:vpme} with initial datum $f_{0,\e}$.

\end{theorem}

\begin{remark}
Since both the KIsE \eqref{eq:KE-iso} and ionic Vlasov--Poisson \eqref{eq:vpme} systems conserve total mass, under our hypotheses both $f_\e(t)$ and $g(t)$ are probability densities for all $0 \leq t \leq T_\ast$, and the Monge-Kantorovich-Wasserstein distance $W_1(f_\e(t), g(t))$ is therefore meaningful.
\end{remark}

\begin{remark} \label{rmk:Improvements}
The main improvements achieved in Theorem~\ref{thm:main} compared to the most recent results on this problem can be summarized as follows:
\begin{enumerate}[label=(\roman*)]
\item The most significant enhancement is related to assumption \eqref{hyp:rate} concerning the size of the perturbation. We are able to replace the previous requirement of a {\bf quadruple-exponential} smallness condition \eqref{eq:4exp} with an almost optimal condition involving a {\bf \emph{single} exponential} (see Remark~\ref{rmk:almost opt} below).
\item \label{item:support} In the previous work \cite{GPI-MFQN}, we required that $f_{0,\e}$ have uniformly bounded energy $\mc{E}_\e[f_\e]$ (defined in Equation~\ref{def:Energy} below) and $L^\infty$ norm, as well as having compact support in velocity, with a bound on the rate of growth as $\e$ tends to zero: for a certain function $R(\e)$,
\be \label{hyp:support-growth}
f_{0,\e}(x,v) = 0 \qquad |v| > R(\e) .
\ee
In this work, these requirements have been replaced with assumption \ref{hyp:moments}, which is a uniform-in-$\e$ version of the minimal assumptions currently known for the well-posedness of the VPME system \cite{GPI-WP}. Notably, the data no longer need to have compact support. 
Furthermore, assumption \ref{hyp:moments} implies that the energy $\mc{E}_{\e}[f_\e]$ is uniformly bounded, eliminating the need for a separate assumption. It is also possible to formulate a statement involving a condition on the support similar to \eqref{hyp:support-growth}, while retaining the single exponential structure in \eqref{hyp:rate}, although we omit it here.
\end{enumerate}
\end{remark}

\begin{remark}
In fact the uniform bound \eqref{hyp:density-uniform} can be weakened to allow the $L^1_t L^\infty_{x}$ norm of $\rho[g_\e]$ to grow at a controlled rate as $\e$ tends to zero: if we replace assumption \eqref{hyp:density-uniform} with
\be
\sup_{\e \in (0, \e_\ast]} \;  \e^{\zeta/2 - 1} \int_0^{T_\ast} \left \| \rho[ g_\e ] \right \|_{L^\infty(\bt^d)} \leq C_0 ,
\ee
then the same conclusion holds. 
In the examples we will present, the solution class of the `regular' solutions $g_\e$ will be relatively strong and the uniform bound \eqref{hyp:density-uniform} thus satisfied.
\end{remark}

Building upon this general result, we now discuss two classes of assumptions on the regular part of the initial data $g_{0,\e}$ under which hypothesis \ref{hyp:regular-convergence} is known to hold.

\subsubsection{Spatially Analytic Case}

As was shown by Grenier \cite{Grenier96} the quasineutral limit is known to hold when the initial data are uniformly analytic with respect to the spatial variable $x$. 
In this case no further structural assumption or regularity in the $v$ variable is required. Theorem~\ref{thm:main} shows that this version of the quasineutral limit is stable with respect to rough perturbations that vanish exponentially quickly in $W_1$ as $\e$ tends to zero.

In order to state this corollary, we first recall the definition of the following analytic norm for functions defined on $\bt^d$: for $\delta > 1$, let
\be
\| g \|_{B_\delta} : = \sum_{k \in \bb{Z}^d} | \cF_x g (k)| \delta^{|k|} ,
\ee
where $\cF_x g (k)$ denotes the Fourier coefficient of $g$ (with respect to the spatial variable $x \in \bt^d$) with index $k \in \bb{Z}^d$.

\begin{cor}[Analytic setting] \label{cor:analytic}

Let $\{ g_{0,\e}\}_{\e \leq 1}$ satisfy, for
$k_0 > d$, $C_0 >0$, $\delta > 1$ and sufficiently small $\kappa_0 > 0$,
\begin{Hlistprime}
\item \label{hyp:analytic}
\begin{align}
& \sup_{\e \leq 1} \sup_{v \in \br^d} (1 + |v|^{k_0}) \| g_{0,\e}(\cdot, v) \|_{B_\delta} \leq C_0 \\
& \sup_{\e \leq 1} \left \| \int_{\br^d} g_{0,\e}(\cdot, v) \di v - 1 \right \|_{B_{\delta}} \leq \kappa_0 ,
\end{align}
\item \label{hyp:analytic-convergence} $g_{0,\e}$ converges to a limit $g_0$ in the sense of distributions as $\e$ tends to zero.
\end{Hlistprime}

Then there exists $T_\ast > 0$, $C > 0$, and a solution $g$ of the KIsE system \eqref{eq:KE-iso} with initial datum $g_0$ such that, for all 
$\{ f_{0,\e} \}_{\e \leq 1}$ measures (case $d=1$) or non-negative $L^\infty$ functions (case $d=2,3$) satisfying \ref{hyp:moments} and \eqref{hyp:rate}, then
\be
\lim_{\e \to 0} \sup_{t \leq T_\ast} W_1(f_\e(t), g(t)) = 0 .
\ee
\end{cor}

\subsubsection{Penrose Stable Case} \label{sec:PenroseStable}

We have recalled already that counterexamples to the quasineutral limit exist in arbitrarily high Sobolev regularity \cite{HKH}. 
This is due to unstable modes inherent to the underlying physics. 
These modes may be excluded by imposing a stability criterion in the style of Penrose \cite{Penrose} (see also the discussions in \cite{MouhotVillani, HKR}) and one then expects that quasineutrality should be valid. This was indeed proved rigorously by Han-Kwan and Rousset \cite{HKR}, who showed that the quasineutral limit is valid in sufficiently high Sobolev regularity for the ion model with linearized Poisson coupling, provided that a suitable Penrose-type stability criterion holds (see Equation~\ref{hyp:Penrose} below). 
As explained in the introduction of \cite{HKR}, their techniques can be extended to apply to the case with nonlinear coupling.
Our new result Theorem~\ref{thm:main} shows that this breakthrough high regularity result is stable with respect to small rough perturbations.

The statement of this corollary uses the velocity-weighted Sobolev spaces $\mc{H}^k_r$, where $k \in \mathbb{N}$ and $0 \leq r \in \br$, whose norm is defined by
\be
\| g \|_{\mc{H}^k_r}^2 : = \sum_{|\alpha| \leq k} \int_{\bt^d \times \br^d} (1 + |v|^2)^r |\partial^\alpha_{x,v} g|^2 \di x \di v ,
\ee
as well as the usual Sobolev space $H^k_x$ for functions of $x$ only.

\begin{cor}[Penrose-stable setting] \label{cor:Penrose}
Let $1 \leq d \leq 3$. 
Let $\{ g_{0,\e}\}_{\e \leq 1}$ satisfy:
\begin{Hlist2prime}
\item \label{hyp:Penrose-regularity} Uniform Sobolev-type bounds and $L^2$ convergence: for $m \in \mathbb{N}$ such that
$2m > 4 + d/2 + \lfloor d/2 \rfloor$,
$r_0 \in \br$ such that
$r_0 > 2 + d/2$, and sufficiently small $\kappa_0 > 0$,
\begin{align}
& \sup_{\e \leq 1}  \| g_{0,\e} \|_{\mc{H}^{2m}_{r_0}} \leq C_0 \\
& \sup_{\e \leq 1} \left \| \int_{\br^d} g_{0,\e}(\cdot, v) \di v - 1 \right \|_{H^{2m}_x} \leq \kappa_0 , \\
& \lim_{\e \to 0} \| g_{0, \e} - g_0 \|_{L^2} = 0 ;
\end{align}
\item  \label{hyp:Penrose} The uniform Penrose-type criterion holds:
\be
\inf_{x \in \bt^d, \e \in (0, 1], \gamma > 0, \tau \in \br, \xi \in \br^d \setminus \{ 0 \} } \left | 1 - \int_0^\infty e^{-(\gamma + i \tau)s} \frac{i \xi}{1 + |\xi|^2} \cF_v \nabla_v g_{0,\e} (x, s \xi) \di s \right | > 0 ,
\ee
where $\cF_v$ denotes the Fourier transform with respect to the velocity variable $v$ only.
\end{Hlist2prime}
Then there exists $T_\ast > 0$ and $C > 0$ such that, for all 
$\{ f_{0,\e} \}_{\e \leq 1}$ measures (case $d=1$) or non-negative $L^\infty$ functions (case $d=2,3$) satisfying \ref{hyp:moments} and \eqref{hyp:rate}, then
\be
\lim_{\e \to 0} \sup_{t \leq T_\ast} W_1(f_\e(t), g(t)) = 0 ,
\ee
where $g \in C \left ( [0, T_\ast] ; \mc{H}^{2m -1 }_{r_0} \right )$ is a solution of the KIsE system \eqref{eq:KE-iso} with initial datum $g_0$ (in fact the unique such solution in the class of $C \left ( [0, T_\ast] ; \mc{H}^{2m -1 }_{r_0} \right )$ solutions with $\rho_g \in L^2 \left ( [0, T_\ast] ; H^{2m}_x \right )$; see \cite{HKR}).

\end{cor}

\begin{remark}[Comparison with the electron case]
Theorem~\ref{thm:main} brings the theory for the ion quasineutral limit into line with the electron case, where the best available result also requires an exponential condition on the smallness of the perturbation \cite{IHK1, Iac22}.
\end{remark} 

\begin{remark}[Sharpness of the result]
\label{rmk:almost opt}
The single exponential condition \eqref{hyp:rate} is `almost optimal' since no polynomial rate $\e^N$ is admissible for \emph{any} $N > 0$.
This is due to the existence of exponentially growing modes for the (full) linearization of the system around a kinetically unstable profile, see \cite{HKH, IHK1, Iac22} for further discussion.

The best possible value of the dimension dependent exponent $\zeta(d)$ is for the moment unclear. In dimension $d=1$, the corresponding electron result achieves the exponent $1$, whereas for the ionic model our argument presently gives the exponent $2$. The instability mechanism discussed above rules out any polynomial rate in $\e$, but does not seem to determine the optimal exponential scale in the ionic case. In dimensions $d=2,3$ this question is closely linked to the problem of obtaining optimal growth estimates for the spatial density $\|\rho_{f_\e}\|_{L^\infty(\bt^d)}$ with respect to $t$ and $\e$, which is not yet fully understood in the periodic setting, for either electrons or ions.
\end{remark}

\begin{remark}[Improvements to the well-posedness theory]
As a consequence of the proof of Theorem~\ref{thm:main}, we are in fact able to improve the assumptions for the well-posedness result of \cite{GPI-WP} for the VPME system \eqref{eq:vpme}. 

More precisely, \cite[Theorem 2.1]{GPI-WP} states that the system \eqref{eq:vpme} has a unique global solution with spatial density bounded in $L^\infty(\bt^d)$, locally uniformly in time, for any $\e > 0$ and any initial datum satisfying
$(1 + |v|^{k_0}) f_{0,\e} \in L^\infty(\bt^d \times \br^d)$ and $( 1 + |v|^{m_0}) f_{0,\e} \in L^1(\bt^d \times \br^d)$ for $k_0 > d$ and $m_0 > d(d-1)$.

As a corollary of the techniques of Section~\ref{sec:density} we can relax the second assumption to require only $m_0 > d$.
Thus, in particular, for $f_{0,\e}$ satisfying  hypothesis \ref{hyp:moments} for $d=2,3$, the unique global bounded density solutions $f_\e$ referred to in the statement of Theorem~\ref{thm:main} exist.
Actually, as the reader can readily check, our proof applies also to the case of the electron Vlasov--Poisson system, so our result also gives the well-posedness of the classical system for initial data that satisfy \ref{hyp:moments} for $d=2,3$.
\end{remark}
\begin{paragraph}{The long time behaviour of plasmas and the quasineutral limit.}
The quasineutral limit can be thought of as a form of long-time limit: as explained in \cite{HanKwanHDR} (see also \cite{IHK1}), by suitable scalings one sees a connection between the quasineutral limit and the study of the long-time behaviour.
A particularly well-known phenomenon in this context is \emph{Landau damping}, see for example \cite{MouhotVillani, BedrossianMasmoudiMouhot, BMM2016, GrenierNguyenRodnianski, HKNguyenRousset2021, HKNguyenRousset-WS, BMM2022, GNR2022, HuangNguyenXu05, HuangNguyenXu06, Gagnebin, GagnebinIacobelli, ChatLukNguyen, nguyen2023landau, Wei2025, BCGIR, IRW}. 
\end{paragraph}

\bigskip

\begin{paragraph}{Structure of the Proof.}
The strategy of the proof differs according to the dimension $d=1,2,3$ under consideration. \\

\noindent \textbf{Dimension $d=1$:} In dimension one, our main result applies even when $f_\e$ is a \emph{measure} solution of \eqref{eq:vpme}.
The one-dimensional ionic Vlasov--Poisson system benefits from \emph{weak-strong} stability estimates in $W_1$, in which only one of the solutions need have bounded spatial density $\rho_{g_\e} \in L^\infty(\bt^1)$; the other solution may then be a measure \cite{IHK1}. However, in \cite{IHK1} the stability and regularity estimates for the electric field diverge exponentially quickly in $\e$ in the quasineutral limit. 

The key step for $d=1$ is to prove new estimates for the ionic electric field with a polynomial rather than exponential dependence on $\e$. The special feature of dimension one is that our estimates must apply where at least one of the spatial densities $\rho_{f_\e}$ is only a measure. This is carried out in Section~\ref{sec:electric_1}. Then, in Section~\ref{sec:stability_1}, we apply these estimates to obtain new quasineutral weak-strong $W_1$ stability estimates for the ionic Vlasov--Poisson system \eqref{eq:vpme}. \\

\noindent \textbf{Dimension $d\geq2$:} In dimension two and higher, Wasserstein stability estimates for (both electron and ion) Vlasov--Poisson systems are at present only known to hold when both solutions have bounded density $\rho_{f_\e} \in L^\infty$. We therefore work with perturbed solutions $f_\e$ in this class (in particular, unlike in dimension $d=1$, measures are excluded). 

We restrict the present higher-dimensional analysis to the cases $d=2,3$. This is consistent with the range in which the classical Vlasov--Poisson theory is best understood for sufficiently regular solutions; see for instance  \cite{AmbrosioColomboFigalli}. In particular, this is the range in which the well-posedness and quantitative estimates used in our argument are available in the form required here. From the physical point of view, the three-dimensional case is the most relevant one, while the two-dimensional case remains a natural intermediate setting in which many of the same difficulties already appear.
This framework traces back to the work of Loeper \cite{Loeper} on $W_2$ stability for electron Vlasov--Poisson; for ions, analogous results were shown in \cite{GPI-WP, GPI-WP-R3}.
Our goal in this work is to prove $W_2$ stability estimates for ionic Vlasov--Poisson with optimised \emph{quantitative} dependence on $\e$.
The choice of $W_2$ is for technical convenience, since under our moment assumptions \ref{hyp:moments} all Wasserstein distances on the initial data are equivalent.

Our optimisation has three main ingredients:
\begin{enumerate}[label=(\roman*)]
\item A new `Loeper-type' \emph{stability estimate for electric fields} produced by the Poisson--Boltzmann equation (Proposition~\ref{prop:regU}) in terms of the $W_2$ distance between the inducing ($L^\infty$) densities. The novelty is that we achieve the same dependence on $\e$ as is known for electron Vlasov--Poisson, given the same bound on $\| \rho_{f_\e} \|_{L^\infty}$. In previous results of this type \cite{GPI-MFQN}, the estimates diverged exponentially in $\e$. These estimates for the Poisson--Boltzmann equation are proved in Section~\ref{sec:electric_23}.
\item Estimation of $W_2$ by means of \emph{nonlinearly anisotropic functionals:} This approach, developed by the second author in \cite{Iac22} in the context of electron Vlasov--Poisson, improves the quantitative dependence of the resulting estimate by an exponential factor. By applying the estimates obtained in Section~\ref{sec:electric_23} we are able to implement the techniques of \cite{Iac22} to prove analogous estimates for bounded density solutions of the ionic Vlasov--Poisson system. This is done in Section~\ref{sec:stability_23}.
\item \emph{Density bounds:} Our $W_2$ estimate is quantified in terms of $L^\infty$ bounds on the spatial densities of \emph{both} solutions. To complete the quantitative analysis, we must therefore obtain bounds on $\| \rho_{f_\e} \|_{L^\infty}$ for the perturbed solutions $f_\e$. This is the subject of Section~\ref{sec:density}.

To do this, we perform an analysis of the growth rate of trajectories of the characteristic flow. This information simultaneously both can be used to control and is influenced by the density $\rho_{f_\e}$.
In previous work \cite{GPI-MFQN} control of $\| \rho_{f_\e} \|_{L^\infty}$ in terms of the characteristic trajectories was achieved by imposing an assumption of compact support on the initial data (Remark~\ref{rmk:Improvements}\ref{item:support}). In the present article, we demonstrate how to remove this assumption by instead formulating the feedback loop on the characteristic trajectories through the behaviour of \emph{velocity moments} (see e.g. \cite{Pallard, ChenChen} for electrons or \cite{GPI-WP} for ions, in both cases without quasineutral scaling). However, an approach based on controlling the growth of a compact support would equally be possible, although we will not write it here.

By either method, the argument is significantly more involved in dimension $d=3$ and requires a new approach to the analysis of the electric field. We discuss this in more detail below.
\end{enumerate}

\noindent \textbf{Density Bounds in Dimension $d=3$:} In dimension three, bounds for the electric field that treat the spatial density $\rho_{f_\e}$ simply as a function of $x$ are not sufficient to be able to close global-in-time growth estimates for the characteristic trajectories. Instead, successful trajectorial arguments (going back to \cite{Pfaffelmoser, Schaeffer, Batt-Rein} for electron Vlasov--Poisson) make use of the velocity dependence of $f$ and exploit the second-order structure of the characteristic flow.
These techniques, however, rely on the representation of the electron electric field as a convolution between $f$ and the Coulomb kernel, which is not available in the ionic case.

Previous works on the three-dimensional ionic Vlasov--Poisson system overcame this by analysing the electric field through a certain decomposition, first suggested for this system in the one-dimensional setting \cite{IHK1}: the full electrostatic potential $U$ is written as a sum $U = \bar U + \widehat U$, where
\be \label{def:Usplit-intro}
- \e^2 \Delta \bar U = \rho - 1, \; \int_{\bt^3} \bar U \di x = 0, \qquad - \e^2 \Delta \widehat U = 1 - e^U .
\ee
This representation combines the availability of the full suite of `electronic' techniques for analysis of $\bar U$ with higher regularity for the `remainder' $\widehat U$ (as proved in \cite{GPI-WP}), thereby enabling estimates for the characteristic trajectories to be proven.

However, while $\widehat U$ is `smoother' than $\bar U$ in general, it is not so clear that it is `small': quantitative estimates of the gain of regularity for $\widehat U$ diverge as $\e$ tends to zero. Previously, all available estimates for $\widehat U$ in suitable norms diverged exponentially fast in $\e$ \cite{GPI-MFQN}.

In this article, we show that a suitable gain of regularity for $\widehat U$ can in fact be achieved with only a \emph{polynomial} loss in $\e$.
We prove a new estimate for the nonlinearity $e^U$, showing a gain of integrability (Lemma~\ref{lem:eU-Lp-eps}).
Using this estimate, we are able to prove growth estimates for characteristics of the ionic Vlasov--Poisson system with a polynomial loss in $\e$ (Section~\ref{sec:moments_3}). It is this polynomial dependence that allows us to conclude that our main result holds with the \emph{single} exponential rate \eqref{hyp:rate} also in dimension $3$. In comparison to \cite{ChenChen, GPI-WP}, we refine the moment propagation argument here with new techniques designed for the ionic case, which allow us to take the exponent $\zeta(3)$ smaller than would otherwise be possible.
\end{paragraph}

\begin{paragraph}{Outline of the Paper.}
The paper is structured as follows: in Section~\ref{sec:prel} we collect a series of preliminary results. In Section~\ref{sec:ElectricField}, we establish novel regularity estimates for the electric field and its stability concerning the spatial density. A crucial improvement compared to prior findings is the derivation of constants that exhibit only polynomial degeneracy in $\e$.
Section~\ref{sec:Stability} combines the outcomes from Section~\ref{sec:ElectricField} with the employment of kinetic-Wasserstein distances, recently introduced by the second author. This combination leads to precise stability estimates for solutions of VPME with bounded density. To apply this result effectively in our context, Section~\ref{sec:density} presents new $L^\infty$ bounds on the spatial density $\rho_f$ for a solution $f$ of the VPME system \eqref{eq:vpme}. Finally, in Section~\ref{sec:proof}, we provide the proof of our main theorem, Theorem~\ref{thm:main}.

\end{paragraph}

\section{Preliminaries}\label{sec:prel}

\subsection{Representation of the torus $\bt^d$} \label{sec:torus}

Throughout this work, $\bt^d$ denotes the unit flat torus in $d$ dimensions.
For the purposes of defining integrals over the torus, we identify points in $\bt^d$ with points in the unit box $[ - 1/2, 1/2 )^d$.
This is equipped with the distance $| \cdot |_{\bt^d}$ defined by
\be \label{def:TorusDistance}
| x |_{\bt^d} : = \inf_{\alpha \in \Z^d} |x + \alpha| .
\ee
In some arguments, it will necessary to keep track of the number of times a path $z(t) : I \to \bt^d$ wraps around the torus.
In this context, we will consider a lifted version of $z(t)$ thought of as a path on $\br^d \times \br^d$. 
In such cases, in order to evaluate quantities of the form $f(z(t))$, we identify functions on $\br_+ \times [ - 1/2, 1/2 )^d \times \br^d$ with their spatially periodic extensions on $\br_+ \times \br^d \times \br^d$ in the natural way:
\be
f(t,x,v) = f(t,x+\alpha,v), \qquad \text{where} \; \alpha \in \bb{Z}^d, \; x + \alpha \in \left [ - \frac{1}{2}, \frac{1}{2} \right )^d.
\ee

\subsection{Monge--Kantorovich--Wasserstein Distances}

We recall the definition of the Wasserstein distances $W_p$ for measures on the phase space.

\begin{defi} \label{def:Wp}
Given two probability measures $\mu,\nu$ on $\bt^d \times \br^d$,
for any $p \in [1, +\infty)$, the Wasserstein distance of order $p$, denoted $W_p$, is defined by
\be \label{eqdef:Wass}
W_p^p(\mu, \nu) = \inf_{\pi \in \Pi(\mu,\nu)} \int_{(\bt^d \times \br^d)^2} \bigl(|x-y|^p_{\bt^d}+|v-w|^p\bigr) \di \pi(x,v,y,w),
\ee
where $\pi \in \mc{P}((\bt^d \times \br^d)^2)$ belongs to the set of \emph{couplings} $ \Pi(\mu,\nu)$: namely, for any Borel subset $\mc{A} \subset \bt^d \times \br^d$,
\be
\pi(\mc{A} \times (\bt^d \times \br^d)) = \mu(\mc{A}) \qquad \pi((\bt^d \times \br^d) \times \mc{A}) = \nu(\mc{A}).
\ee
We note that  $W_p(\mu, \nu)<\infty$ for $\mu, \nu \in \mc{P}_p$, where $\mc{P}_p$ denotes the set of probability measures $\gamma$ for which 
\be
\int_{\bt^d \times \br^d} |v|^p \di \gamma(x,v) < + \infty.
\ee
\end{defi}

We will occasionally use the same symbol for a probability measure and, when it is absolutely continuous with respect to Lebesgue measure, for its density. Accordingly, expressions of the form $\int \varphi \, dg$ always denote integration with respect to the measure $g$.
Our proof of Theorem~\ref{thm:main} relies on a new stability estimate for solutions of the VPME system \eqref{eq:vpme} in $W_2$ (Proposition~\ref{prop:stability}). To prove this estimate, we make use of a new technique proposed in \cite{Iac22}, in which we consider
a quantity related to the Wasserstein distance with a nonlinearly defined kinetic structure (see Section~\ref{sec:Stability}).

In order to obtain our final result in $W_1$, we will need a couple of simple estimates between different powers of the Wasserstein distance. We consider only the cases $p=1,2$, since this is what is relevant for us. 
\begin{lemma}
\label{lemma:Wp}
Let $\mu,\nu$ be two probability densities on $\bt^d \times \br^d$ such that
$$
\int_{\bt^d \times \br^d} |v|^k \di \mu(x,v) \leq C_k,\qquad \int_{\bt^d \times \br^d} |v|^k \di \nu(x,v)\leq C_k
$$
for some $C_k<\infty$ and $k>2$.
Then
$$
W_1(\mu,\nu)\leq \sqrt{2}W_2(\mu,\nu),\qquad W_2(\mu,\nu) \leq 3(1+2C_k)^{\frac{1}{k-1}}W_1(\mu,\nu)^{\frac{k-2}{k-1}}.
$$ 
\end{lemma}

\subsection{Density Estimates Using Moments}

We recall the following well-known `interpolation' estimate (see for example \cite{Lions-Perthame}), which states that bounds on the velocity moments of $f$ imply $L^p$ bounds on the spatial density $\rho$.

\begin{lemma} \label{lem:moments}
Let $d \geq 1$.
Let $f \in L^\infty(\bt^d \times \br^d)$ satisfy, for some $k > 0$,
\be
M_k : = \int_{\bt^d \times \br^d} |v|^k | f(x,v) | \di x \di v < + \infty .
\ee
Then the spatial density
\be
\rho(x) : = \int_{\br^d} f(x,v) \di v
\ee
belongs to $L^{1 + k/d}(\bt^d)$ with the estimate
\be \label{est:rhoLp-moment}
\| \rho \|_{L^{1 + k/d} (\bt^d)} \leq C_{ k,d} \| f \|^{\frac{k}{d+k}}_{L^\infty (\bt^d \times \br^d) } M_k^{\frac{d}{d+k}} .
\ee
\end{lemma}

\subsection{Energy Functional}

The energy of the VPME system \eqref{eq:vpme} is given by the functional
\be \label{def:Energy}
\mc{E}_\e[f_\e] := \frac{1}{2}\int_{\bt^d \times \br^d} |v|^2 f_\e \di x \di v + \frac{\e^2}{2} \int_{\bt^d} |\nabla U_\e|^2 \di x +  \int U_\e e^{U_\e} \di x .
\ee
This quantity is conserved by all sufficiently regular solutions of \eqref{eq:vpme}, and in particular by the strong solutions constructed in \cite{GPI-WP} that we will use in the current work.

Under hypothesis~\ref{hyp:moments}, the energy of the initial data $f_{0,\e}$ is bounded uniformly in $\e$ -- we sketch the argument below in Lemma~\ref{lem:InitialEnergy}. Therefore, the energy of solutions to the VPME system \eqref{eq:vpme} starting from these data is bounded both uniformly in $\e$ and uniformly for all time. Moreover, under hypothesis~\ref{hyp:analytic}, the functions $g_{0,\e}$ also satisfying \ref{hyp:moments}, and thus the energy of solutions with initial data $g_{0,\e}$ is bounded uniformly in both $\e$ and time.

\begin{lemma} \label{lem:InitialEnergy}
Let $f_{0,\e}$ satisfy \ref{hyp:moments}. 
Then there exists a constant $C_1 > 0$ depending on $C_0$ only such that
\be
\mc{E}_\e [f_{0,\e}] \leq C_1.
\ee
\end{lemma}
\begin{proof}
Hypothesis~\ref{hyp:moments} implies that
\be
\| (1 + |v|^2) f_{0,\e} \|_{L^1} + \| f_{0,\e} \|_{L^\infty} \leq C_0,
\ee
so that the kinetic energy term is uniformly bounded. 

Moreover, by Lemma~\ref{lem:moments}, for some constant $C_0 ' > 0$ depending only on $C_0$,
\be
\| \rho_{0,\e} \|_{L^{(d+2)/d}} \leq C_0 ' .
\ee
Since $\e^2 \Delta_x U_\e = e^{U_\e} - \rho_{0,\e}$, the remaining terms satisfy
\begin{align}
 \frac{\e^2}{2} \int_{\bt^d } |\nabla U_\e|^2 \di x +  \int_{\bt^d } U_\e e^{U_\e} \di x & = \int_{\bt^d } U_\e \rho_{0,\e} \di x
  \leq  \int_{\bt^d } (U_\e)_+ \rho_{0,\e} \di x .
\end{align}
By H\"older's inequality,
\begin{align}
\int_{\bt^d } (U_\e)_+ \rho_{0,\e} \di x & \leq \| \rho_{0,\e} \|_{L^{(d+2)/d}} \| (U_\e)_+ \|_{L^{(d+2)/2}} \\
& \leq \| \rho_{0,\e} \|_{L^{(d+2)/d}} \| (U_\e)_+^{d/2} \|_{L^{(d+2)/d}}^{2/d}.
\end{align}
Finally, note that there exists a constant $c_d > 0$ such that $y^{d/2} \leq c_d e^y$ for all $y \geq 0$. 
Hence, by Lemma~\ref{lem:eU-Lp} below,
\begin{align}
\int_{\bt^d } (U_\e)_+ \rho_{0,\e} \di x & \leq c_d^{2/d} \| \rho_{0,\e} \|_{L^{(d+2)/d}} \| e^{U_\e} \|_{L^{(d+2)/d}}^{2/d} \\
& \leq c_d^{2/d} \| \rho_{0,\e} \|_{L^{(d+2)/d}}^{(d+2)/d} .
\end{align}
Thus
$\mc{E}_\e [f_{0,\e}] \leq C_1$,
where $C_1$ depends on $C_0$ only.
\end{proof}

We recall the following consequence.
Since $\mc{E}[f_\e(t)]$ is uniformly bounded for all $t$ and $\e$, $f_\e(t)$ has uniformly bounded second velocity moment $M_2(t)$. Since the transport equation also conserves the $L^\infty$ norm of $f_\e$, the following uniform $L^p$-type bound on $\rho[f_\e]$ can be deduced.

\begin{lemma}
Let $f_\e$ be a solution of the VPME system~\eqref{eq:vpme} with initial datum $f_{0,\e}$ that satisfies \ref{hyp:moments} ($f_\e$ is then the global unique solution with bounded density).
Then there exists a constant $C_1 > 0$ depending on $C_0$ only such that
\be \label{est:53-uniform}
\| \rho[f_\e(t)] \|_{L^{d+2/d}} \leq C_1  \qquad \text{for all} \; t \geq 0 .
\ee
\end{lemma}

\section{Estimates for the Electric Field} \label{sec:ElectricField}

This section focuses on obtaining estimates for solutions of the \emph{Poisson--Boltzmann} equation
\be
\e^2 \Delta U = \rho - e^U ,
\ee
which will be used to prove Wasserstein stability estimates for the ionic Vlasov--Poisson system \eqref{eq:vpme}. We are interested in the regularity of the electric field $- \nabla U$ and its stability with respect to variations in the inducing density $\rho$. In both cases, our central aim is to mitigate as much as possible the speed of divergence of the constants in our estimates as $\e$ tends to zero.

The expected regularity of $\rho$ for a solution of \eqref{eq:vpme} and the overall strategy of proof are different for different dimensions $d=1,2,3$. We divide this section accordingly.

\subsection{Case $d=1$} \label{sec:electric_1}

In dimension $d=1$, we work in a \emph{weak-strong} setting, in which we expect one solution to have bounded density $\rho_1 \in L^\infty(\bt^1)$, while the other may be merely a probability measure $\rho_2 \in \mc{P}(\bt^1)$. The weak-strong stability estimates for one-dimensional electron Vlasov--Poisson in \cite{Hauray} make use of an explicit convolution representation of the electric field, which in the ion case is not available due to the nonlinearity. We therefore analyse the Poisson--Boltzmann equation using a decomposition method that was already introduced in \cite{IHK1}: we introduce potentials $\bar U$ and $\widehat U$ such that
\be
- \e^2 \bar U'' = \rho - 1, \; \; \int_{\bt} \bar U \di x = 0, \qquad \qquad \e^{2} \widehat U '' = 1 - e^{\bar U + \widehat U}.
\ee
Then $U : = \bar U + \widehat U$ satisfies the Poisson--Boltzmann equation
\be
- \e^2 U'' = \rho - e^U .
\ee

With this decomposition, the `singular' part $\bar U$ then satisfies exactly the \emph{linear} Poisson equation that would appear in the electron Vlasov--Poisson system and is therefore amenable to the corresponding techniques. Meanwhile $\widehat U$ is more regular \cite[Lemma 2.2]{IHK1} -- however, this improved regularity comes at a cost in $\e$. In this section, we seek stability estimates for the `regular' part of the electric field $- \widehat U'$ with respect to the inducing density $\rho$. In previous work \cite{IHK1} such estimates were obtained with an exponential loss in $\e$; here we are able to replace this by an algebraic loss.

\begin{prop} \label{prop:regU-1d}
Let $d=1$.
\begin{enumerate}[label=(\roman*)]
\item Let $\rho \in \mc{P}(\bt^1)$ be a probability measure.
There exist a unique $\bar U, \widehat U \in W^{1,2}$ such that
\be
- \e^2 \bar U'' = \rho - 1, \; \int_{\bt} \bar U \di x = 0, \qquad \e^{2} \widehat U '' = 1 - e^{\bar U + \widehat U}.
\ee
Then $U = \bar U + \widehat U$ is a continuous function on $\bt^1$ satisfying
\be
\| e^U \|_{L^1(\bt^1)} = 1, \qquad \| U \|_{L^\infty} \leq \frac{1}{3} \e^{-2}, \qquad \| U \|_{\Lip} \leq \e^{-2}, 
\ee
\item $\widehat U '$ is Lipschitz, with estimate
\be
\|  \widehat U ' \|_{\text{Lip}} \leq \e^{-4} .
\ee

\item Let $\rho_i \in \mc{P}(\bt^1)$, ($i=1,2$) be probability measures. Then
\be
\| \widehat U_1 ' -  \widehat U_2 ' \|_{L^2} \leq \frac{1}{4} \e^{-3} W_1(\rho_1, \rho_2) , 
\ee
\end{enumerate}
\end{prop}

\begin{proof} \emph{Parts (i) and (ii)}:
The existence and uniqueness of $\bar U$ and $\widehat U$ is shown in \cite{IHK1}, where it is also shown that $\bar U$ is continuous and $\widehat U$ is twice continuously differentiable on $\bt^1$. We will now establish the quantitative estimates.

$U = \bar U + \widehat U$ satisfies
\be
- \e^2 U'' = \rho - e^U
\ee
in the sense of distributions on $\bt^1$. Thus, testing the equation with the constant function 1 gives
\be
\| e^U \|_{L^1(\bt^1)} = \int_{\bt^1} e^U \di x = \int_{\bt^1} 1 \di \rho = 1,
\ee
since $e^U > 0$ and $\rho$ is assumed to be a probability measure. 

Moreover, $U$ has the representation
\be
U = \e^{-2} G_1 \ast (\rho - e^U) + \int_{\TT^1} U \di x ,
\ee
where
\be \label{def:G1}
G_1 := \frac{1}{2} \left ( x^2 - |x| + \frac{1}{6} \right ) \qquad x \in (- \frac{1}{2}, \frac{1}{2} ] 
\ee
denotes the Green's function of the Laplacian on the one-dimensional torus $\bt^1$ (here represented using its fundamental domain $(- \frac{1}{2}, \frac{1}{2} ] $). That is,
\be
- G_1 '' = \delta_0 - 1, \qquad \int_{\TT^1} G_1 \di x = 0.
\ee
We note in particular that $G_1$ is Lipschitz and bounded, with
\be
\| G_1 \|_{L^\infty(\TT^1)} = \frac{1}{12}, \qquad \| G_1 \|_{\text{Lip}} \leq \frac{1}{2} .
\ee
Therefore,
\begin{align}
\| U - \int_{\TT^1} U \di x \|_{L^\infty} &\leq \e^{-2} \int_{\bt} G_1 \di (\rho - e^U) \\
& \leq \e^{-2} \| G_1\|_{L^\infty} \| \rho - e^U \|_{\text{TV}} .
\end{align}
Since $\rho$ is a probability measure and $e^U$ a probability density function, we have
\be
\| \rho - e^U \|_{\text{TV}}  \leq \int_{\bt^1} 1 \di \rho + \int_{\bt^1} e^U \di x \leq 2 . 
\ee
Hence
\be
\| U - \int_{\TT^1} U \di x \|_{L^\infty}  \leq \frac{1}{6} \e^{-2} .
\ee
It remains to estimate $\int_{\bt^1} U \di x$. Since the exponential function is convex, by Jensen's inequality
\be
\exp(\int_{\bt^1} U \di x) \leq \int_{\bt^1} e^U \di x = 1,
\ee
and thus $\int_{\bt^1} U \di x \leq 0$. For a lower bound, we observe that, since $U \leq \int_{\bt^1} U \di x + \frac{1}{6} \e^{-2}$,
\be
1 = \int_{\bt^1} e^U \di x \leq \exp \left ( \int_{\bt^1} U \di x + \frac{1}{6} \e^{-2} \right ) .
\ee
Hence $\int_{\bt^1} U \di x \geq \frac{1}{6} \e^{-2}$, and we deduce that $\| U \|_{L^\infty} \leq \frac{1}{3} \e^{-2}$.

For the Lipschitz regularity of $U$, note that
\begin{align}
|U(x) - U(y)| & \leq \e^{-2} \int | G_1(x-z) - G_1(y-z) | \di |\rho - e^U| \\
& \leq 2 \e^{-2}  \| G_1 \|_{\text{Lip}} |x-y| \\
& \leq \e^{-2} |x-y| .
\end{align}

Next we look at the Lipschitz regularity of $\widehat U'$. Note that $\widehat U '$ has the representation
\be
\widehat U ' = \e^{-2} G_1 ' \ast (1 - e^U).
\ee
Hence
\begin{align}
|\widehat U '(x) - \widehat U '(y) | &\leq \e^{-2} \left | \int \left ( e^{U(x-z)} - e^{U(y-z)} \right ) G_1'(z) \di z \right | \\
& \leq \e^{-2} \int_{\bt^1} \int_0^1 |U(x-z) - U(y-z)| e^{\alpha U(x-z) + (1-\alpha) U(y-z)} | G_1'(z)| \di \alpha \di z \\
& \leq \e^{-2} |x-y| \| U \|_{\Lip}\int_0^1  \int_{\bt^1}  e^{\alpha U(x-z) + (1-\alpha) U(y-z)} | G_1'(z)| \di z \di \alpha \\
& \leq \e^{-4} |x-y| \|  G_1' \|_{L^\infty} \int_0^1 \| e^{U(x-\cdot)} \|_{L^1}^\alpha  \| e^{U(y-\cdot)} \|_{L^1}^{1-\alpha} \di \alpha .
\end{align}
Since $\|  G_1' \|_{L^\infty} \leq \frac{1}{2}$ and $ \| e^U \|_{L^1} = 1$, we obtain
\be
\| \widehat U'  \|_{\text{Lip}} \leq \frac{1}{2} \e^{-4} .
\ee
This completes the proof of Parts (i) and (ii).

\noindent \emph{Part (iii):}
The difference between the potentials satisfies the equation
\be
- \e^2 (\widehat U_1 - \widehat U_2 )'' = e^{U_2} - e^{U_1} .
\ee
Testing this with $\widehat U_1 - \widehat U_2$ gives
\be
\e^2 \| \widehat U_1 ' - \widehat U_2 ' \|_{L^2}^2 = - \int_{\bt^1} (\widehat U_1 - \widehat U_2 ) (e^{U_1} - e^{U_2}) \di x .
\ee
Rearranging gives
\be
\e^2 \| \widehat U_1 ' - \widehat U_2 ' \|_{L^2}^2 + \int_{\bt^1} ( U_1 -  U_2 ) (e^{U_1} - e^{U_2}) \di x =  \int_{\bt^1} (\bar U_1 - \bar U_2 ) (e^{U_1} - e^{U_2}) \di x .
\ee
Writing 
\be
e^{U_1} - e^{U_2} = \int_0^1 (U_1 - U_2) e^{\alpha U_1} e^{(1-\alpha)U_2} \di \alpha,
\ee
we find that
\begin{align}
\| e^{U_1} - e^{U_2} \|_{L^1(\bt^1)} & \leq \int_0^1 \int_{\bt^1} (U_1 - U_2) e^{\alpha U_1} e^{(1-\alpha)U_2} \di x \di \alpha \\
& \leq \left ( \int_0^1 \int_{\bt^1} (U_1 - U_2)^2 e^{\alpha U_1} e^{(1-\alpha)U_2} \di x \di \alpha \right )^{1/2} \left ( \int_0^1 \int_{\bt^1} e^{\alpha U_1} e^{(1-\alpha)U_2} \di x \di \alpha \right )^{1/2} \\
& \leq \left ( \int_0^1 \int_{\bt^1} (U_1 - U_2)^2 e^{\alpha U_1} e^{(1-\alpha)U_2} \di x \di \alpha \right )^{1/2} \left ( \int_0^1  \| e^{ U_1} \|_{L^1(\bt^1)}^\alpha \| e^{ U_1} \|_{L^1(\bt^1)}^{1-\alpha} \di \alpha \right )^{1/2} \\
& \leq \left ( \int_0^1 \int_{\bt^1} (U_1 - U_2)^2 e^{\alpha U_1} e^{(1-\alpha)U_2} \di x \di \alpha \right )^{1/2} \\
& \leq \left ( \int_{\bt^1} (U_1 - U_2)  (e^{U_1} - e^{U_2})  \di x  \right )^{1/2} ,
\end{align}
where we have used that $e^{U_i}$ has total integral 1 for $i=1,2$.

Then, applying H\"older inequality,
\begin{align}
\int_{\bt^1} (\bar U_1 - \bar U_2 ) (e^{U_1} - e^{U_2}) \di x & \leq  \| \bar U_1 - \bar U_2 \|_{L^\infty} \| e^{U_1} - e^{U_2} \|_{L^1(\bt^1)} \\
& \leq  \| \bar U_1 - \bar U_2 \|_{L^\infty}  \left ( \int_{\bt^1} (U_1 - U_2)  (e^{U_1} - e^{U_2})  \di x  \right )^{1/2} .
\end{align}
Thus
\be
\e^2 \| \widehat U_1 ' - \widehat U_2 ' \|_{L^2}^2 + \frac{1}{2} \int_{\bt^1} ( U_1 -  U_2 ) (e^{U_1} - e^{U_2}) \di x = \frac{1}{2} \| \bar U_1 - \bar U_2 \|_{L^\infty}^2 ,
\ee
whence
\be
 \| \widehat U_1 ' - \widehat U_2 ' \|_{L^2} \leq \frac{\e^{-1}}{2} \| \bar U_1 - \bar U_2 \|_{L^\infty} , 
\ee

Using the Kantorovich duality characterisation of $W_1$, we have
\begin{align}
\| \bar U_1 - \bar U_2 \|_{L^\infty} & = \e^{-2} \| G_1 \ast (\rho_1 - \rho_2) \|_{L^\infty} \\
& \leq \e^{-2} \| G_1 \|_{\text{Lip}} W_1 (\rho_1, \rho_2) \\
& \leq \frac{\e^{-2}}{2} W_1 (\rho_1, \rho_2).
\end{align}
This completes the proof.

\end{proof}

\subsection{Case $d=2,3$: Wasserstein Stability} \label{sec:electric_23}

In dimension $d=2,3$, we use a `Loeper-type' stability estimate for the electric field. In these estimates, we work under the assumption that \emph{both} solutions have $L^\infty$ spatial density $\rho$. The $L^2$ deviation between the electric fields can then be controlled in terms of the second order Wasserstein distance between the inducing densities. For electron Vlasov--Poisson, such estimates were first obtained by Loeper \cite{Loeper}, and later quantified with respect to quasineutral scaling in \cite{IHK2}. Similar estimates for ionic Vlasov--Poisson were obtained in \cite{GPI-WP, GPI-MFQN}. The key improvement here compared to previous results is that we obtain constants that degenerate \emph{polynomially} in $\e$. In fact, the dependence on $\e$ is identical to that seen in the Vlasov--Poisson system for electrons \cite{IHK2}. 

\begin{prop} \label{prop:regU} Let $d =2,3$.
(i) Let $h \in L^\infty (\bt^d)$. Then there exists a unique  $U \in W^{1,2}(\bt^d)$ satisfying
$$
\e^2\Delta U =e^U-h .
$$
Moreover, $\nabla U$ is a log-Lipschitz function satisfying
\be \label{nonlinear-logLipschitz}
|\nabla U(x) - \nabla U(y)| \leq C \| h \|_{L^\infty} \e^{-2} |x-y| \left ( 1+ (\log|x-y|)_+ \right )
\ee

\noindent (ii) If, for $i=1,2$, $0 \leq h_i \in L^\infty$, with $U_i \in W^{1,2}$ satisfying,
\be
\e^2\Delta U_i =e^{U_i}-h_i,
\ee
and
\be
\int_{\bt^d} h_1 \di x = \int_{\bt^d} h_2 \di x, 
\ee
then
\be \label{nonlinear-Loeper}
\| \nabla U_1 - \nabla U_2 \|_{L^2} \leq \e^{-2} \max_i \| h_i \|_{L^\infty}^{1/2} W_2(h_1, h_2)
\ee

\end{prop}

For the proof of the above proposition, we will require the following integrability estimate for the nonlinearity $e^U$: notice that the constant here is independent of $\e$.

\begin{lemma}\label{lem:eU-Lp} Let $d\geq1$.
Let $h\in L^\infty(\bt^d)$ and let $U \in W^{1,2}(\bt^d)$ be a solution of 
\be
\e^2\Delta U =e^U-h .
\ee
Then, for all $p \in [1,+\infty]$,
\be
\| e^U \|_{L^p} \leq \| h \|_{L^p} .
\ee
\end{lemma}
\begin{proof}
We first consider the case $p<\infty$.
The proof follows from the following a priori estimate: formally testing the equation with the function $e^{(p-1)U}$ and integrating by parts gives
\be
0 \leq \e^2 (p-1) \int_{\bt^d} e^{(p-1)U} |\nabla U|^2 \di x =  \int_{\bt^d} e^{(p-1)U} h  \di x -  \int_{\bt^d} e^{p U}  \di x .
\ee
By rearranging terms and applying H\"older's inequality, we obtain
\be
\| e^U \|_{L^p}^p \leq \| e^U \|_{L^p}^{p-1} \| h \|_{L^p},
\ee
and thus
\be
\| e^U \|_{L^p} \leq \| h \|_{L^p}\qquad \text{for all }p< + \infty .
\ee
This argument can be made rigorous using a truncation procedure.

Letting $p\to + \infty$ in the bound above, we conclude the validity of the lemma also in the case $p=+\infty$.
\end{proof}

The second ingredient is the following stability estimate for the Poisson--Boltzmann equation.

\begin{lemma} \label{lem:U-stab}
Let $d=2,3$.
For $i=1,2$, let $h_i \in L^\infty(\bt^d)$ and let $U_i$ satisfy 
\be
\e^2\Delta U_i =e^{U_i}-h_i .
\ee
Then
\be
\| \nabla U_1 - \nabla U_2 \|_{L^2} \leq \e^{-2} \| 
\nabla \Delta^{-1} (h_1 - h_2) \|_{L^2}
\ee
\end{lemma}
\begin{proof}
Subtracting the equations for $U_1$ and $U_2$ gives
\be
\e^2\Delta (U_1 - U_2) =(e^{U_1}-e^{U_2})-(h_1 - h_2) .
\ee
After testing with $(U_1 - U_2) $ and integrating by parts, we obtain:
\be
\e^2 \int_{\bt^d} |\nabla U_1 - \nabla U_2|^2 \di x = \int_{\bt^d} (h_1 - h_2) (U_1 - U_2) \di x - \int_{\bt^d} (e^{U_1}-e^{U_2}) (U_1 - U_2) \di x .
\ee
Since $(e^x - e^y)(x-y) \geq 0$ for any $x,y \in \br$, we have
\be
\e^2 \int_{\bt^d} |\nabla U_1 - \nabla U_2|^2 \di x \leq \int_{\bt^d} (h_1 - h_2) (U_1 - U_2) \di x .
\ee
By Parseval-Plancherel,
\be
\e^2 \int_{\bt^d} |\nabla U_1 - \nabla U_2|^2 \di x \leq \| \nabla (U_1 - U_2) \|_{L^2} \| \nabla \Delta^{-1} (h_1 - h_2)\|_{L^2} .
\ee
By applying Young's inequality for products with a small parameter, we obtain
\be
\| \nabla U_1 - \nabla U_2 \|_{L^2} \leq \e^{-2} \| 
\nabla \Delta^{-1} (h_1 - h_2) \|_{L^2} 
\ee
as required.
\end{proof}

\begin{proof}[Proof of Proposition~\ref{prop:regU}]

The existence and uniqueness of $U$ for $h \in L^\infty(\bt^d)$ is obtained as in \cite{GPI-WP}.

For the log-Lipschitz regularity, we first apply Lemma~\ref{lem:eU-Lp} so as to obtain the estimate
\be
\| e^U \|_{L^\infty} \leq \| h \|_{L^\infty} .
\ee
Thus
\be
\| \Delta U \|_{L^\infty} \leq 2 \e^{-2}  \| h \|_{L^\infty} .
\ee
The log-Lipschitz bound \eqref{nonlinear-logLipschitz} then follows from regularity estimates for solutions of the Poisson equation -- see for example \cite{Yudovich, BM, GPI1}.

Below, in Lemma~\ref{lem:U-stab}, we prove that 
\be
\| \nabla U_1 - \nabla U_2 \|_{L^2} \leq \e^{-2} \| 
\nabla \Delta^{-1} (h_1 - h_2) \|_{L^2}
\ee
We then control the $H^{-1}$ norm by applying the following estimate, due to Loeper \cite[Theorem 2.9]{Loeper}:
\be
 \| \nabla \Delta^{-1} (h_1 - h_2) \|_{L^2} \leq \max_i \| h_i \|_{L^\infty}^{1/2} W_2(h_1, h_2) .
\ee
This concludes the proof of estimate \eqref{nonlinear-Loeper}.
\end{proof}

\subsection{Case $d=3$: Gain of Integrability for the Nonlinearity}

Applying Proposition~\ref{prop:regU} in practice requires $L^\infty$ estimates on the spatial density $\rho$ for solutions of the Vlasov--Poisson system. However, in dimension $d=3$, the known techniques for obtaining such estimates for the electron Vlasov--Poisson system on $\bt^3$ rely on the representation of the electric field as a convolution with the Coulomb kernel, in order to make use of the interplay between the second-order nature of the characteristic flow and the full phase-space density. In the ionic case this representation is not available. 

Previously, in the study of global well-posedness \cite{GPI-WP}, this obstacle was overcome by making use of the decomposition $U = \bar U + \widehat U$, where 
\be \label{def:Usplit-3D}
- \e^2 \Delta \bar U = \rho - 1, \; \int_{\bt^3} \bar U \di x = 0, \qquad - \e^2 \Delta \widehat U = 1 - e^U .
\ee
The existence and uniqueness of $\bar U$ and $\widehat U$ was shown in \cite{GPI-WP}. As we discussed for $d=1$ above, for $d=3$ the term $\widehat U$ is more regular than $\bar U$ ($C^{2, \alpha}$, for some $\alpha > 0$, if $\rho \in L^p(\bt^3)$ with $p > 3/2$) while the `singular' part $\bar U$ has a convolution representation, since it satisfies a linear Poisson equation.
However, the additional regularity on $\widehat U$ is obtained at the expense of constants that diverge exponentially as $\e$ tends to zero \cite{GPI-MFQN}. 

In order to achieve a single exponential rate in Theorem~\ref{thm:main} for $d=3$, we will need to control a suitably strong norm on $\widehat U$ (in terms of a sufficiently \emph{weak} norm on $\rho$), but with constants that diverge \emph{polynomially} as $\e$ tends to zero.
To do this, we prove a gain of integrability for the nonlinearity $e^U$ (which forces the Poisson equation satisfied by $\widehat U$) compared to $\rho$. 
Crucially, the bounds in our new estimate degenerate only polynomially with respect to $\e$ rather than exponentially.

The key idea is to exploit the equation satisfied by $e^U$.
Since $\Delta(e^U) = e^U \Delta U + e^U |\nabla U|^2$, it follows that
\be \label{eq:eU}
-\e^2 \Delta(e^U) + \e^2 |\nabla U|^2 e^U = e^U(\rho - e^U) .
\ee

\begin{lemma} \label{lem:eU-Lp-eps}
	Let $d=3$. 
	Assume that $\rho \in L^q$ for some $3/2 < q < + \infty$. 
	Then, for all $r$ such that $q \leq r < + \infty$,
there exists an exponent $\alpha = \alpha(q,r)$, defined 
by
\be \label{def:alpha}
\alpha(q,r) : = \frac{q^{-1} - r^{-1}}{2/3 - q^{-1}} .
\ee
and constants $C_{q,r}, c_{q,r}>0$ 
such that
\be
\| e^{U} \|_{L^r} \leq C_{q,r} (\e^{-2})^{\alpha} \| \rho \|_{L^{q}}^{\alpha + 1}  + c_{q,r} \e^2,
\ee
\end{lemma}
\begin{proof}
If $r=q$ the result follows directly from Lemma~\ref{lem:eU-Lp}. For $r > q$,
test equation \eqref{eq:eU} with the function $e^{(r-2) U}$; after integrating by parts, we obtain the following equality:
\be \label{est:first-test}
\frac{4}{r-1} \e^2 \int_{\bt^3} |\nabla e^{\frac{r-1}{2}U}|^2 \di x + \int_{\bt^3} e^{rU} \di x = \int_{\bt^3} e^{(r-1)U}\rho  \di x .
\ee
By the Sobolev-Gagliardo-Nirenberg inequality on the torus \cite{Strichartz, BenyiOh}, there exists a constant $C>0$, independent of $r$, such that
\be \label{est:SGN}
\left \| e^{\frac{r-1}{2}U} - \left \langle e^{\frac{r-1}{2}U}\right \rangle \right \|_{L^6} \leq C \|\nabla e^{\frac{r-1}{2}U}\|_{L^2} ,
\ee
where $\langle \cdot \rangle$ denotes the average value:
\be
\left \langle e^{\frac{r-1}{2}U} \right \rangle : = \frac{1}{|\bt^3|} \int_{\bt^3} e^{\frac{r-1}{2}U} \di x = \int_{\bt^3} e^{\frac{r-1}{2}U} \di x,
\ee
where the last equality follows since we have normalised the torus $\bt^3$ such that $|\bt^3|=1$ (Section~\ref{sec:torus}).
We therefore expect that estimate \eqref{est:first-test} will imply that $e^{\frac{r-1}{2}U} \in L^{2 r '} \cap L^6$, if we can control the right hand side.

The right hand side of \eqref{est:first-test} is
\be
 \int_{\bt^3} e^{(r-1)U}\rho  \di x =  \int_{\bt^3} \left( e^{\frac{(r-1)U}{2}} \right)^2 \rho  \di x.
 \ee
 Since $r > q > 3/2$, then $3 > q' > r'$. Thus there exists $\theta \in (0,1)$ such that
 \be \label{def:theta}
\frac{1}{q'} = 1 - \frac{1}{q} =\frac{1-\theta}{3} + \frac{\theta}{r '} , \quad \text{i.e.} \; \theta = \frac{2/3 - 1/q}{2/3- 1/r} .
\ee
We now write
 \be
 \int_{\bt^3} \left( e^{\frac{(r-1)U}{2}} \right)^2 \rho  \di x  = \int_{\bt^3} \left( e^{\frac{(r-1)U}{2}} \right)^{2\theta} \left( e^{\frac{(r-1)U}{2}} \right)^{2(1-\theta)} \rho  \di x .
\ee
To handle the average, we observe that, since $1-\theta < 1$, for all $a,b \geq0$ we have the estimate
\be \label{est:sum-fractional-power}
(a + b)^{2(1-\theta)} \leq 2^{1-\theta}(a^{2(1-\theta)} + b^{2(1-\theta)}).
\ee
By writing $e^{\frac{(r-1)U}{2}} = \left (e^{\frac{(r-1)U}{2}} - \left \langle e^{\frac{(r-1)U}{2}} \right \rangle \right ) + \left \langle e^{\frac{(r-1)U}{2}} \right \rangle$ and applying \eqref{est:sum-fractional-power}, we find that
\begin{multline}
\int_{\bt^3} \left( e^{\frac{(r-1)U}{2}} \right)^{2\theta} \left( e^{\frac{(r-1)U}{2}} \right)^{2(1-\theta)} \rho  \di x \leq 2^{1- \theta} \underbrace{\int_{\bt^3} \left( e^{\frac{(r-1)U}{2}} \right)^{2\theta} \left( e^{\frac{(r-1)U}{2}} - \left \langle e^{\frac{(r-1)U}{2}} \right \rangle \right)^{2(1-\theta)} \rho  \di x}_{=: I_1} \\
+ 2^{1-\theta} \underbrace{ \left \langle e^{\frac{(r-1)U}{2}} \right \rangle^{2(1-\theta)} \int_{\bt^3} \left( e^{\frac{(r-1)U}{2}} \right)^{2\theta}  \rho  \di x}_{=: I_2} .
\end{multline}
To estimate $I_1$, we interpolate between $L^{2r'}$ and $L^6$:
by the choice of $\theta$ in \eqref{def:theta},
\begin{align}
I_1 &= \int_{\bt^3} \left( e^{\frac{(r-1)U}{2}} \right)^{2\theta} \left( e^{\frac{(r-1)U}{2}} - \left \langle e^{\frac{(r-1)U}{2}} \right \rangle \right)^{2(1-\theta)} \rho  \di x \\
 &\leq \| e^{\frac{(r-1)U}{2}} \|_{L^{2 r'}}^{2 \theta} \| e^{\frac{(r-1)U}{2}} - \left \langle e^{\frac{(r-1)U}{2}} \right \rangle \|_{L^6}^{2(1-\theta)} \| \rho \|_{L^q} .
\end{align}
Then, by \eqref{est:SGN}, for some constant $C_{q,r}>0$ independent of $\e$ (that may change from line to line),
\be
I_1 \leq C_{q,r} \| e^{rU} \|_{L^1}^{\theta/r'}  \| \nabla (e^{\frac{(r-1)U}{2}} ) \|_{L^2}^{2(1-\theta)} \| \rho \|_{L^q} .
\ee
Hence
\be
2^{1-\theta} I_1 \leq \| e^{rU} \|_{L^1}^{\theta/r'}  \left ( \frac{4 \e^2}{(1-\theta)(r-1)} \| \nabla (e^{\frac{(r-1)U}{2}} ) \|_{L^2}^2 \right )^{1-\theta} \, C_{q,r} \e^{-2(1-\theta)} \| \rho \|_{L^q} .
\ee
By Young's inequality for products with exponents $\frac{r'}{\theta}, \frac{1}{1-\theta}$ and $\frac{r}{\theta}$,
\be \label{est:I1}
2^{1-\theta} I_1 \leq \frac{\theta}{r'}\| e^{rU} \|_{L^1} + \frac{4 \e^2}{r-1} \| \nabla (e^{\frac{(r-1)U}{2}} ) \|_{L^2}^2 + C_{q,r} \left ( \e^{-2(1-\theta)} \| \rho \|_{L^q} \right )^{r/\theta}.
\ee

For $I_2$, we use only that $e^{\frac{(r-1)U}{2}} \in L^{2r '}$: by H\"older's inequality,
\be
\left \langle e^{\frac{(r-1)U}{2}} \right \rangle = \left \|  e^{\frac{(r-1)U}{2}} \right \|_{L^1} \leq C_r \| e^{\frac{(r-1)U}{2}} \|_{L^{2r'}} \leq C_r \| e^{rU} \|_{L^1}^{1/2r'} .
\ee
Furthermore, by the choice of $\theta$ in \eqref{def:theta},
\begin{align}
\int_{\bt^3} \left( e^{\frac{(r-1)U}{2}} \right)^{2\theta}  \rho  \di x  \leq C \| e^{\frac{(r-1)U}{2}} \|_{L^{2r'}}^{2 \theta} \| \rho \|_{L^q}  \leq C \| e^{rU} \|_{L^{1}}^{\theta/r'} \| \rho \|_{L^q}  .
\end{align}
It follows that (recalling again that $|\bt^3| = 1$)
\begin{align}
2^{1-\theta} I_2 &= 2^{1-\theta} \left \langle e^{\frac{(r-1)U}{2}} \right \rangle^{2(1-\theta)} \int_{\bt^3} \left( e^{\frac{(r-1)U}{2}} \right)^{2\theta}  \rho  \di x \\
& \leq \left ( 2 |\bt^3|^{-  \frac{1}{r'}} \| e^{rU} \|_{L^1}^{1/r'} \right)^{1-\theta} |\bt^3|^{\frac{1-\theta}{3}} \| e^{rU} \|_{L^{1}}^{\theta/r'} \| \rho \|_{L^q}  \\
& \leq \| e^{rU} \|_{L^1}^{1/r'} \, C_r^{1-\theta} \| \rho \|_{L^q}  \\
& \leq \left ( (1-\theta)\| e^{rU} \|_{L^1} \right )^{1/r'} (1-\theta)^{-1/r'} C_r^{1-\theta} \| \rho \|_{L^q} . 
\end{align}
By Young's inequality with exponents $\frac{1}{r}$, $\frac{1}{r'}$, we find that
\be \label{est:I2}
2^{1-\theta} I_2 \leq \frac{(1-\theta)}{r'} \| e^{rU} \|_{L^1} + C_{q,r} \| \rho \|_{L^q}^r . 
\ee
Altogether, by estimates \eqref{est:first-test}, \eqref{est:I1} and \eqref{est:I2} we have
\begin{multline}
\frac{4}{r-1} \e^2 \int_{\bt^3} |\nabla e^{\frac{r-1}{2}U}|^2 \di x + \int_{\bt^3} e^{rU} \di x \leq  \frac{1}{r'} \| e^{rU} \|_{L^1} + \frac{4 \e^2}{r-1} \| \nabla (e^{\frac{(r-1)U}{2}} ) \|_{L^2}^2 \\
+ C_{q,r} \left ( \e^{-2(1-\theta)} \| \rho \|_{L^q} \right )^{r/\theta}
+ C_{q,r} \| \rho \|_{L^q}^r .
\end{multline}
We rearrange this to find
\be
\| e^U \|_{L^r}^r \leq C_{q,r} \left ( \e^{-2(1-\theta)} \| \rho \|_{L^q} \right )^{r/\theta}
+ C_{q,r} \| \rho \|_{L^q}^r  .
\ee
The second term is lower order:
by Young's inequality,
\be
\| \rho \|_{L^q}^r \leq \theta \left (  \e^{-2(1-\theta)} \| \rho \|_{L^q} \right )^{r/\theta} + (1-\theta) \e^{2r}
\ee
Thus
\be
\| e^U \|_{L^r}^r \leq C_{q,r} \left ( \left (  \e^{-2(1-\theta)} \| \rho \|_{L^q} \right )^{r/\theta} +  \e^{2r} \right ).
\ee
Finally, by taking the $r$th root we find that
\be
\| e^U \|_{L^r} \leq C_{q,r} \left (  \e^{-2(\theta^{-1} - 1)} \| \rho \|_{L^q}^{\theta^{-1}}+  \e^{2} \right ).
\ee
Finally, we compute the exponent $\alpha = \theta^{-1} - 1$ as a function of $q$ and $r$:
\be
\alpha(q,r) : =  \theta(q,r)^{-1}-1 = \frac{2/3 - 1/r }{2/3 - 1/q }  -1= \frac{q^{-1} - r^{-1}}{2/3 - q^{-1}} .
\ee

\end{proof}

\section{Stability} \label{sec:Stability}

In this section, we apply the estimates for the Poisson--Boltzmann equation from Section~\ref{sec:ElectricField} to obtain new quasineutral Wasserstein stability estimates for the ionic Vlasov--Poisson system \eqref{eq:vpme}.

\subsection{Case $d=1$} \label{sec:stability_1}

In one dimension, we have a weak-strong type estimate along the lines of \cite{Hauray, IHK1}.
\begin{lemma} \label{lem:stability-1d}
Let $d=1$.
Let $f_1, f_2$ be two weak solutions of the $(VPME)_{\e}$ system \eqref{eq:vpme}.
Assume that
\be
\rho_1(s,x) := \int_{\br} f_1(s,x,v) \di v \in L^1 \left ( [0, t) ; L^\infty(\bt) \right ).
\ee
($f_2$ may be a measure solution). Set
$A(t) : =  \| \rho_1(t,\cdot) \|_{L^\infty(\bt)}$.
Then
\be
W_1(f_1(t), f_2(t)) \leq \e^{-2} \exp\left(C\int_0^t \left( A(s) + \e^{-2} \right)\di s\right) W_1(f_1(0), f_2(0)).
\ee
\end{lemma}

\begin{proof}
In this proof, we make use of the characteristic flows associated to the two solutions. Write
\be
E_i = \bar E_i + \widehat E_i,
\ee
according to the decomposition of Section~3.1. We consider trajectories $(X_i(t),V_i(t))$ satisfying
\be \label{eq:chars_1d}
X_i(t)=x+\int_0^t V_i(s)\di s, \qquad V_i(t)=v+\int_0^t E_i(s,X_i(s))\di s,
\ee
for given initial data $(x,v)\in \bt\times\br$, where $E_i=-U_i'$ and
\be
-\e^2 U_i'' = \rho_{f_i} - e^{U_i}.
\ee

Since $\rho_1 \in L^1([0,t);L^\infty(\bt))$ by assumption, the singular part $\bar E_1$ is Lipschitz in the space variable for almost every time, with
\be
\| \bar E_1(s) \|_{\Lip(\bt)} \leq \e^{-2}\left(1+\|\rho_1(s)\|_{L^\infty(\bt)}\right).
\ee
Moreover, Proposition~\ref{prop:regU-1d} yields
\be
\|\widehat E_1(s)\|_{\Lip(\bt)} \leq \e^{-4},
\ee
and $E_1$ is bounded on $[0,t]\times \bt$. Hence the characteristic ODE \eqref{eq:chars_1d} has a unique global solution
\be
Z_1(t;x,v)=(X_1(t;x,v),V_1(t;x,v))
\ee
for each $(x,v)\in \bt\times\br$.

When $f_2$ is a measure, the field $E_2$ may have jump singularities, as can be seen from the form of the Green's function $G_1$ \eqref{def:G1}. Thus the trajectory $(X_2,V_2)$ need not be unique for a given $(x,v)\in \bt\times\br$. Nevertheless, as discussed in \cite{IHK1}, one may apply \cite[Theorem 3.2]{Ambrosio2008} to the continuity equation on $\bt\times\br$ with vector field
\be
b_2(s,x,v):=(v,E_2(s,x)).
\ee
Since $\rho_{f_2}(s)$ is a probability measure on $\bt$ for each $s$, Proposition~\ref{prop:regU-1d} implies that $E_2$ is bounded on $[0,t]\times \bt$. Identifying $\bt$ with its fundamental domain $[-1/2,1/2)$, we therefore have
\be
\frac{|b_2(s,x,v)|}{1+|(x,v)|}
\leq \frac{|v|+\|E_2\|_{L^\infty([0,t]\times \bt)}}{1+|x|+|v|}
\leq 1+\|E_2\|_{L^\infty([0,t]\times \bt)},
\ee
since $|x|\leq 1/2$. Thus the integrability condition required by the superposition principle is satisfied automatically, and no additional moment assumption on $f_2$ is needed.

It follows that there exists a probability measure $\nu_2$ on the space of continuous paths
\be
Z_2=(X_2,V_2)\in C([0,+\infty);\bt\times\br),
\ee
supported on paths satisfying \eqref{eq:chars_1d} for some initial datum $(x,v)\in \bt\times\br$, and such that for all continuous functions $\psi\in C(\bt\times\br)$ and all $t\in[0,+\infty)$,
\be
\int \psi(X_2(t),V_2(t))\di \nu_2 = \int_{\bt\times\br}\psi(x,v)\,f_2(t,\di x \di v).
\ee
Moreover, $\nu_2$ can be disintegrated with respect to the initial value of $(X_2,V_2)$:
\be
\int \psi(X_2(t),V_2(t))\di \nu_2
=
\int_{\bt\times\br}\int \psi(X_2(t),V_2(t))\di \tilde \nu_2^{(x,v)}\, f_2(0,\di x\di v),
\ee
where each $\tilde \nu_2^{(x,v)}$ is a probability measure on paths satisfying
\be
(X_2(0),V_2(0))=(x,v).
\ee

Using these flows, we construct a functional controlling $W_1(f_1,f_2)$. Let
\be
\pi_0 \in \mc{P}((\bt\times\br)^2)
\ee
be a coupling of the initial data $f_1(0)$ and $f_2(0)$. For $t>0$, define $\pi_t$ by requiring that for all test functions $\psi\in C((\bt\times\br)^2)$,
\be
\int_{(\bt\times\br)^2}\psi(x,v;y,w)\di \pi_t(x,v,y,w)
=
\int_{(\bt\times\br)^2}\int \psi(Z_1(t;x,v);Z_2(t))\di \tilde \nu_2^{(y,w)}(Z_2)\di \pi_0(x,v,y,w).
\ee

Now define
\be
D(t):=\int \left(\e^{-2}|x-y|+|v-w|\right)\di \pi_t(x,v,y,w).
\ee
Then $D(t)$ is an upper bound on $\e^{-2}W_1(f_1(t),f_2(t))$. Using \eqref{eq:chars_1d}, we obtain
\be
D(t)\leq D(0)+\int_0^t\int\int \left(\e^{-2}|V_1(s)-V_2(s)|+|E_1(s,X_1(s))-E_2(s,X_2(s))|\right)\di \tilde \nu_2 \di \pi_0 \di s.
\ee

In \cite{Hauray}, the following estimate is proved:
\be
\int \int |\bar E_1(X_1(s))-\bar E_2(X_2(s))|\di \tilde \nu_2 \di \pi_0
\leq
8\|\rho_1(s)\|_{L^\infty(\bt)}\int\int \e^{-2}|X_1(s)-X_2(s)|\di \tilde \nu_2 \di \pi_0.
\ee

Then
\begin{align}
\int \int |\widehat E_1(X_1(s))-\widehat E_2(X_2(s))|\di \tilde \nu_2 \di \pi_0
&\leq \int \int |\widehat E_1(X_1(s))-\widehat E_2(X_1(s))|\di \tilde \nu_2 \di \pi_0 \notag\\
&\qquad + \int \int |\widehat E_2(X_1(s))-\widehat E_2(X_2(s))|\di \tilde \nu_2 \di \pi_0 \notag\\
&\leq \|\rho_1(s)\|_{L^2(\bt)} \|\widehat E_1(s)-\widehat E_2(s)\|_{L^2(\bt)} \notag\\
&\qquad + \|\widehat E_2(s)\|_{\Lip(\bt)} \int \int |X_1(s)-X_2(s)|\di \tilde \nu_2 \di \pi_0 \notag\\
&\leq \|\rho_1(s)\|_{L^\infty(\bt)}^{1/2}\e^{-3}W_1(\rho_1(s),\rho_2(s)) \notag\\
& \qquad+ \e^{-4}\int \int |X_1(s)-X_2(s)|\di \tilde \nu_2 \di \pi_0 \notag\\
&\leq C\left(\|\rho_1(s)\|_{L^\infty(\bt)}+\e^{-2}\right)\int \int \e^{-2}|X_1(s)-X_2(s)|\di \tilde \nu_2 \di \pi_0.
\end{align}

Altogether,
\be
D(t)\leq D(0)+C\int_0^t \left(\|\rho_1(s)\|_{L^\infty(\bt)}+\e^{-2}\right)D(s)\di s.
\ee
By Gr\"onwall's inequality,
\be
D(t)\leq \exp\left(C\int_0^t \left(A(s)+\e^{-2}\right)\di s\right)D(0),
\ee
and hence
\be
W_1(f_1(t),f_2(t))
\leq
\e^{-2}\exp\left(C\int_0^t \left(A(s)+\e^{-2}\right)\di s\right) D(0).
\ee
Taking the infimum over all couplings $\pi_0$ completes the proof.
\end{proof}

\subsection{Case $d=2,3$} \label{sec:stability_23}

Using the estimates of Proposition~\ref{prop:regU}, we are able to prove the following stability estimate for use in dimension $d=2,3$, the VPME equivalent of \cite[Theorem 3.1]{Iac22}.

To state this result, we introduce the notation $H$ for the function
\be \label{def:H}
H(w) = \sqrt{\left |\log ( w |\log \frac{1}{2} w| ) \right |} \qquad w \in (0, 2 e^{-1})
\ee 
which is strictly decreasing and diverges to $+ \infty$ as $w$ tends to zero. Furthermore, for real numbers $a \in \br$ we use the notation $a_+ = \max \{a, 0\}$ to denote the positive part of $a$.

\begin{prop}
\label{prop:stability}
Let $d=2,3$.
Let $\e \leq 1$, and let $f_1, f_2$ be two weak solutions of the $(VPME)_{\e}$ system \eqref{eq:vpme}, and set 
$$
\rho_1:= \int_{\br^d} f_1 \, dv, \quad \rho_2= \int_{\br^d} f_2 \, dv.
$$ 
Define the function
\be
\label{eq:At}
B(t):=\|\rho_1(t)\|_{L^\infty(\mathbb T^d)}+\|\rho_2(t)\|_{L^\infty(\mathbb T^d)},
 \ee
 and assume that $B(t) \in L^1([0,T])$ for some $T>0$. There exist a dimensional constant $C_d>0$ such that the following holds:
if $W_2(f_1(0),f_2(0))\leq \e \sqrt{2 e^{-1}}$ then, for all $t \in [0,T]$,
\be \label{eq:stability23}
\min \left \{ \frac{2\e}{e},  W_2(f_1(t),f_2(t))^2 \right \} \leq 2 e^{-\left( H \left [ \e^{-2}W_2(f_1(0),f_2(0))^2 \right ] - \frac{C_d}{\e}\int_0^t B(s)\,ds\right)_+^2},
\ee
where we interpret the right hand side as zero if $W_2(f_1(0),f_2(0)) = 0$.
\end{prop}

\begin{remark}
In particular the conclusion of the theorem is meaningful when $\e^{-1}W_2(f_1(0),f_2(0))$ is small enough that
\begin{equation}\label{eq:main hyp}
H \left ( \e^{-2}W_2(f_1(0),f_2(0))^2 \right ) \geq \frac{C_d}{\e}\int_0^T B(s)\,ds+\sqrt{\left|\log\left(\frac{\e}e\right)\right|},
\end{equation}
in which case the estimate \eqref{eq:stability23} truly provides a bound for the Wasserstein distance, becoming
\be \label{eq:stability23-W2}
 W_2(f_1(t),f_2(t))^2 \leq 2 e^{-\left( H [ \e^{-2}W_2(f_1(0),f_2(0))^2 ] - \frac{C_d}{\e}\int_0^t B(s)\,ds\right)_+^2}.
\ee
\end{remark}

\begin{proof}
The proof of this result follows exactly the same method as the proof of \cite[Theorem 3.1]{Iac22}, except that we replace \cite[Lemma 3.6]{Iac22} (stability and regularity estimates for the electric field) with Proposition~\ref{prop:regU} everywhere that it is used.
\end{proof}

\section{Growth Estimates on $\| \rho \|_{L^\infty}$} \label{sec:density}
The goal of this section is to obtain new $L^\infty$ bounds on the spatial density $\rho_f$ for a solution $f$ of the VPME system \eqref{eq:vpme}, so to control
the quantity $B(t)$ appearing in the statement of Proposition~\ref{prop:stability}.
We will prove the following proposition.

\begin{prop} \label{prop:mass-growth}
Let $d = 2,3$ and $f_{0,\e} \in L^1 \cap L^\infty(\bt^d \times \br^d)$ satisfy the assumptions \ref{hyp:moments}. Let $T>0$ be fixed. Then there exists $\e_\ast$ depending on $T, k$ and $d$ such that for all $\e \leq \e_\ast$:

If $d=2$, there exists a constant $C>0$ depending on $C_0$ such that
\be \label{est:A-dim2}
\int_0^T B(s) \di s \leq C \e^{-4} (1+ T)^3 ( 1 + \log{ (1 + \e^{-2}T) })
\ee

If $d=3$, there exists a constant $C_k>0$ depending on $C_0$ and $k$ such that
\be \label{est:A-dim3}
\int_0^{T} B(s) \di s \leq C_k (T +1)^{4} \; \e^{-30} 
\ee

\end{prop}

The strategy will be to control the maximal possible growth in the velocity coordinate of a characteristic trajectory of the system.
More precisely, we will use the following notation for the characteristic flow: let the pair $X(t; s,x,v), V(t; s,x,v)$ denote the 
solution of the system of ODEs
\be
\dot X(t; s,x,v) = V(t; s,x,v), \; \dot V(t; s, x,v) = E(X(t; s,x,v)), \quad X(s; s, x,v) = x, \; V(s ; s,x,v) = v.
\ee
We will study the quantity
\be \label{def:growth-V}
Q_\ast(t) : = \sup_{(x,v) \in \bt^d \times \br^d} |V(t; 0,x,v) - v|.
\ee
This has been shown in \cite{GPI-WP} to be finite in the case $d=2,3$ for all $t \in [0,+\infty)$, under the assumptions \ref{hyp:moments} with $k_0 > d(d-1)$. In the case $d(d-1) \geq k_0 > d$ we will apply our argument to a series of regularised solutions (see \cite[Section 6]{GPI-WP} for the procedure); the estimates we obtain will be uniform in the regularisation parameter and hence will pass to the limit.

In particular, following this argument for any fixed $\e >0$ shows that, under assumption \ref{hyp:moments}, the
global weak solution $f_\e$ constructed in \cite[Theorem 6.1]{GPI-WP} in fact has bounded density $\rho_{f_\e} \in L^\infty_{\text{loc}}([0, + \infty) ; L^\infty(\bt^d))$, and thus by \cite[Theorem 4.1]{GPI-WP} is the unique solution in this class.
We may therefore relax the condition $m_0 > d(d-1)$ in \cite[Theorem 2.1]{GPI-WP} to $m_0 > d$.

Our interest in the quantity $Q_\ast$ is motivated by the following lemma: it gives us the control of $\rho_f$ in $L^\infty$ that we seek.

\begin{lemma}  \label{lem:rho-Linfty}
Let $d \geq 1$ and $t \geq 0$. Assume that the assumptions \ref{hyp:moments} hold and that quantity $Q_\ast(t)$, defined in equation \eqref{def:growth-V}, is finite. Then
\be
\| \rho_f(t) \|_{L^\infty(\bt^d \times \br^d)} \leq C (1 + Q_\ast(t)^d) .
\ee
\end{lemma}
\begin{proof}
Using the representation of $f$ in terms of the characteristic flow and the weighted $L^\infty$ estimate  \ref{hyp:moments} on $f_{0,\e}$, we may obtain the estimate
\be \label{est:f-upper1}
f(t, x,v) \leq \frac{C}{1 + |V(t;0,x,v)|^{k_0}} \quad \text{for all } (t,x,v) \in [0,+\infty) \times \bt^d \times \br^d .
\ee
We note by the (reverse) triangle inequality that
\be
|V(t;0,x,v)| \geq \left ( |v| - |V(t; 0,x,v) - v| \right )_+ \geq \left ( |v| - \sup_{x',v'} |V(t; 0,x',v') - v'| \right )_+  \geq (|v| - Q_\ast(t))_+ ,
\ee
where the last inequality follows directly from the definition of $Q_\ast$.
We then deduce from \eqref{est:f-upper1} that
\be \label{est:f-upper-Qast}
f(t,x,v) \leq \frac{C}{1 + \left ( |v| - Q_\ast(t) \right )_+^{k_0}}.
\ee

Next, we integrate \eqref{est:f-upper-Qast} over all $v \in \br^d$ to obtain a bound on $\rho_f$:
\be
\rho_f(x) \leq C \int_{\br^d} \frac{1}{1 + (|v|-Q_\ast(t))_+^{k_0}} \di v.
\ee
The integrand is radially symmetric in $v$. We therefore change to polar coordinates to find that
\begin{align}
\rho_f(x) & \leq C \int_0^\infty  \frac{r^{d-1}}{1 + (r-Q_\ast(t))_+^{k_0}} \di r \\
& \leq C \int_0^{Q_\ast(t)} r^{d-1} \di r + C \int_{Q_\ast(t)}^\infty  \frac{r^{d-1}}{1 + (r-Q_\ast(t))_+^{k_0}} \di r \\
& \leq C \int_0^{Q_\ast(t)} r^{d-1} \di r + C \int_0^\infty  \frac{(r + Q_\ast(t))^{d-1}}{1 + r^{k_0}} \di r .
\end{align}
We observe that
\be
(r + Q_\ast(t))^{d-1} \leq 2^{d-1} \left ( r^{d-1} + Q_\ast(t)^{d-1} \right ) .
\ee
Next compute the integral:
\be
\int_0^{Q_\ast(t)} r^{d-1} \di r = \frac{1}{d} Q_\ast(t)^d.
\ee
Finally, since $k_0 > d$ we may estimate
\be
\int_0^\infty  \frac{r^{d-1}}{1 + r^{k_0}} \di r +
\int_0^\infty  \frac{1}{1 + r^{k_0}} \di r \leq C_{d, k_0} < + \infty.
\ee
We conclude that
\be
\rho_f(x) \leq C_{d, k_0} \left (Q_\ast(t)^d + Q_\ast(t)^{d-1} + 1 \right )  \leq C_{d, k_0}(1 + Q_\ast(t)^d)
\ee
for all $x \in \bt^d$. The statement follows immediately.
\end{proof}

Our aim is therefore to obtain estimates on $Q_\ast$, as these will entail estimates on $B$.
As in previous works on this subject \cite{IHK2, GPI-MFQN}, our method will differ depending on the dimension.

\subsection{Case $d=2$}

\begin{prop} \label{prop:mass-growth-2d}
Let $d=2$.
Let $(1 + |v|^{k_0})f_0 \in L^1 \cap L^\infty(\bt^2 \times \br^2)$ for some $k_0 > 2$. Let $f$ denote the unique bounded density solution of \eqref{eq:vpme} with initial datum $f_0$. Then the spatial density $\rho_f$ satisfies the estimate
\be
\sup_{[0,t]} \| \rho_f(t, \cdot) \|_{L^\infty(\bt^2)} \leq  C ( 1 + \e^{-4} t^2) (1 + \log(1 + \e^{-2} t)) , \quad \text{for all} \; t>0.
\ee
\end{prop}

We will need the following estimate for the electric field, which can be found in \cite[Lemma 6.2]{GPI-MFQN}.

\begin{lemma} \label{lem:2d-interpolation}
Let $h \in L^1 \cap L^\infty(\bt^2)$, and let $U$ be the unique $W^{1,2}(\bt^2)$ solution of the Poisson equation
\be
\e^2 \Delta U = h .
\ee
Then there exists a constant $C$ depending only on $\| h \|_{L^2(\bt^2)}$ such that
\be
\| \nabla U \|_{L^\infty(\bt^2)} \leq C \e^{-2} \left ( 1 +  |\log \| h \|_{L^\infty(\bt^2)} |^{1/2}  \right ) .
\ee
\end{lemma}

\begin{lemma} \label{lem:characteristic-growth}
Under the assumptions of Proposition~\ref{prop:mass-growth-2d}, $Q_\ast$ satisfies the estimate
\be
Q_\ast(t)^2 \leq C \e^{-4} t^2 (1 + \log(1 + \e^{-2} t)) ,
\ee
for some constant $C>0$ independent of $t$ and $\e$.
\end{lemma}
\begin{proof}
By Lemma~\ref{lem:eU-Lp}, the electrostatic potential $U$ satisfies the assumptions of Lemma~\ref{lem:2d-interpolation}, with $\| h \|_{L^\infty(\bt^2)} \leq C(1 + Q_\ast(t)^2)$ by Lemma~\ref{lem:rho-Linfty}. Hence the electric field is uniformly bounded:
\be \label{est:electric-Qast}
\| E(t) \|_{L^\infty(\bt^2)} = \| \nabla U (t) \|_{L^\infty(\bt^2)} \leq C \e^{-2} \left ( 1 +  |\log C(1+Q_\ast(t)^2) |^{1/2}  \right ) .
\ee
Next, observe that for any characteristic trajectory $(X(t; 0,x,v), V(t; 0,x,v))$,
\be
|V(t; 0,x,v) - v| \leq \int_0^t |E(X(\tau; 0,x,v))| \di \tau \leq t \| E \|_{L^\infty([0,t] \times \bt^d)} .
\ee
By \eqref{est:electric-Qast},
\be
|V(t; 0,x,v) - v| \leq C \e^{-2} t \sup_{s \in [0,t]} \left ( 1 +  |\log C(1+Q_\ast(s)^2) |^{1/2}  \right ) .
\ee
Taking supremum over $(x,v)$, we obtain that
\begin{align}
Q_\ast(t) &\leq C \e^{-2} t \sup_{s \in [0,t]} \left ( 1 +  |\log C(1+Q_\ast(s)^2) |^{1/2}  \right ) \\
& \leq C t \e^{-2} \left ( 1 +  |\log C(1+ Q_\ast(t)^2) |^{1/2}  \right ) .
\end{align}
Since $\sqrt{x+y} \leq \sqrt{x} + \sqrt{y} \leq \sqrt{2(x+y)} $, we find that
\be
Q_\ast(t) \leq C \e^{-2} t \left ( 1 +  \log (1+ Q_\ast(t)^2 )  \right )^{1/2} ,
\ee
where $C>0$ is a larger constant. We rearrange this to obtain the inequality
\be
\frac{Q_\ast(t)^2}{ 1 +  \log (1+ Q_\ast(t)^2 ) } \leq C  \e^{-4} t^2 . 
\ee
The function
\be
b : y \mapsto \frac{y}{1 + \log(1+y)} 
\ee
is continuous and strictly increasing for $y \geq 0$ -- see Lemma~\ref{lem:app-inverse} below -- and therefore has a well-defined, continuous, strictly increasing inverse $b^{-1}$. We also show in Lemma~\ref{lem:app-inverse} that this inverse obeys the bound
\be
b^{-1}(u) \leq 2 u (1 + \log(1+u)) .
\ee
We deduce that, for some $C > 0$,
\be
Q_\ast(t)^2 \leq C \e^{-4} t^2 (1 + \log(1 + C \e^{-4} t^2)) .
\ee
Since $\log(1+u) \leq \log(1 + \sqrt{u})^2 \leq 2 \log(1 + \sqrt{u})$, after possibly enlarging the constant $C>0$ we find that
\be
Q_\ast(t)^2 \leq C \e^{-4} t^2 (1 + \log(1 + \e^{-2} t)) 
\ee
which completes the proof.

\end{proof}

Proposition~\ref{prop:mass-growth-2d} then follows from Lemma~\ref{lem:rho-Linfty} and Lemma~\ref{lem:characteristic-growth}. The estimate \eqref{est:A-dim2} then follows upon integrating over time.

\subsection{Case $d=3$} \label{sec:moments_3}

In order to control the $L^\infty$ norm of the density in the three dimensional case, we will make use of techniques for estimating the growth of characteristic trajectories over time. These can be traced back to the development of the well-posedness theory for the 3D electron Vlasov--Poisson system \cite{Pfaffelmoser, Schaeffer, Batt-Rein} and results on the propagation of moments \cite{Pallard, ChenChen} on the torus $\bt^3$, where the approach of Lions-Perthame does not apply and techniques based on characteristic trajectories are used instead.

Our method takes as a starting point the techniques of Chen and Chen \cite{ChenChen} for the propagation of moments for the electron model.
We therefore introduce the notation
\be \label{def:Mk}
M_k(t) : = \sup_{s \in[0,t]} \int_{\bt^3 \times \br^3} (1 + |v|^k) f(s,x,v) \di x \di v 
\ee
for the velocity moment of order $k$.

We will prove an estimate for small increments of the characteristic trajectories. For all $t \in[0,T]$ and $\delta \in (0,t]$, we define $Q(t,\delta)$ by
\be
Q(t, \delta) : = \sup_{(x,v) \in \bt^3 \times \br^3} \int_{t-\delta}^t |E(X(s ; 0, x, v)) | \di s .
\ee 
Observe that $Q_\ast(t) \leq Q(t,t)$, so that obtaining estimates on $Q(t,\delta)$ for all $\delta$ will suffice to control the density.

The main new steps required compared to \cite{ChenChen} are:
\begin{enumerate}[label=(\roman*)]
\item To handle the fact that the electric field depends on $\rho$ through the nonlinear Poisson--Boltzmann equation rather than a linear Poisson equation. We will do this by using the \emph{splitting} of the field $E = \bar E + \widehat E$, $\bar E = - \nabla \bar U$ and $\widehat E = - \nabla \widehat U$ for $\bar U, \widehat U$ as defined in Equation \eqref{def:Usplit-3D}; and
\item To quantify carefully the dependence of constants on $\e$ in the quasineutral scaling. For this we will need to revisit the arguments of \cite{ChenChen} in detail.
\end{enumerate}

First, we relate estimates on the moments to estimates on the electric field.
The Coulomb kernel in the three-dimensional torus is the function $K_{\bt^3}$ defined by $K_{\bt^3} = - \nabla G_{\bt^3}$,
\be 
- \Delta G_{\bt^3} = \delta_0 - 1 \qquad \text{in} \; \bt^3 .
\ee 
We note the following result for convolutions against 
$K_{\bt^3}$.
The case $q = +\infty$ is proved in \cite[Lemma 4.5.4]{Hormander}; the general case can be proved using a similar interpolation argument.

\begin{lemma} \label{lem:Efield-Hormander}
Let $d=3$, $1 \leq p < 3 < q \leq +\infty$ and let $h \in L^p \cap L^q$. Then the Coulomb kernel
	\be
	\| K_{\bt^3} \ast h \|_{L^\infty} \leq C_{p,q} \| h \|_{L^p}^{1-\theta} \| h \|_{L^q}^{\theta},
	\ee 
	where the exponent $\theta$ satisfies $\frac{1}{3} = \frac{1-\theta}{p} + \frac{\theta}{q}$. 
\end{lemma}

By combining Lemma~\ref{lem:Efield-Hormander} with the gain of integrability estimates for $e^U$ (Lemma~\ref{lem:eU-Lp-eps}), we obtain new  $L^\infty$ estimates for $\widehat E_\e$, in which the constants diverge at polynomial rather than exponential rate in $\e$. In order to optimise as much as possible the exponent $\zeta(3)$ in the assumption \eqref{hyp:rate}, we apply Lemma~\ref{lem:eU-Lp-eps} in two different ways, obtaining two bounds: 
\begin{enumerate}[label=(\roman*)]
\item The first depends on $f$ through the $k^{\text{th}}$ velocity moment $M_k$. The maximal power of $M_k$ for which we will be able to close the estimate on $Q(t, \delta)$ is $M_k^{\frac{1}{2(k-2)}}$; this is insufficient to provide the integrability for $e^U$ that we require without a loss in $\e$, but this loss is at a polynomial rate in $\e$.
\item The second estimate is uniform in time -- although not $\e$, with a worse rate of divergence than the first estimate -- and can therefore take over if the moment $M_k$ grows too large. 
\end{enumerate}

\begin{lemma} \label{lem:est-smooth}
	There exists a constant depending only on $\| f_0 \|_{L^\infty}$, $\mc{E}[f_0]$ such that, for all $t \in [0, +\infty)$,
	\be
	\| \widehat E_\e \|_{L^\infty} \leq C \e^{-7} (\e^{-3} \wedge M_k^{\frac{1}{2(k-2)}}) .
	\ee 
\end{lemma}
\begin{proof}
By applying Lemma~\ref{lem:Efield-Hormander} with the choice $p = \frac{5}{3}$, $q = r$, we deduce that for any $r \in (3, +\infty)$,
\be
\| \widehat E_\e \|_{L^\infty} = \e^{-2} \| K_{\bt^3} \ast e^U \|_{L^\infty} \leq C_r \e^{-2} \| e^U \|_{L^{5/3}}^{1-\beta} \| e^U \|_{L^r}^{\beta},
\ee 
where $\beta = \frac{4}{3} \left ({3} - \frac{5}{r} \right )^{-1}$.
By Lemma~\ref{lem:eU-Lp-eps}, 
\be
\| e^{U} \|_{L^r} \leq (C_r \e^{-2})^{3(3 - 5/r)} \| \rho \|_{L^{5/3}}^{5(2 - 3/r)} + c_r \e^2,
\ee
and thus 
\be
\| \widehat E_\e \|_{L^\infty} \leq C_r \e^{-2} (C_r \e^{-2})^{4} (1+ \| \rho \|_{L^{5/3}})^{10} .
\ee 
By \eqref{est:53-uniform}, $\| \rho \|_{L^{5/3}}$ is uniformly bounded and hence
\be
\| \widehat E_\e \|_{L^\infty} \leq C_r \e^{-10} .
\ee 
Finally, we fix some $r > 3$ and thereby deduce the result.

Next, by the moment interpolation estimate \eqref{est:rhoLp-moment} we recall that
\be \label{est:moment-Lp-recall}
\| \rho \|_{L^{1 + k/3}} \leq C M_k^{\frac{3}{k+3}} .
\ee
By Lemma~\ref{lem:Efield-Hormander}, for any $\eta, \phi \in [0,1]$ and $r$ satisfying $\eta + \phi \leq 1$ and
\be
\frac{1}{3} = \phi \frac{3}{k+3} + \frac{\eta}{r} + \frac{3}{5} (1-\eta-\phi) ,
\ee
\be \label{est:Eeps-3interpolation}
\| \widehat E_\e \|_{L^\infty} = \e^{-2} \| K_{\bt^3} \ast e^U \|_{L^\infty} \leq C_r \e^{-2} \| e^U \|_{L^{5/3}}^{1-\eta-\phi} \| e^U\|_{L^{1 + k/3}}^\phi \| e^U \|_{L^r}^{\eta} .
\ee 

Choose $\phi = \frac{k+3}{6(k-2)}$, which implies that
\be
\eta \left  (\frac{3}{5} - \frac{1}{r} \right ) = \frac{1}{6} .
\ee
By \eqref{est:53-uniform}, Lemmas~\ref{lem:eU-Lp} and \ref{lem:eU-Lp-eps}, \eqref{est:moment-Lp-recall} and \eqref{est:Eeps-3interpolation},
\be
\| \widehat E_\e \|_{L^\infty}  \leq C_r \e^{-2} M_k^{\frac{1}{2(k-2)}} \e^{-30(3/5 - 1/r)\eta} \leq C_r \e^{-7} M_k^{\frac{1}{2(k-2)}}
\ee 
for any admissible choice of $r$.

\end{proof}

We are able to prove the following estimate on $Q(t,\delta)$; this is analogous to the estimate \cite[Lemma 3.2]{ChenChen}, but for the ion case and quantified in $\e$.
\begin{lemma} \label{lem:est-Q}
Let $k > 3$ and assume that $\sup_{s\in[0,t]}M_k(s)$ is finite. Then, for all $\delta \in [0,t]$,
\be \label{est:Q-Q}
Q(t,\delta)^{3/2} \leq C \e^{-2} \delta^{1/2} \left((\delta Q(t,\delta))^{1/2} \Big ( Q(t,\delta)^{4/3} + M_k(t)^{\frac{1}{2(k-2)}} + \e^{-5} (\e^{-3} \wedge M_k^{\frac{1}{2(k-2)}}) \Big) + M_k(t)^{\frac{1}{2(k-2)}} \right ).
\ee
\end{lemma}

\begin{proof}
We use the decomposition $E = \bar E + \widehat E$.

To estimate $\bar E$, we use the standard local decomposition of the periodic Coulomb kernel (see e.g. \cite{Titchmarsh}). More precisely, after identifying $\bt^3$ with its fundamental domain $[-1/2,1/2)^3$, there exist smooth functions $K_0 \in C^\infty(B_{1/4}(0))$ and $K_1 \in C^\infty(\bt^3 \setminus B_{1/4}(0))$ such that
\be
K_{\bt^3}(x)= \frac{x}{4\pi |x|^3}+K_0(x), \qquad \mbox{if } |x|<\frac14,
\ee
and
\be
K_{\bt^3}(x)=K_1(x), \qquad \mbox{if } |x|\ge \frac14.
\ee
In other words, $K_0$ is the smooth remainder obtained after subtracting the Coulomb singularity from $K_{\bt^3}$ near the origin, while $K_1$ is simply the restriction of $K_{\bt^3}$ away from the singularity. In particular, both $K_0$ and $K_1$ are bounded on their respective domains.

We may then write
\begin{multline}
Q(t, \delta) \leq \sup_{(x,v) \in \bt^3 \times \br^3} C \e^{-2} \int_{t-\delta}^t \int_{\br^d} \int_{\br^d} \frac{f(s,x',v')}{|x' - X(s; 0,x,v)|^2} \mathbbm{1}_{B_{1/4}}(x' - X(s; 0,x,v)) \di x' \di v' \di s \\
 + \delta \e^{-2} \| (K_0 + K_1) \ast \rho_f \|_{L^\infty((t-\delta, t] \times \bt^3)} + \delta \| \widehat E \|_{L^\infty((t-\delta, t] \times \bt^3)} .
\end{multline}

By Lemma~\ref{lem:est-smooth}, there exists a constant $C>0$ such that
\be
\| \widehat E_\e \|_{L^\infty} \leq  C \e^{-7} (\e^{-3} \wedge M_k^{\frac{1}{2(k-2)}}).
\ee
Thus
\begin{multline} \label{est:Q-prelim}
Q(t, \delta) \leq  C \e^{-7} (\e^{-3} \wedge M_k^{\frac{1}{2(k-2)}}) \\
 + \sup_{(x,v) \in \bt^3 \times \br^3} C \e^{-2} \int_{t-\delta}^t \int_{\br^3 \times \br^3} \frac{f(s,x',v')}{|x' - X(s; 0,x,v)|^2} \mathbbm{1}_{B_{1/4}}(x' - X(s; 0,x,v)) \di x' \di v' \di s  .
\end{multline}

The second term of \eqref{est:Q-prelim} is estimated using methods based on \cite{Pfaffelmoser, Schaeffer, Batt-Rein, Pallard, ChenChen}. 
First, fix a particular characteristic trajectory $X_\ast(s) : = X(s ; 0,x_\ast ,v_\ast)$, $V_\ast(s) : = V(s ; 0,x_\ast,v_\ast)$ considered as a lifted trajectory in $\br^3 \times \br^3$, as was done in \cite{Batt-Rein}.
Next, we consider the following decomposition of the set $[t-\delta, t] \times \br^3 \times \br^3$: for some parameters $R, \gamma > 0$ to be determined, let $\Lambda_\gamma : \br^3 \to \br_+$ denote the function 
\be \label{def:Lambda}
\Lambda_\gamma(v) = 1 + \gamma |v|^2 \mathbbm{1}_{|v| \leq \gamma} + \gamma^{3-k} |v|^k \mathbbm{1}_{|v| > \gamma} ,
\ee 
and let the sets $\Omega, \Omega_G, \Omega_B, \Omega_U \subset [t-\delta, t] \times \bt^3 \times \br^3$ be defined by
\begin{align} \label{def:G}
& \Omega : = \{ (s,x,v) \in [t-\delta, t] \times \br^3 \times \br^3 : |x - X_\ast(s)| < \frac{1}{4} \} \\
&\Omega_G : = \{ (s,x,v) \in \Omega : |v - V_\ast(s)| \leq 5 Q(t,\delta) \; \text{or} \; |v| \leq 5 Q(t,\delta) \} \\ \label{def:B}
&\Omega_B : = \left \{ (s,x,v) \in \Omega : |x - X_\ast(s)| \leq \frac{R}{ \Lambda_\gamma(v)} \right \} \setminus \Omega_G \\ \label{def:U}
&\Omega_U : = \Omega \setminus (\Omega_G \cup \Omega_B) .
\end{align}

The decomposition \eqref{def:G}-\eqref{def:B}-\eqref{def:U} is taken as in \cite[Lemma 3.2]{ChenChen}, except that we replace the function $\Lambda(v) = 1 + |v|^{1+\e}$ in the definition \eqref{def:B} of the set $\Omega_B$ with $\Lambda_\gamma$ as defined above in \eqref{def:Lambda}. The purpose of this is to allow us to obtain a sharp exponent in our eventual final estimate, with no `loss of an epsilon'.
Observe that
\be \label{est:Lambda-inv-int}
\int_{\br^d} \frac{1}{\Lambda_\gamma(v)} \di v \leq \gamma^{-1} \int_{|v| \leq \gamma} |v|^{-2} +  \gamma^{k-3} \int_{|v| > \gamma}  |v|^{-k} \di v \leq C,
\ee
for some constant $C>0$ independent of $\gamma$, and
\be
\int_{\bt^d \times \br^d} \Lambda_\gamma(v) f(t,x,v) \di x \di v \leq 1 + \gamma M_2(t) +  \gamma^{3-k} M_k(t) .
\ee
From now on, we will set $\gamma : = \left ( M_k(t) M_2(t)^{-1} \right )^{1/(k-2)}$. This ensures that
\be \label{int:Lambda}
\sup_{s \leq t} \int_{\bt^d \times \br^d} \Lambda_\gamma(v) f(s,x,v) \di x \di v \leq 1 + M_2(t)^{\frac{k-3}{k-2}} M_k(t)^{\frac{1}{k-2}} \leq C M_k(t)^{\frac{1}{k-2}} ,
\ee
where the last inequality follows from $M_k \geq1$ and the conservation of energy: $M_2(t) \leq M_2(0)$. We also write $\Lambda : = \Lambda_\gamma$ to lighten the notation.

\textbf{Region $\Omega_G$:} In this region, either $|v| \leq 5 Q(t,\delta)$ or $|v - V_\ast(s)| \leq 5 Q(t,\delta)$. Hence
\be
0 \leq f\rvert_{\Omega_G}(s,x,v) \leq \| f_0 \|_{L^\infty} \left ( \mathbbm{1}_{|v| \leq 5 Q(t,\delta)} +  \mathbbm{1}_{|v - V_\ast(s) | \leq 5 Q(t,\delta)} \right ) .
\ee
Thus, for all $(s,x) \in [t- \delta, t]$,
\be
0 \leq \int_{\br^3} f (s,x, v) \di v  \leq \int_{\br^3} f \rvert_{\Omega_G} (s,x, v) \di v \leq C \| f_0 \|_{L^\infty} Q(t,\delta)^3 ,
\ee
and hence
\be
\left \| \int_{\br^3} f \rvert_{\Omega_G} (\cdot, \cdot, v) \di v \right \|_{L^\infty_{s,x}} \leq C \| f_0 \|_{L^\infty} Q(t,\delta)^3 .
\ee
It then follows by \cite[Lemma 4.5.4]{Hormander} (see Lemma~\ref{lem:Efield-Hormander}, case $q=+\infty$) that
\be
\int_{\Omega_G} \frac{f(s,x,v)}{|x - X_\ast(s)|^2} \di x \di v \di s \leq C \delta Q(t,\delta)^{4/3} .
\ee

\textbf{Region $\Omega_B$:} In this region, $|x - X_\ast(s)| \leq R \Lambda (v)^{-1}$; hence
\begin{align}
\int_{\Omega_B} \frac{f(s,x,v)}{|x - X_\ast(s)|^2} \di x \di v \di s &\leq \| f_0 \|_{L^\infty} \int_{t-\delta}^t \int_{\br^3} \int_{|y| \leq R \Lambda(v)^{-1}} |y|^{-2} \di y \di v \di s \\
& \leq C \| f_0 \|_{L^\infty} R \int_{t-\delta}^t \int_{\br^3} \Lambda (v)^{-1} \di v .
\end{align}
Thus, by \eqref{est:Lambda-inv-int},
\be
\int_{\Omega_B} \frac{f(s,x,v)}{|x - X_\ast(s)|^2} \di x \di v \di s \leq C \delta R.
\ee

\textbf{Region $\Omega_U$:} Here we follow the arguments of \cite{Batt-Rein, Pallard, ChenChen}, replacing the function $1 + |v|^{1+\e}$ in \cite{ChenChen} by $\Lambda (v)$: we wish to estimate
\be 
I_U : = \int_{\Omega_U} \frac{f(s,x,v)}{|x - X_\ast(s)|^2} \di x \di v \di s .
\ee
We perform the change of variables $(\tilde x, \tilde v) = \left ( X(t; s,x,v) , V(t; s,x,v) \right )$. Since then $f(s,x,v) = f(t,\tilde x, \tilde v)$, we have
\be
I_U = \int_{t-\delta}^t \int_{\br^3 \times \br^3}  f(t,\tilde x, \tilde v) \frac{\mathbbm{1}_U \left (s, Z(s ; t, \tilde x, \tilde v) \right )}{|X(s ; t, \tilde x, \tilde v) - X_\ast(s)|^2} \di \tilde x \di \tilde v \di s .
\ee

Next, write the $x$ domain as the union $\br^3 = \bigcup_{\alpha \in \Z^3} \alpha + \left [ - \frac{1}{2}, \frac{1}{2} \right )^3$. Then
\be
I_U = \sum_{\alpha \in \Z^3} \int_{\left [ - \frac{1}{2}, \frac{1}{2} \right )^3} \int_{\br^3} f(t, x + \alpha, v) \int_{t-\delta}^t \frac{\mathbbm{1}_U \left (s, Z(s; t, x+\alpha, v) \right )}{| X(s; t, x+\alpha, v) - X_\ast(s)|^2} \di s \di x \di v .
\ee
By periodicity we note that $f(t, x+\alpha, v) = f(t,x,v)$ and $E(s,x+\alpha) = E(s,x)$ for all $\alpha \in \Z^3$ and all $s \geq 0$. It follows that the (lifted) flow commutes with shifts in the $x$ variable:
\be
X(s; t, x+\alpha, v) = \alpha + X(s; t,x,v) , \qquad V(s; t, x+\alpha, v) = V(s; t,x,v) .
\ee
We introduce the shorthand $\tilde Z(s) = \left (\tilde X(s), \tilde V(s) \right ) = \left ( X(s; t,x,v),  V(s; t,x,v) \right )$ for $(x,v) \in \left [ - \frac{1}{2}, \frac{1}{2} \right )^3 \times \br^3$, and $\tilde X_\alpha(s) = \alpha + \tilde X(s)$, $\tilde Z_\alpha = (\tilde X_\alpha, \tilde V)$. Thus
\be
I_U = \int_{\left [ - \frac{1}{2}, \frac{1}{2} \right )^3} \int_{\br^3} f(t, x, v) \sum_{\alpha \in \Z^3}  \int_{t-\delta}^t \frac{\mathbbm{1}_U \left (s, \tilde Z_\alpha(s) \right )}{| \tilde X_\alpha (s) - X_\ast(s)|^2} \di s \di x \di v .
\ee
We need only include those $\alpha$ in the set
\be
A(x,v) : = \left \{ \alpha \in \Z^3 : \exists s \in [t-\delta, t], \; (s, \tilde X_\alpha(s), \tilde V(s) ) \in \Omega_U \right \} ;
\ee
hence
\be
I_U = \int_{\left [ - \frac{1}{2}, \frac{1}{2} \right )^3} \int_{\br^3} f(t, x, v) \sum_{\alpha \in A(x,v)}  \int_{t-\delta}^t \frac{\mathbbm{1}_U \left (s, \tilde Z_\alpha(s) \right )}{| \tilde X_\alpha (s) - X_\ast(s)|^2} \di s \di x \di v .
\ee

We now seek a lower bound on $| \tilde X_\alpha (s) - X_\ast(s)|$, given that $\left (s, \tilde X_\alpha (s), \tilde V(s) \right ) \in \Omega_U$.
First, from the definition of $\Omega_U$ we have
\be \label{est:lower-static}
|\tilde X_\alpha(s) - X_\ast(s)| > \frac{R}{\Lambda \left ( \tilde V(s) \right )} .
\ee
A second estimate can be obtained by using the dynamics of solutions to the characteristic ODE. Arguing exactly as in \cite{Schaeffer, Batt-Rein}, we may show that for all $\tau \in [t-\delta, t]$,
\be \label{est:lower-dynamic}
| \tilde X_\alpha (\tau) - X_\ast(\tau) | \geq |\tau - \tau_\ast| \left ( |\tilde V(\tau_\ast) - V_\ast( \tau_\ast) | - Q(t,\delta) \right ),
\ee
where $\tau_\ast$ is such that
\be
 |\tilde X_\alpha (\tau_\ast) - X_\ast( \tau_\ast) | = \min_{\tau \in [t-\delta, t]} | \tilde X(\tau) - X_\ast(\tau) | .
\ee
Indeed, letting
\be
\xi(\tau) : = \tilde X_\alpha (\tau) - X_\ast(\tau),
\ee
by the mean value theorem we have
\be
|\xi(\tau) - \xi(\tau_\ast) - (\tau - \tau_\ast) \dot \xi(\tau_\ast)| \leq |\tau - \tau_\ast| \sup_{\theta \in [t-\delta, t]} |\dot \xi (\tau_\ast) - \dot \xi(\theta)| \leq |\tau - \tau_\ast| Q(t,\delta) .
\ee
Since $\tau_\ast$ is a minimiser, $(\tau - \tau_\ast) \xi (\tau_\ast) \cdot \dot \xi(t_\ast) \geq 0$, and hence
\be
|  \xi(\tau_\ast) + (\tau - \tau_\ast) \dot \xi(\tau_\ast) |^2 \geq |\tau - \tau_\ast|^2 |\dot \xi(\tau_\ast)| .
\ee
We conclude by the (reverse) triangle inequality.

Next, we determine bounds for \eqref{est:lower-static} and \eqref{est:lower-dynamic} depending on the values of the characteristic trajectories at the final time $t$. First recall that, by definition of $\Omega_U$, $| \tilde V(s) - V_\ast(s)| > 5Q$. Moreover, by definition of $Q(t,\delta)$, for any $\tau \in [t-\delta, t]$ we may estimate
\begin{align}
| ( \tilde V(\tau) - V_\ast(\tau)  ) - ( \tilde V(s) - V_\ast(s) ) | & \leq 2 Q(t, \delta) \\
& \leq \frac{2}{5} |  \tilde V(s) - V_\ast(s) | .
\end{align}
Thus, by the triangle inequality,
\be \label{est:global-lb}
\frac{3}{5}  |  \tilde V(s) - V_\ast(s) | \leq | \tilde V(\tau) - V_\ast(\tau)| \leq \frac{7}{5}  |  \tilde V(s) - V_\ast(s) | .
\ee
Choosing $\tau = t$, we have
\be
 \frac{5}{7}  | v - V_\ast(t)| \leq |  \tilde V(s) - V_\ast(s) |  \leq  \frac{5}{3}  | v - V_\ast(t)| .
\ee
Hence we may rewrite the global bounds \eqref{est:global-lb} in terms of the value at time $t$: for all $\tau \in [t-\delta, t]$,
\be \label{est:global-bound-t}
\frac{3}{7}  | v - V_\ast(t)|  \leq | \tilde V(\tau) - V_\ast(\tau)| \leq \frac{7}{3}  |  \tilde V(s) - V_\ast(s) | .
\ee
Moreover, by \eqref{est:global-lb} once again, 
\be
Q(t,\delta) \leq \frac{1}{5} |  \tilde V(s) - V_\ast(s) |  \leq \frac{1}{3}  | v - V_\ast(t)| .
\ee
Then, by \eqref{est:lower-dynamic},
\be \label{est:lower-dynamic-t}
| \tilde X_\alpha (\tau) - X_\ast(\tau) | \geq \frac{2}{21} |\tau - \tau_\ast|  | v - V_\ast(t)| .
\ee

Similarly, since
\be
|\tilde V(s) - v| \leq Q(t,\delta) \leq \frac{1}{5} |\tilde V(s)|,
\ee
then 
\be
|v| \geq  |\tilde V(s)| - |\tilde V(s) - v| \geq \frac{4}{5}  |\tilde V(s)| .
\ee
Since $\Lambda$ is a radially increasing function,
\be
\Lambda(|\tilde V(s)|) \leq \Lambda \left ( \frac{5}{4} v \right ) \leq \left ( \frac{5}{4} \right )^k \Lambda (v) .
\ee
Hence, by \eqref{est:lower-static},
\be \label{est:lower-static-t}
|\tilde X_\alpha(s) - X_\ast(s)| > C_k \frac{R}{\Lambda \left (v \right )} .
\ee

By combining \eqref{est:lower-dynamic-t} and \eqref{est:lower-static-t}, we find that, for some $\tau_\ast \in [t-\delta, t]$
\be
|\tilde X_\alpha(s) - X_\ast(s)|^{-2} \leq C_k \left ( R^{-2} \Lambda(v)^2 \wedge |s - \tau_\ast|^{-2} |v - V_\ast(t)|^{-2} \right ).
\ee
Therefore
\be
I_U \leq C_k \int_{\left [ - \frac{1}{2}, \frac{1}{2} \right )^3} \int_{\br^3} f(t, x, v) |A(x,v)| \sup_{\tau_\ast} \int_{t-\delta}^t \left ( R^{-2} \Lambda(v)^2 \wedge |s - \tau_\ast|^{-2} |v - V_\ast(t)|^{-2} \right ) \di s \di x \di v .
\ee
We calculate
\begin{align}
 \sup_{\tau_\ast} \int_{t-\delta}^t \left ( R^{-2} \Lambda(v)^2 \wedge |s - \tau_\ast|^{-2} |v - V_\ast(t)|^{-2} \right ) \di s & \leq 2 \int_0^\delta \left ( R^{-2} \Lambda(v)^2 \wedge |\theta|^{-2} |v - V_\ast(t)|^{-2} \right )  \di \theta \\
 & \leq C \frac{\Lambda(v)}{R |v - V_\ast(t)|}.
\end{align}

Finally, by \cite[Lemma 3]{Batt-Rein} and \eqref{est:global-bound-t},
\begin{align}
|A(x,v)| & \leq C \left ( 1 + \int_{t-\delta}^t |\tilde V(\tau) - V_\ast (\tau) | \di t \right ) \\
& \leq C \left ( 1 + \delta | v - V_\ast (t) | \right ) .
\end{align}
Therefore
\be
I_U \leq C_k R^{-1}  \int_{\left [ - \frac{1}{2}, \frac{1}{2} \right )^3} \int_{\br^3} f(t, x, v) \Lambda(v) \left ( \delta  +  | v - V_\ast (t) |^{-1}  \right ) \di x \di v .
\ee
By \eqref{est:global-lb}, $ | v - V_\ast (t) |^{-1} \leq \frac{1}{3} Q(t,\delta)^{-1}$, and thus by \eqref{int:Lambda} we find that
\begin{align}
I_U & \leq C_k R^{-1} \left ( \delta  +  Q(t,\delta)^{-1}  \right ) \int_{\left [ - \frac{1}{2}, \frac{1}{2} \right )^3} \int_{\br^3} f(t, x, v) \Lambda(v)  \di x \di v \\
& \leq C_k \delta \frac{1}{R} (1 + (\delta Q(t,\delta))^{-1} )M_k(t)^{\frac{1}{k-2}} .
\end{align}

Summing over $\Omega_G, \Omega_B, \Omega_U$ gives
\begin{multline}
	\int_{t-\delta}^t \int_{\bt^d} \int_{\br^d} \frac{f(s,x,v)}{|x - X_\ast(s)|^2} \di x \di v \di s \leq
C \delta Q(t,\delta)^{4/3} 
+ C \delta R
+ C \delta \frac{1}{R} (1 + (\delta Q(t,\delta))^{-1} )M_k(t)^{\frac{1}{k-2}} .
\end{multline}
The optimal choice of $R$ is $R = (1 + (\delta Q(t,\delta))^{-1} )^{1/2} M_k(t)^{\frac{1}{2(k-2)}}$, giving the estimate
\be
\int_{t-\delta}^t \int_{\bt^d} \int_{\br^d} \frac{f(s,x,v)}{|x - X_\ast(s)|^2} \di x \di v \di s \leq C \delta \left (Q(t,\delta)^{4/3} + \left (1 + (\delta Q(t,\delta))^{-1} \right )^{1/2}M_k(t)^{\frac{1}{2(k-2)}}  \right ) .
\ee
Substituting this into inequality \eqref{est:Q-prelim} gives 
\be
Q(t,\delta) \leq C \e^{-2} \delta \left (Q(t,\delta)^{4/3} + (\delta Q(t,\delta))^{-1/2}M_k(t)^{\frac{1}{2(k-2)}} + (M_k(t)^{\frac{1}{2(k-2)}} + \e^{-5} (\e^{-3} \wedge M_k^{\frac{1}{2(k-2)}}) \right ) .
\ee
Hence
\be
Q(t,\delta)^{3/2} \leq C \e^{-2} \delta^{1/2} \left ((\delta Q(t,\delta))^{1/2} \Big( Q(t,\delta)^{4/3} + M_k(t)^{\frac{1}{2(k-2)}} +\e^{-5} (\e^{-3} \wedge M_k^{\frac{1}{2(k-2)}} \Big) + M_k(t)^{\frac{1}{2(k-2)}} \right ) ;
\ee
this completes the proof.
\end{proof}

The relation \eqref{est:Q-Q} can then be resolved so as to obtain an estimate on $Q(t,t)$, by using an extension of the method of \cite[Proposition 3.3]{ChenChen}. We will explain the argument in detail in order to show how to keep proper track of the dependence on $\e$ and handle the extra term $\e^{-7} (\e^{-3} \wedge M_k^{\frac{1}{2(k-2)}}) $ arising from the additional part $\widehat E$ of the electric field.

\begin{lemma} \label{lem:Qsmalld}
There exists $\e_\ast > 0$ such that the following holds for all $\e \leq \e_\ast$.
For all $t \geq 0$ there exists $\delta_\ast(t)$ such that for all $\delta \leq \delta_\ast(t)$
\be \label{est:Q-smalld-statement}
Q(t,\delta) \leq C \e^{-4/3} \delta^{1/3}  M_k(t)^{\frac{1}{3(k-2)}} .
\ee
Explicitly, we may take
\be
\delta_\ast(t) : = C \e M_k(t)^{\frac{1}{2(k-2)}} (M_k(t)^{\frac{1}{2(k-2)}} + \e^{-5} (\e^{-3} \wedge M_k^{\frac{1}{2(k-2)}}) )^{-3/2} 
\ee
for some $C>0$ independent of $\epsilon$ and $t$. 
\end{lemma}
\begin{proof}
The idea of the proof is the following:
since $\lim_{\delta \to 0}Q(t,\delta) = 0$, for sufficiently small $\delta \in (0,t]$ (we will determine precisely how small later in the proof) we have
\be \label{hyp:smalld}
(\delta Q(t,\delta))^{1/2} \left ( Q(t,\delta)^{4/3} + M_k(t)^{\frac{1}{2(k-2)}} + \e^{-5} (\e^{-3} \wedge M_k^{\frac{1}{2(k-2)}}) \right ) \leq 2 M_k(t)^{\frac{1}{2(k-2)}}.
\ee
As long as this holds, \eqref{est:Q-Q} will imply an estimate of the form
\be
Q(t,\delta)^{3/2} \leq C \e^{-2} \delta^{1/2}  M_k(t)^{\frac{1}{2(k-2)}} 
\ee
and we may therefore conclude that
\be \label{est:Q-smalld}
Q(t,\delta) \leq C \e^{-4/3} \delta^{1/3}  M_k(t)^{\frac{1}{3(k-2)}} .
\ee

We now wish to explicitly quantify how small $\delta$ must be in order for \eqref{hyp:smalld}, and thus \eqref{est:Q-smalld}, to hold. That is, we will find $\delta_\ast(t)$ such that \eqref{hyp:smalld} holds for all $\delta \in (0, \delta_\ast(t)]$. First, let
\begin{multline}
 \label{def:bard}
\bar \delta(t) : = \sup \Bigg \{ \delta \in (0,t] :  \delta^{1/2} Q(t,\delta)^{11/6} \leq M_k(t)^{\frac{1}{2(k-2)}}, \\
(\delta Q(t,\delta))^{1/2} ( M_k(t)^{\frac{1}{2(k-2)}} + \e^{-5} (\e^{-3} \wedge M_k^{\frac{1}{2(k-2)}}) ) \leq M_k(t)^{\frac{1}{2(k-2)}} \Bigg \} .
\end{multline}
Then \eqref{hyp:smalld} holds at least for all $\delta \in (0, \bar \delta(t)]$. 

We next seek a lower bound on $\bar \delta(t)$. If $\bar \delta(t) = t$, then there is nothing to prove. Otherwise, one of the two inequalities defining the supremum \eqref{def:bard} is attained when $\delta = \bar \delta(t)$. We consider each of the cases separately. For ease of reading we shorten $\bar \delta(t)$ to $\bar \delta$ from now on.

\noindent \textbf{Case 1: $(\bar \delta Q(t,\bar \delta))^{1/2} ( M_k(t)^{\frac{1}{2(k-2)}} + \e^{-5} (\e^{-3} \wedge M_k^{\frac{1}{2(k-2)}} )) = M_k(t)^{\frac{1}{2(k-2)}}$.}  
Hence
\be
M_k(t)^{\frac{1}{2(k-2)}}  ( M_k(t)^{\frac{1}{2(k-2)}} + \e^{-5} (\e^{-3} \wedge M_k^{\frac{1}{2(k-2)}} ))^{-1}  = (\bar \delta Q(t,\bar \delta))^{1/2} \leq C \e^{-2/3} {\bar \delta}^{2/3}  M_k(t)^{\frac{1}{6(k-2)}} .
\ee
After rearranging, we obtain
\be \label{est:12}
\bar \delta \geq C \e M_k(t)^{\frac{1}{2(k-2)}} (M_k(t)^{\frac{1}{2(k-2)}} + \e^{-5} (\e^{-3} \wedge M_k^{\frac{1}{2(k-2)}}))^{-3/2} 
\ee
which is the desired lower bound on $\bar \delta$.

\noindent \textbf{Case 2: $\bar \delta^{1/2} Q(t, \bar \delta)^{11/6} = M_k(t)^{\frac{1}{2(k-2)}}$.} We will show that this case is excluded for all sufficiently small $\e$. 

Since \eqref{hyp:smalld} holds for $\delta = \bar \delta$, by \eqref{est:Q-smalld}
\be
M_k(t)^{\frac{1}{(k-2)}} = \bar \delta Q(t,\bar \delta)^{11/3}  \leq C \e^{-44/9}  M_k(t)^{\frac{11}{9(k-2)}} {\bar {\delta}}^{20/9}.
\ee
We rearrange this to find that
$C \e^{11/5} M_k(t)^{-\frac{1}{10(k-2)}} \leq \bar \delta$,
and thus 
\be \label{est:Q-case2-upper}
Q(t,\bar \delta) = (\bar \delta^{-1/2} M_k(t)^{\frac{1}{2(k-2)}})^{6/11} \leq C \e^{-3/5} M_k(t)^{\frac{3}{10(k-2)}}.
\ee

At the same time,  
from the second criterion defining $\bar \delta$ \eqref{def:bard} we have
\be
M_k(t)^{\frac{1}{2(k-2)}} + \e^{-5} (\e^{-3} \wedge M_k(t)^{\frac{1}{2(k-2)}}) \leq Q(t, \bar \delta)^{-1/2} \bar{\delta}^{-1/2} M_k(t)^{\frac{1}{2(k-2)}}
\ee
and hence, substituting $ \bar \delta^{- 1/2} M_k(t)^{\frac{1}{2(k-2)}} = Q(t, \bar \delta)^{11/6}$, we obtain
\be \label{est:Q-case2-lower}
M_k(t)^{\frac{1}{2(k-2)}} + \e^{-5} (\e^{-3} \wedge M_k(t)^{\frac{1}{2(k-2)}}) \leq Q(t, \bar \delta)^{4/3}.
\ee 
We combine this with \eqref{est:Q-case2-upper} to obtain
\be \label{est:Q-case2-UL}
M_k(t)^{\frac{1}{2(k-2)}} + \e^{-5} (\e^{-3} \wedge M_k(t)^{\frac{1}{2(k-2)}}) \leq Q(t, \bar \delta)^{4/3} \leq C \e^{-4/5} M_k(t)^{\frac{2}{5(k-2)}} .
\ee 

If $M_k(t)^{\frac{1}{2(k-2)}} \leq \e^{-3}$, then
\be
\e^{-5} M_k^{\frac{1}{2(k-2)}} \leq Q(t, \bar \delta)^{4/3} \leq C \e^{-4/5} M_k(t)^{\frac{2}{5(k-2)}}.
\ee 
Rearranging gives $\e^{-21}  \leq C M_k(t)^{-\frac{1}{2(k-2)}}$. Since $f(t, \cdot , \cdot)$ has unit mass for all $t$, from definition \eqref{def:Mk} we have $M_k(t) \geq 1$, and hence 
\be \label{est:eps-LB1}
\e^{-21}  \leq C
\ee
must hold for some universal constant $C>0$.

Otherwise, we have $M_k^{\frac{1}{2(k-2)}} > \e^{-3}$, whence
$\e^{-8} + M_k(t)^{\frac{1}{2(k-2)}} \leq Q(t, \bar \delta)^{4/3}$. Hence
\be
\e^{-8/5} M_k(t)^{\frac{2}{5(k-2)}} = (\e^{-8})^{1/5} (M_k(t)^{\frac{1}{2(k-2)}} )^{4/5} \leq Q(t, \delta)^{4/3} \leq C \e^{-4/5} M_k(t)^{\frac{2}{5(k-2)}} ,
\ee
from which it follows that
\be \label{est:eps-LB2}
\e^{-4/5} \leq C,
\ee
for some universal constant $C >0$. 

We may choose $\e_\ast > 0$ depending only on universal constants such that \eqref{est:eps-LB1}-\eqref{est:eps-LB2} are both contradicted for all $\e \in (0, \e_\ast)$. Hence Case 2 cannot occur for this range of $\epsilon$. It follows that, for all $\e \in (0, \e_\ast)$, the lower bound for $\bar \delta$ \eqref{est:12} holds.

Finally, we define
\be
\delta_\ast(t) : = C \e M_k(t)^{\frac{1}{2(k-2)}} (M_k(t)^{\frac{1}{2(k-2)}} + \e^{-5} (\e^{-3} \wedge M_k^{\frac{1}{2(k-2)}}) )^{-3/2} ,
\ee
where $C>0$ is the constant in \eqref{est:12}. We have shown that, for all $\e \leq \e_\ast$, $\bar \delta(t) \geq \delta_\ast(t)$. Hence, since \eqref{hyp:smalld} holds for all $\delta \in (0, \bar \delta(t)]$, it certainly holds for all $\delta \in (0, \delta_\ast(t)] \subset (0, \bar \delta(t)]$. It follows that \eqref{est:Q-smalld-statement} holds for all $\delta \in (0, \delta_\ast(t)]$, as required.
\end{proof}

Next, we wish to prove an estimate for increments $Q(t, t-s)$ where $t-s > \delta_\ast(t)$. To do this, we use the approach from Chen and Chen \cite[Proposition 3.3]{ChenChen}, in which the interval $[s,t]$ is subdivided into subintervals, each of which is small enough that the estimate of Lemma~\ref{lem:Qsmalld} may be applied.
We summarise this part of their estimates, quantified in $\e$, in the following lemma.

\begin{lemma} \label{lem:general-subdivision}
	Let $0\leq a < b$. Then
	\be 
	Q(b, b-a) \leq C \e^{-4/3} \left( b-a
	 \right) \Delta(a,b)^{-2/3}  M_k(b)^{\frac{1}{3(k-2)}} ,
	\ee
where
\be 
\Delta(a,b) : = \inf_{s\in[a,b]} \delta_\ast(s) \wedge (b-a) .
\ee 
\end{lemma}
\begin{proof}

Split the time interval $[a, b]$ into subintervals of length no more than $\Delta(a,b)$ (see Figure~\ref{fig:interval-split}): let $n = n(a,b) : = \lfloor \frac{b-a}{\Delta(a,b)} \rfloor$, such that 
\be \label{def:interval-split}
[a, b] = [a, b - n \Delta(a,b)) \cup \bigcup_{j=1}^{n(a,b)} (b - j\Delta(a,b), b - (j-1)\Delta(a,b)].
\ee
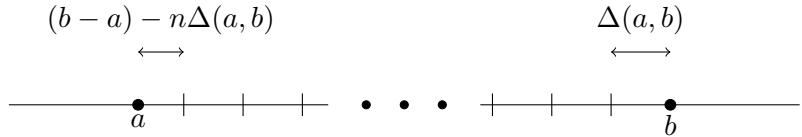
\begin{figure}[h] 
\centering
\begin{tikzpicture}
    \draw (-5.2,0) -- (-1,0);
    \draw (1,0) -- (5.2,0);
    
    \filldraw (-3.5,0) circle (2pt) node[below] {$a$};
    
    \draw[<->] (-3.5,0.7) -- (-2.9,0.7);
    \node[above] at (-3.2,0.8) {$(b-a)-n\Delta(a,b)$};
    
    \foreach \x in {-2.9,-2.12,-1.34} {
    \draw (\x,-0.15) -- (\x,0.15);
    }

    \filldraw (-0.5,0) circle (1.5pt) node[below] {};
    \filldraw (0,0) circle (1.5pt) node[below] {};
    \filldraw (0.5,0) circle (1.5pt) node[below] {};

    \foreach \x in {1.16, 1.94,2.72} {
    \draw (\x,0.15) -- (\x,-0.15);
    }    

    \draw[<->] (2.72,0.7) -- (3.5,0.7);
    \node[above] at (3.11,0.8) {$\Delta(a,b)$};
    
    \filldraw (3.5,0) circle (2pt) node[below] {$b$};
\end{tikzpicture}
\caption{ Decomposition of the interval $[a,b]$ into subintervals of length no more than $\Delta(a,b)$.}
\label{fig:interval-split}
\end{figure}
Then, by splitting the integral defining $Q(b,b-a)$ according to the regions \eqref{def:interval-split}, we find that
\be
Q(b, b-a) \leq Q(b - n \Delta(a,b), b - n \Delta(a,b)) + \sum_{j=1}^{n} Q( b - (j-1) \Delta(a,b), \Delta(a,b)) .
\ee
Now estimate each summand using \eqref{est:Q-smalld}:
\be
Q(b, b-a) \leq C \e^{-4/3} \Delta(a,b)^{1/3} \sum_{j=0}^{n} M_k(b - j \Delta(a,b) )^{\frac{1}{3(k-2)}} .
\ee
Since $M_k(s)$ is a non-decreasing function of $s$, $M_k(b - j \Delta(a,b)) \leq M_k(b)$ for all $j$ and hence
\be 
Q(b, b-a) \leq C (n(a,b) + 1) \e^{-4/3} \Delta(a,b)^{1/3} M_k(b)^{\frac{1}{3(k-2)}} .
\ee
Since $n(a,b) \leq (b-a)\Delta(a,b)^{-1}$, this completes the proof.
\end{proof}

Our next step is to apply the previous result in order to estimate $Q(t,t)$. To do so we need to estimate the infimum of $\delta_\ast$.
We begin by writing 
\be
\delta_\ast(s) = h_\e(M_k(s)^{\frac{1}{2(k-2)}})
\ee
where
\be \label{def:heps}
h_\e(z) : = \begin{cases}
 C \e (1 + \e ^{-5})^{-3/2}z^{-1/2}	& z \leq \e ^{-3} \\
 C \e z(z + \e ^{-8})^{-3/2} & z > \e ^{-3} . 
 \end{cases}
\ee

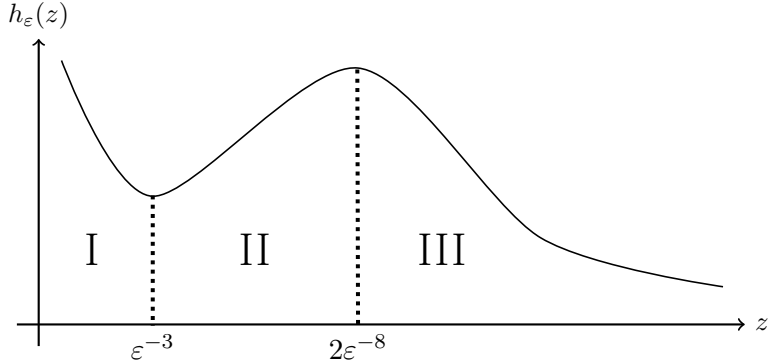
\begin{figure}[h] 
\centering
\begin{tikzpicture}
    \datavisualization [school book axes, visualize as smooth line, x axis={label=$z$, ticks={none}}, y axis={label=$h_\e(z)$, ticks={none}}, style={thick}]
    data {
      x, y
      0.3, 3.5
      1.5, 1.7
      4.2, 3.4
      6.6, 1.15
      9, 0.5
    };

    \node at (1.5, 0) [below] {$\e^{-3}$};
    \node at (4.2, 0) [below] {$2\e^{-8}$};

    \node at (0.7, 1) {\LARGE{I}};
    \node at (2.84, 1) {\LARGE{II}};
    \node at (5.3, 1) {\LARGE{III}};

    \node at (1.5, 1.98) (x1) [below] {};
    \node at (4.19, 3.65) (x2) [below] {};

    \draw[dotted, line width=1.5pt] (x1) -- (1.5, -0.02);
    \draw[dotted, line width=1.5pt] (x2) -- (4.2, -0.05);
\end{tikzpicture}
\caption{The function $h_\e(z)$, showing its monotonicity properties for $z$ in different regions.}
\label{fig:heps}
\end{figure}

Observe (see Figure~\ref{fig:heps}) that $h_\e$ is a decreasing function for $z \in [0, \e ^{-3} \cup (2\e ^{-8}, +\infty)$ and increasing for $z \in (\e ^{-3}, 2\e ^{-8})$.
Since $M_k(s)$ is a non-decreasing function of $s$, we may identify corresponding time intervals of monotonicity for $\delta_\ast(s)$. Let
\begin{align}
	t_I &:= \inf \{ s \geq 0 : M_k(s)^{\frac{1}{2(k-2)}} > \e^{-3} \} \\
	t_{II} &:= \inf \{ s \geq 0 : M_k(s)^{\frac{1}{2(k-2)}} > 2 \e^{-8} \} .
\end{align}
Then $\delta_\ast(s)$ is a non-increasing function of $s$ for $s \in [0, t_I) \cup (t_{II}, + \infty)$ and a non-decreasing function of $s$ for $s \in (t_I, t_{II})$.

We therefore see a difference between the non-increasing regions $[0, t_I)$ and $[t_{II}, + \infty)$ and the non-decreasing region $[t_I, t_{II})$: in the non-increasing regions, the infimum of $\delta_\ast$ of any subinterval is attained at the right hand endpoint of the interval, whereas in the non-decreasing region the infimum is attained at the left hand endpoint.
Consequently, we will use different methods to estimate $Q$ depending on whether we consider a non-increasing or non-decreasing region.

\begin{remark}
At this point it is instructive to compare the corresponding function in the electron Vlasov--Poisson case, which is 
\be
\delta_\ast(s) = C \e M_k(s)^{-\frac{1}{4(k-2)}}.
\ee
This is a non-increasing function of $s$ for \emph{all} $s$. In the ion case, we will follow the argument of \cite{ChenChen} in the non-increasing regions (Lemma~\ref{lem:I-III}), as this is suited to the non-increasing case. In the non-decreasing region, however, we will develop a new argument (see Lemma~\ref{lem:II}).
\end{remark}

In Regions I and III we follow the method of \cite{ChenChen} and deduce the following estimate by a direct application of Lemma~\ref{lem:general-subdivision}.

\begin{lemma}[Estimates on $Q$, Regions I and III]
 \label{lem:I-III}
For all $t \leq t_I$,
	\be
	Q(t , t ) \leq C \e^{-7}  t  M_k(t )^{\frac{1}{2(k-2)}} .
	\ee 
For all $t > t_{II}$,
	\be
	Q(t, t - t_{II}) \leq C \e^{-2} \left( t - t_{II}
	 \right) M_k(t)^{\frac{1}{2(k-2)}} .
	\ee 

\end{lemma}
\begin{proof}
Note that $\delta(s)$ is non-increasing for all $s \in [0, t_I)$ and all $s > t_{II}$. Thus, if $t \leq t_I$, the infimal value over $[0,t]$ is realised at the upper endpoint of the interval, $s = t$, and
\be \label{est:Delta0t}
	\Delta(0 , t ) = \delta_\ast(t ) \wedge t 
 = h_\e(M_k(t )^{\frac{1}{2(k-2)}}) \wedge t .
\ee
For $t > t_{II}$, an identical argument shows that
\be \label{est:DeltatII}
\Delta(t_{II}, t ) = \delta_\ast(t) \wedge (t - t_{II}).
\ee

	By definition of $t_I$, $M_k(s)^{\frac{1}{2(k-2)}} \leq \e^{-3}$ for all $s \leq t_I$; in particular, when $t \leq t_I$ we may substitute the definition of $h_\e$ \eqref{def:heps} for $z\leq\e^{-3}$ to rewrite \eqref{est:Delta0t} as
	\be 
	\Delta(0, t) = C \e (1 + \e ^{-5})^{-3/2} M_k(t)^{- \frac{1}{4(k-2)}} \wedge t .
	\ee 
Then, by Lemma~\ref{lem:general-subdivision},
	\be
	Q(t, t) \leq C \e^{-2} (1 + \e ^{-5}) t
	  M_k(t )^{\frac{1}{2(k-2)}}  \leq C \e^{-7}  t
	  M_k(t )^{\frac{1}{2(k-2)}}  .
	\ee 
	
For $t > t_{II}$, by definition of $t_{II}$ we have $M_k(t)^{\frac{1}{2(k-2)}} \geq 2 \e^{-8}$. Then, substituting the definition of $h_\e$ \eqref{def:heps} for $z \geq 2\e^{-8}$ into \eqref{est:DeltatII} gives
	\begin{align}
	\Delta(t_{II}, t)  =  C \e M_k(t)^{\frac{1}{2(k-2)}}(M_k(t )^{\frac{1}{2(k-2)}} + \e ^{-8})^{-3/2} 
	 \geq C \e M_k(t )^{- \frac{1}{4(k-2)}} ,
	\end{align} 
and by Lemma~\ref{lem:general-subdivision},
	\be
	Q(t_{II}, t - t_{II}) \leq C \e^{-2} \left( t - t_{II}
	 \right) M_k(t)^{\frac{1}{2(k-2)}} .
	\ee 
\end{proof}

In Region II, the argument in Lemma~\ref{lem:I-III} would give
\begin{align}
\Delta(t_I, t ) = \delta_\ast(t_I) = h_\e(M_k(t_I)^{\frac{1}{2(k-2)}}) = h(\e^{-3}) = C (1 + \e ^{5})^{-3/2}\e^{10}, \quad t \in [t_I, t_{II}],
\end{align} 
and thus we would find the estimate
\be 
Q(t , t  - t_I) \leq C \e^{-8}  (t  - t_I) M_k(t )^{\frac{1}{3(k-2)}} .
\ee 
At $t=t_{II}$ this implies that
\be 
Q(t_{II} , t_{II} - t_I) \leq C \e^{-40/3}  ( t_{II} - t_I) .
\ee 
In the following lemma, we show that we may in fact obtain an improved estimate of order $\e^{-10}$, by using a further subdivision of the interval $[t_I, t_{II}]$. This is a key difference in our proof from the method of Chen and Chen \cite{ChenChen} for the electron model.

\begin{lemma}[Region II] \label{lem:II}
Let $t \in (t_I, t_{II}]$. Then
	\be 
	Q(t , t - t_{I}) \leq C \e^{-10} (t  - t_I) + C \e^3 M_k(t)^{\frac{1}{2(k-2)}} .
	\ee 
\end{lemma}

\begin{remark}
Recall that $M_k(t)^{\frac{1}{2(k-2)}}  \leq \e^{-8}$ for $t \leq t_{II}$, so that the second term is of lower order in $\e^{-1}$:
	\be 
	Q(t , t - t_{I}) \leq C \e^{-10} (t  - t_I) + C \e^{-5} \leq C \e^{-10} (1 + t  - t_I) .
	\ee 
\end{remark}

\begin{proof}

Consider the following subdivision of the interval $[t_{I}, t ]$: first, let $\tau_0 := t_{I}$; then, for $j \geq 1$, let 
\be \label{def:tau_j}
\tau_j := \inf \{ s > \tau_{j-1} : M_k(s)^{\frac{1}{2(k-2)}} > 2 M_k(\tau_{j-1})^{\frac{1}{2(k-2)}} \} .
\ee 
Since $M_k(s)$ is a continuous function of $s$, we have $M_k(\tau_j)^{\frac{1}{2(k-2)}} = 2 M_k(\tau_{j-1})^{\frac{1}{2(k-2)}} = 2^j M_k(\tau_0)^{\frac{1}{2(k-2)}}$ for all $j$.
Then let
\be 
J(t) : = \inf \{ j \geq 0 : \tau_j \geq t \} ;
\ee 
note that $J(t)$ is finite with
\be
J(t) \leq 1 + \frac{\log{M_k(t)^{\frac{1}{2(k-2)}}} - \log{M_k(t_I)^{\frac{1}{2(k-2)}}}}{\log 2} \leq 1 + \frac{\log{(\e^3 M_k(t)^{\frac{1}{2(k-2)}}} ) }{\log 2} ,
\ee
since
\be
2^{J(t) -1} \e^{-3} = 2^{J(t) -1} M_k( t_{I} )^{\frac{1}{2(k-2)}} = M_k( \tau_{J(t)-1} )^{\frac{1}{2(k-2)}} \leq M_k(t)^{\frac{1}{2(k-2)}} .
\ee
For convenience we then redefine $\tau_{J(t)} = t $.
Thus $[t_I, t]$ may be written as the following (almost disjoint) union of intervals (see Figure~\ref{fig:tau}):
\be 
[t_{I}, t ] = \bigcup_{j=1}^{J(t)} [\tau_{j-1}, \tau_j].
\ee 

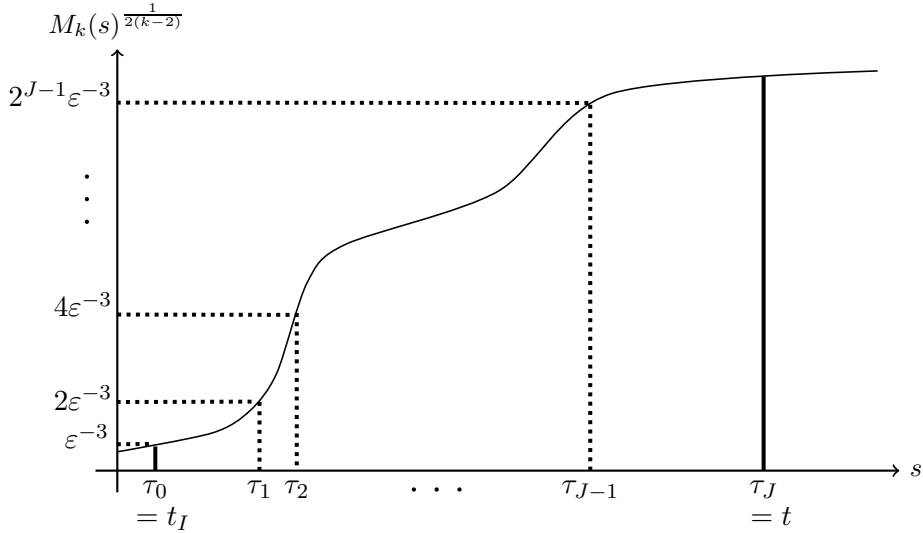
\begin{figure}[h] 
\centering
\begin{tikzpicture}
    \datavisualization [school book axes, visualize as smooth line, x axis={label=$s$, ticks={none}}, y axis={label=$M_k(s)^{\frac{1}{2(k-2)}}$, ticks={none}}, style={thick}]
    data {
      x, y
      0, 0.25
      1.25, 0.5
      1.75, 0.8
      2.1, 1.3
      2.5, 2.5
      3, 3
      5, 3.7
      6.5, 5
      10, 5.3
    };
    
    \node at (0.5, 0) [below] {$\tau_0$};
    \node at (1.87, 0) [below] {$\tau_1$};
    \node at (2.36, 0) [below] {$\tau_2$};
    \node at (6.22, 0) [below] {$\tau_{J-1}$};
    \node at (8.5, 0) [below] {$\tau_J$};
    \node at (0.6, -0.35) [below] {$=t_I$};
    \node at (8.6, -0.35) [below] {$= t$};
    
        \node at (-0.4, 0.75) [below] {$\e^{-3}$};
    \node at (-0.45, 1.26) [below] {$2 \e^{-3}$};
    \node at (-0.45, 2.5) [below] {$4 \e^{-3}$};
    \node at (-0.75, 5.3) [below] {$2^{J-1} \e^{-3}$};

    \node at (0.5, 0.59) (x11) [below] {};
    \node at (0.55, 0.49) (x12) [below] {};
    \node at (1.87, 1.18) (x21) [below] {};
    \node at (1.9, 1.05) (x22) [below] {};
    \node at (2.36, 2.36) (x31) [below] {};
    \node at (2.4, 2.2) (x32) [below] {};
    \node at (2.3, 1.5) [below] {};
    \node at (6.22, 5.12) (x41) [below] {};
    \node at (6.33, 5.01) (x42) [below] {};
    \node at (8.5, 5.5) (x5) [below] {};
    
    \draw[line width=1.5pt] (x11) -- (0.5, 0);
    \draw[dotted, line width=1.5pt] (x12) -- (-0.03, 0.35);
    \draw[dotted, line width=1.5pt] (x21) -- (1.87, 0);
    \draw[dotted, line width=1.5pt] (x22) -- (-0.03, 0.91);
    \draw[dotted, line width=1.5pt] (x31) -- (2.36,0);
    \draw[dotted, line width=1.5pt] (x32) -- (0, 2.07);
    \draw[dotted, line width=1.5pt] (x41) -- (6.22,0);
    \draw[dotted, line width=1.5pt] (x42) -- (0,4.88);
    \draw[line width=1.5pt] (x5) -- (8.5, 0);
    
    \filldraw (3.9,-0.25) circle (0.6pt) node[below] {};
    \filldraw (4.2,-0.25) circle (0.6pt) node[below] {};
    \filldraw (4.5,-0.25) circle (0.6pt) node[below] {};
    
    \filldraw (-0.4,3.3) circle (0.6pt) node[below] {};
    \filldraw (-0.4,3.6) circle (0.6pt) node[below] {};
    \filldraw (-0.4,3.9) circle (0.6pt) node[below] {};
\end{tikzpicture}
\caption{Construction of the times $(\tau_i)_{i=0}^J$.}
\label{fig:tau}
\end{figure}
We will now apply Lemma~\ref{lem:general-subdivision} on each subinterval $[\tau_{j-1}, \tau_j]$.
Since $\delta_\ast (s)$ is increasing for $s\in (t_{I}, t_{II})$, its minimal value is found at the left endpoint of any subinterval, i.e.
\begin{align}
\Delta(\tau_{j-1}, \tau_j ) &= \delta_\ast(\tau_{j-1}) \\
&= C \e M_k(\tau)^{\frac{1}{2(k-2)}} (M_k(\tau)^{\frac{1}{2(k-2)}} + \e^{-8} )^{-3/2} \wedge (\tau_j - \tau_{j-1}) \\
& \geq C \e^{13} M_k(\tau_{j-1})^{\frac{1}{2(k-2)}} \wedge (\tau_j - \tau_{j-1})
\end{align}
Then, by Lemma~\ref{lem:general-subdivision},
	\be \label{est:Q-M-cases}
	Q(\tau_j, \tau_j - \tau_{j-1}) \leq
	\begin{cases}
		C \e^{-10} (\tau_j - \tau_{j-1}) \left( \frac{M_k(\tau_j)}{M_k(\tau_{j-1})} \right )^{\frac{1}{3(k-2)}} & \text{if} \; \;\Delta(\tau_{j-1}, \tau_j ) = C \e^{13} M_k(\tau_{j-1})^{\frac{1}{2(k-2)}}, \; \text{or} \\
		C \e^{-4/3} (\tau_j - \tau_{j-1})^{1/3} M_k(\tau_j)^{\frac{1}{3(k-2)}}  & \text{if} \; \; \Delta(\tau_{j-1}, \tau_j ) = \tau_j - \tau_{j-1} .
	\end{cases}
	\ee
In the second case, $\tau_j - \tau_{j-1} \leq C \e^{13} M_k(\tau_{j-1})^{\frac{1}{2(k-2)}}$, and so
\be
 Q(\tau_j, \tau_j - \tau_{j-1}) \leq C \e^{3} M_k(\tau_{j-1})^{\frac{1}{2(k-2)}} \left( \frac{M_k(\tau_j)}{M_k(\tau_{j-1})} \right )^{\frac{1}{3(k-2)}}  .
\ee
Now recall that, by definition of the $\tau_j$ \eqref{def:tau_j}, $ \frac{M_k(\tau_j)}{M_k(\tau_{j-1})} = 2$. By summing the two cases of \eqref{est:Q-M-cases} we have
\be
	Q(\tau_j, \tau_j - \tau_{j-1}) \leq C 2^{\frac{1}{3(k-2)}} \e^{-10} (\tau_j - \tau_{j-1} ) + C 2^{\frac{1}{3(k-2)}}  \e^{3} M_k(\tau_{j-1})^{\frac{1}{2(k-2)}}.
\ee
Since $M_k(\tau_{j-1}) = 2^{j-1} M_k(\tau_0)$
for $j \leq J(t)-1$ and $M_k(\tau_{J(t)-1}) \leq M_k(t)$,
\be
Q(\tau_j, \tau_j - \tau_{j-1}) \leq \begin{cases}
C 2^{\frac{1}{3(k-2)}} \e^{-10} (\tau_j - \tau_{j-1} ) + C  \e^{3}  2^{\frac{1}{3(k-2)} + \frac{j-1}{2(k-2)}} M_k(\tau_0)^{\frac{1}{2(k-2)}} & j \leq J(t) -1, \\
 	C 2^{\frac{1}{3(k-2)}} \e^{-10} (\tau_j - \tau_{j-1} ) + C  \e^{3} M_k(t)^{\frac{1}{2(k-2)}} &  j = J(t) .
 \end{cases}
\ee
By summing over the subintervals, we obtain the estimate
\begin{align}
Q(t_I, t) & \leq \sum_{j=1}^{J(t)} Q(\tau_j, \tau_j - \tau_{j-1}) \\
& \leq C 2^{\frac{1}{3(k-2)}} \e^{-10} (t - t_I) + C \e^{3}  2^{\frac{1}{3(k-2)}} \frac{2^{\frac{J(t)-1}{2(k-2)}} - 1}{2^{\frac{1}{2(k-2)}}-1} M_k(t_I)^{\frac{1}{2(k-2)}} + C  \e^{3} M_k(t)^{\frac{1}{2(k-2)}}.
\end{align}
Then, using the fact that $2^{\frac{J(t)-1}{2(k-2)}} M_k(t_I)^{\frac{1}{2(k-2)}} = M_k(\tau_{J(t)-1})^{\frac{1}{2(k-2)}} \leq M_k(t)^{\frac{1}{2(k-2)}}$, we find that
\be
Q(t_I, t) \leq C 2^{\frac{1}{3(k-2)}} \left (  \e^{-10} (t - t_I) + C \e^{3} \frac{2^{\frac{1}{2(k-2)}}}{2^{\frac{1}{2(k-2)}}-1} M_k(t)^{\frac{1}{2(k-2)}} \right ) ,
\ee
which completes the proof.
\end{proof}

We combine the previous results to obtain an estimate on $Q(t,t)$.

\begin{lemma} \label{lem:Q-Mk}
For all $t \geq 0$,
\be
Q(t,t) \leq C \e^{-2} (1 + t) \left ( M_k(t)^{\frac{1}{2(k-2)}} \vee \e^{-8} \right ) .
\ee
\end{lemma}

\begin{proof}

In the case $t \leq t_I$, we simply apply Lemma~\ref{lem:I-III} to obtain
\be
	Q(t , t ) \leq C \e^{-7}  t  M_k(t )^{\frac{1}{2(k-2)}}  \leq C \e^{-10} t \leq C \e^{-2} t \left ( M_k(t)^{\frac{1}{2(k-2)}} \vee \e^{-8} \right ),
\ee
since $M_k(t )^{\frac{1}{2(k-2)}} \leq \e^{-3}$.

In the case $t \in (t_I, t_{II}]$,
\be
Q(t,t) \leq Q(t_I, t_I) + Q(t, t- t_I) .
\ee
We bound the first term using Lemma~\ref{lem:I-III} and the fact that $M_k(t_I )^{\frac{1}{2(k-2)}} = \e^{-3}$ by definition:
\be
Q(t_I, t_I) \leq C \e^{-10}  t_I .
\ee
We combine this with the bound on the second term given by Lemma~\ref{lem:II} to obtain
\begin{align}
Q(t,t) & \leq C \e^{-10}  t_I  + C \e^{-10} (t  - t_I) + C \e^3 M_k(t)^{\frac{1}{2(k-2)}} \\
& \leq C \e^{-10}  t + C \e^3 M_k(t)^{\frac{1}{2(k-2)}}  \label{est:QII} \\
& \leq C \e^{-10} (t + \e^5) ,
\end{align}
since $M_k(t)^{\frac{1}{2(k-2)}} \leq 2 \e^{-8}$. We conclude that
\be \label{est:QI-II}
Q(t,t) \leq C \e^{-10}(t + 1) \leq C \e^{-2} (1 + t) \left ( M_k(t)^{\frac{1}{2(k-2)}} \vee \e^{-8} \right ) , \qquad \text{for all} \; t \in [0,t_{II}] .
\ee
In the case $t > t_{II}$, we first write
\be
Q(t,t) \leq Q(t_{II}, t_{II}) + Q(t, t- t_{II}) .
\ee
By \eqref{est:QII} and Lemma~\ref{lem:I-III}, we have the estimate
\begin{align}
Q(t,t) & \leq C \e^{-10}  (1 + t_{II}) + C \e^{-2} \left( t - t_{II}
	 \right) M_k(t)^{\frac{1}{2(k-2)}} \\
	 & \leq C \e^{-2} \left ( \e^{-8} (1+ t_{II} ) +  \left( t - t_{II}
	 \right) M_k(t)^{\frac{1}{2(k-2)}} \right ) .
\end{align}
Then
\begin{align}
Q(t,t)	 & \leq C \e^{-2} \left ( M_k(t)^{\frac{1}{2(k-2)}} \vee \e^{-8} \right ) \left ( (1 +  t_{II} )  +  ( t - t_{II})  \right ) \\
	 & \leq C \e^{-2} (1 + t) \left ( M_k(t)^{\frac{1}{2(k-2)}} \vee \e^{-8} \right ) .
\end{align}
Combining this with \eqref{est:QI-II} yields the result.
\end{proof}

The next step is to resolve the relation between $M_k$ and $Q(t,t)$ so as to obtain an estimate on $Q(t,t)$ that depends solely on $t, \e,$ and the initial data. To do this, we will require 
the following estimate from \cite[Proposition 3.1]{ChenChen}, which allows the moments to be controlled in terms of $Q(t,t)$.

\begin{lemma} \label{lem:Mk-Q}
	There exists a constant $C$ depending on $k$, $M_2(0)$, $M_k(0) $ such that 
	\be
	M_k(t) \leq C (1 + Q(t,t)^{\max\{2, k-2\}})
	\ee 
\end{lemma}

Using this result, we may first obtain lower bounds on the time $t_{II}$.

\begin{lemma}
The time $t_{II}$ satisfies the lower bound
\be
\e^{-2 (4 \min \{ 4, k \} - 13)} \leq C_k(1 + t_{II}) .
\ee

\end{lemma}
\begin{proof}

By definition of $t_{II}$, $M_k(t_{II})^{\frac{1}{2(k-2)}} = 2 \e^{-8}$. Moreover, by Lemma~\ref{lem:Q-Mk}, $Q(t_{II}, t_{II}) \leq C \e^{-10} (1+ t_{II})$. We substitute these bounds into the estimate obtained from Lemma~\ref{lem:Mk-Q}:
\begin{align}
\e^{-16(k-2)} &= C M_k(t_{II}) \\
& \leq C (1 + Q(t_{II},t_{II})^{\max\{2, (k-2) \}}) \\
& \leq C \e^{-10 \max\{2, (k-2) \}} (1+ t_{II})^{\max\{2, (k-2) \}}.
\end{align}
Upon rearranging this inequality, we obtain
\be
\e^{-2 (4 \min \{ 4, k \} - 13)} \leq C_k(1 + t_{II}) .
\ee

\end{proof}

Thus if $k > 13/4$ then $t_{II} \to +\infty$ as $\e \to 0$. It follows that for $T_\ast > 0$ fixed there exists $\e_\ast > 0$ such that $T_\ast < t_{II}$ for all $\e < \e_\ast$, and thus the interval of interest $[0, T_\ast]$ is entirely contained within the regions I and II. We deduce directly from Lemma~\ref{lem:Q-Mk} that for all $t \in [0,T_\ast]$ and $\e < \e_\ast$,
\be \label{est:QI-II-final}
Q(t,t) \leq C \e^{-10} (t+1) .
\ee
It remains only to complete the estimate for the case $3 < k \leq 13/4$.

\begin{lemma} \label{lem:Q-final-3d}
Let $k \leq 13/4$. Then
\be
Q(t,t) \leq C \e^{-2 \cdot \frac{k-2}{k-3}} (t+1)^{\frac{k-2}{k-3}} .
\ee

\end{lemma}
\begin{proof}
If $t \leq t_{II}$, then the estimate \eqref{est:QI-II-final} holds. Otherwise, $t > t_{II}$ and so $M_k(t)^{\frac{1}{2(k-2)}} \geq 2 \e^{-8}$.
Then, Lemmas~\ref{lem:Q-Mk} and \ref{lem:Mk-Q} give
\be
Q(t,t) \leq C \e^{-2} (t+1)  M_k(t)^{\frac{1}{2(k-2)}} \leq C \e^{-2} (t+1) \, Q(t,t)^{\frac{1}{k-2}} ,
\ee
since $k \leq 13/4$ and so in particular $k < 4$ (we assume without loss of generality that $Q(t,t) > 1$ so as to absorb the additive constant). Thus
\be \label{est:Q-III-longtime}
Q(t,t) \leq C \e^{-2\frac{k-2}{k-3}} (1+t)^{\frac{k-2}{k-3}} .
\ee
We conclude that, for each $t \geq 0$, $Q(t,t)$ is bounded by the maximum of the two bounds \eqref{est:QI-II-final} and \eqref{est:Q-III-longtime}
\be
Q(t,t) \leq C \max\{ \e^{-10} (t+1), \e^{-2\frac{k-2}{k-3}} (1+t)^{\frac{k-2}{k-3}} \}.
\ee
Since $k \leq 13/4$, we have $2 \cdot \frac{k-2}{k-3} \geq 10$, and we conclude that
\be 
Q(t,t)\leq C \e^{-2 \cdot \frac{k-2}{k-3}} (1+t)^{\frac{k-2}{k-3}} .
\ee
\end{proof}

\section{Proofs of the Main Results}\label{sec:proof}
Thanks to all the results from the previous sections, we can now prove our main results.

\subsection{Proof of Theorem~\ref{thm:main}}

\subsubsection{Case $d=1$}

The one-dimensional case $d=1$ is a \emph{weak-strong} result, wherein $\{ f_{0,\e} \}_{\e \leq 1}$ may be measures,
and the `strong' solutions $\{ g_\e\}_{\e \leq 1}$ need only satisfy $\rho[g_\e] \in L^1 \left ( [0, T_\ast ) ; L^\infty(\bt) \right )$.

We first comment on the existence of solutions $f_\e$ with respective initial datum $f_{0,\e}$. By (H3), each $f_{0,\e}$ has a finite first moment with respect to velocity. Hence, by \cite[Theorem 1.1]{IHK1}, there exists at least one weak solution $f_\e$ of the VPME system \eqref{eq:vpme} with initial datum $f_{0,\e}$. In dimension $d=1$, we do not use Loeper's uniqueness theory. Rather, the argument relies on the weak-strong stability method in $W_1$, based on Hauray's estimate for the one-dimensional Vlasov--Poisson equation and its adaptation to the present VPME setting developed in Section~\ref{sec:stability_1}. This is the natural framework in the present setting, since the rough solution $f_\e$ may be measure-valued, while the comparison solution $g_\e$ only needs to satisfy $\rho[g_\e]\in L^1([0,T_*);L^\infty(\T))$. Therefore we fix any sequence $\{ f_\e \}_{\e \leq 1}$ of such weak solutions.

By the triangle inequality for $W_1$,
\be
W_1 \left ( f_\e, g \right ) \leq W_1 \left ( f_\e, g_\e \right ) + W_1 \left ( g_\e, g \right ) .
\ee
We apply Lemma~\ref{lem:stability-1d} to the first term to find that
\be
W_1 \left ( f_\e, g \right ) \leq  \e^{-2} \exp(\| \rho[g_\e] \|_{L^1\left ([0, t) ; L^\infty(\bt) \right )} + C \e^{-2}t ) W_1(f_{0,\e}, g_{0,\e}) + W_1 \left ( g_\e, g \right ) .
\ee

Hence, by \eqref{hyp:rate}, if
\be
W_1(f_{0,\e}, g_{0,\e}) \leq \exp(- C_1 \e^{-2}), \qquad C_1 > C T_\ast + C_0,
\ee
then
\be
\lim_{\e \to 0} \sup_{t \in [0,T_\ast]} W_1(f_\e, g) = 0 .
\ee

\subsubsection{Case $d=2,3$}

In the cases $d=2$ and $d=3$, we require both the rough solutions $\{ f_\e \}_{\e \leq \e_\ast}$ and `smooth' solutions $\{ g_\e \}_{\e \leq \e_\ast}$ to have spatial densities uniformly bounded in $L^1 \left ( [0, T_\ast ] ; L^\infty(\bt^d) \right )$. 
In comparison to the $d=1$ case we must therefore additionally obtain $L^1 \left ( [0, T_\ast ] ; L^\infty(\bt^d) \right )$ bounds on $\rho[f_\e]$.

First note that, by the triangle inequality for $W_1$
and the first inequality in Lemma~\ref{lemma:Wp},
\be
W_1(f_\e, g) \leq W_1(f_\e, g_\e) + W_1(g_\e, g) 
\leq \sqrt{2} W_2(f_\e, g_\e) + W_1(g_\e, g) .
\ee
The term $W_1(g_\e, g)$ will converge uniformly to zero as $\e$ tends to zero, by assumption \ref{hyp:regular-convergence}. 
It remains to show that
\be \label{eq:lim-fg}
\lim_{\e \to 0} \sup_{[0,T]} W_2(f_\e, g_\e) = 0 .
\ee
For this, we apply Proposition~\ref{prop:stability} to $f_\e$ and $g_\e$: if $W_2(f_{0,\e},g_{0,\e})\leq \sqrt{2 e^{-1}} \e$ and
\begin{equation} \label{est:W2-hyp-proof}
H \left [ \e^{-2}W_2(f_{0,\e}, g_{0,\e})^2 \right ] \geq \frac{C_d}{\e}\int_0^{T_\ast} B(s)\,ds+\sqrt{\left|\log\left(\frac{\e}e\right)\right|} ,
\end{equation}
where
\be
B(t) = \| \rho[f_\e](t) \|_{L^\infty} + \| \rho[g_\e](t) \|_{L^\infty}, 
\ee
and we recall that the function $H$ was defined in Equation \eqref{def:H}, then, for all $t \in [0,T_\ast ]$,
\be
W_2(f_\e(t),g_\e(t))^2 \leq 2 e^{-\left( H \left [ \e^{-2}W_2(f_{0,\e}, g_{0,\e})^2 \right ]  - \frac{C_d}{\e}\int_0^t B(s)\,ds\right)^2} \leq \frac{2}{e} \, \e .
\ee
By the second inequality in Lemma~\ref{lemma:Wp} and \ref{hyp:regular-moments}-\ref{hyp:moments},
\be \label{est:swap-W1-W2}
W_2(f_{0,\e},g_{0,\e})\leq 3(1 + 2 C_0)^{\frac{1}{(j_0 \vee k_0)-1}} W_1 (f_{0,\e},g_{0,\e})^{\frac{(j_0 \vee k_0)-2}{(j_0 \vee k_0)-1}} .
\ee
It is then clear from \eqref{hyp:rate} that $W_2(f_{0,\e},g_{0,\e})\leq \sqrt{2 e^{-1}} \e$ is satisfied for all $\e$ sufficiently small. To complete the proof it therefore suffices to show that \eqref{est:W2-hyp-proof} holds.

Now suppose that
\be \label{hyp:H4-final-proof}
W_2(f_{0,\e},g_{0,\e}) \leq \exp \left (- \widetilde{C}  \e^{-\zeta}(1 + \mathbbm{1}_{d=2} |\log \e|^2) \right ),
\ee
for an exponent $\zeta > 0$ and constant $\widetilde{C} >0$ to be determined. Note that, by estimate \eqref{est:swap-W1-W2}, \eqref{hyp:H4-final-proof} is implied by \eqref{hyp:rate} if $C$ is large enough in terms of $\widetilde{C}$, $C_0$ and $k_0$.

For $\e > 0$ sufficiently small, we have 
\be
\e^{-2}W_2(f_{0,\e}, g_{0,\e})^2 \leq \e^{-2} \exp \left (- 2 \widetilde{C} \e^{-\zeta}(1 + \mathbbm{1}_{d=2} |\log \e|^2) \right ) \leq  \exp \left (- \widetilde{C}  \e^{-\zeta}(1 + \mathbbm{1}_{d=2} |\log \e|^2) \right ).
\ee
We assume that $\widetilde C > \log \frac{e}{2}$, which ensures that the right hand side is less than $2 e^{-1}$ for all $\e \geq 1$.
We observed previously that the function $H$ is decreasing on the interval $(0, 2 e^{-1})$ and tends to $+\infty$ as its argument tends to zero. Thus
\be
H \left [ \e^{-2}W_2(f_{0,\e}, g_{0,\e})^2 \right ] \geq H \left [ \exp \left (- \widetilde{C} \e^{-\zeta}(1 + \mathbbm{1}_{d=2} |\log \e|^2) \right ) \right ]
\ee
(where we allow the left hand side formally to take the value $+\infty$ in the case $W_2(f_{0,\e}, g_{0,\e})^2 = 0$).

Now note that
\begin{align}
H(w)^2 & = - \log (- w \log \frac{1}{2} w) = - \log w - \log \left [ \log \left ( \frac{2}{w} \right ) \right ] \\
& = - \log w - \log 2 - \log \log \left ( \frac{2}{w} \right )^{1/2}\\
& \geq - \log w - \log 2 - \frac{1}{2} \log \left ( \frac{2}{w} \right ) \\
& \geq - \frac{1}{2} \log w - \frac{3}{2} \log 2
\end{align}
where in the third line we have used the inequality $\log y \leq y$. Hence
\be
H \left [ \exp \left (- \widetilde{C} \e^{-\zeta}(1 + \mathbbm{1}_{d=2} |\log \e|^2) \right ) \right ] \geq \frac{1}{\sqrt{2}} \left ( \widetilde{C}  \e^{-\zeta}(1 + \mathbbm{1}_{d=2} |\log \e|^2)  - 3 \log 2 \right )^{1/2}
\ee
We assume that $\widetilde C > 6 \log 2$, which ensures that 
\begin{align}
H \left [ \exp \left (-\widetilde{C} \e^{-\zeta}(1 + \mathbbm{1}_{d=2} |\log \e|^2) \right ) \right ] & \geq \frac{1}{2} \sqrt{\widetilde C} \e^{-\zeta/2}  (1 + \mathbbm{1}_{d=2} |\log \e|^2)^{1/2} \\
& \geq \frac{1}{2^{3/2}} \sqrt{\widetilde C} \e^{-\zeta/2}  (1 + \mathbbm{1}_{d=2} |\log \e|) .
\end{align}
It remains to show that $\widetilde{C} > 6 \log 2$ and $\zeta$ can be chosen such that for all sufficiently small $\e$,
\be
 \sqrt{ \widetilde{C}} \e^{-\zeta/2}(1 + \mathbbm{1}_{d=2} |\log \e|)  \geq \frac{2^{3/2} C_d}{\e}\int_0^{T_\ast} B(s)\,ds+ 2^{3/2} \sqrt{\left|\log\left(\frac{\e}e\right)\right|} .
\ee
The term $2^{3/2} \sqrt{\left|\log\left(\frac{\e}e\right)\right|}$ is clearly of lower order than $\e^{-\zeta/2}$ for any $\zeta > 0$. It therefore suffices to show that
\be \label{lb:B-A}
 \sqrt{ \widetilde{C}} \e^{-\zeta/2}(1 + \mathbbm{1}_{d=2} |\log \e|)  \geq \frac{4 C_d}{\e}\int_0^{T_\ast} B(s)\,ds .
\ee

By \ref{hyp:regular-moments}, \ref{hyp:moments}, Proposition~\ref{prop:mass-growth-2d} and Lemma~\ref{lem:Q-final-3d},
\be
 B(t) \leq \begin{cases} C_2  (1 + T_\ast)^3 \e^{-4} ( 1 + |\log \e|) & d=2 \\
C_3 (T_\ast+1)^{3 + \frac{12}{1 - (13 - 4k_0)_+}} \e^{-6 \left (1  + \frac{4}{1 - (13 - 4k_0)_+} \right )} & d=3 ,
\end{cases}
\ee
and hence 
\be
\frac{1}{\e} \int_0^{T_\ast} B(t) \di t \leq \begin{cases} C_2(T_\ast) \, \e^{-5} ( 1 + |\log \e|) & d=2 \\
C_3(T_\ast, k_0) \, \e^{-7  + \frac{24}{1 - (13 - 4k_0)_+}} & d=3 .
\end{cases}
\ee
Therefore, by choosing
\be
\zeta = \begin{cases}
10 & d=2 \\
2 + 12 \left (1 +  \frac{4}{1 - (13 - 4k_0)_+} \right ) & d=3 ,
\end{cases}
\ee
and $\widetilde{C} > 6 \log 2$ large enough in terms of $T_\ast$, $k_0$ and $d$, we can ensure that \eqref{lb:B-A} holds, which implies the convergence \eqref{eq:lim-fg}.
Thus there exists $C>0$ such that 
\be
\lim_{\e \to 0} \sup_{t \leq T_\ast} W_2(f_\e(t), g(t)) = 0
\ee
under the hypothesis \ref{hyp:regular-moments}-\eqref{hyp:rate}, which completes the proof.
\qed

\subsection{Proof of Corollary~\ref{cor:analytic}}

To complete the proof of Corollary~\ref{cor:analytic}, it remains to establish that hypotheses \ref{hyp:regular-moments} and \ref{hyp:regular-convergence} are satisfied under the assumptions \ref{hyp:analytic} and \ref{hyp:analytic-convergence}.
This can be shown using an adaptation to the ion model of the methods of Grenier \cite{Grenier96} (see the discussion in \cite{IHK1}).

In \cite{Grenier96}, the author introduces a representation of the plasma as a superposition of a possibly uncountable collection of fluids $(\rho_\e^\theta, u_\e^\theta)_{\theta \in \Theta}$, and shows that this multi-fluid system has a solution under the uniform analyticity assumption \ref{hyp:analytic} on the initial data, for each $\e \in (0, \e_\ast)$, on a uniform time interval $T_\ast$; by 
\cite{GPI-WP} these are the unique solutions $g_\e$ of \eqref{eq:vpme} with $\rho[g_\e] \in L^1([0, T_\ast] ; L^\infty(\bt^d))$.
Moreover, the techniques of \cite{Grenier96} can be used to show that the quasineutral limit holds: there exists a solution $g$ of KIsE \eqref{eq:KE-iso} such that $g_\e$ converges to $g$. The convergence can be stated as follows:

For any $\delta ' < \delta$, there exists a time $T_\ast > 0$ and multi-fluids $(\rho_\e^\theta, u_\e^\theta)_{\theta \in \br^{d}}$ bounded in $C([0,T_\ast] ; B_{\delta '})$ such that
\be
g_\e(t,x,v) = \int_{\br^d} \rho^\theta_\e(t,x) \delta_0(v - u^\theta_\e(t,x)) \frac{\di \theta}{1 + |\theta|^{k_0}} , \qquad \text{for all } \e \in (0,1]
\ee
and multi-fluids $(\rho^\theta, u^\theta)_{\theta \in \br^d}$ such that
\be \label{eq:lim-Hs}
\lim_{\e \to 0} \sup_{t \in [0,T_\ast], \theta \in \br^d} \left ( \| \rho^\theta_\e - \rho^\theta \|_{H^s(\bt^d)} + \| u^\theta_\e - u^\theta \|_{H^s(\bt^d)} \right ) = 0 \qquad \text{for all} \; s \in \bb{N}
\ee
and the function
\be
g(t,x,v) : = \int_{\br^d} \rho^\theta(t,x) \delta_0(v - u^\theta(t,x)) \frac{\di \theta}{1 + |\theta|^{k_0}} 
\ee
defines a solution to KIsE \eqref{eq:KE-iso}.
The multi-fluid $H^s$ convergence \eqref{eq:lim-Hs} then implies that
\be
\lim_{\e \to 0} \sup_{t \in [0,T_\ast]} W_p (g_\e, g) = 0 , \qquad p \in [1, + \infty) ,
\ee
and the uniform $C([0,T_\ast] ; B_{\delta '})$ bounds imply that \ref{hyp:regular-moments} is satisfied.
\qed

\subsection{Proof of Corollary~\ref{cor:Penrose}}

To complete the proof of Corollary~\ref{cor:Penrose}, we once again need to establish that hypotheses \ref{hyp:regular-moments} and \ref{hyp:regular-convergence} are satisfied under the assumptions \ref{hyp:Penrose-regularity} and \ref{hyp:Penrose}.
This can be shown using the techniques of \cite{HKR}, with adaptations to case with the full nonlinearity $e^U$ (see the discussion in the introduction to \cite{HKR}).

The analogue of \cite[Theorem 1]{HKR} shows that, under \ref{hyp:Penrose-regularity}-\ref{hyp:Penrose}, there exist solutions $g_\e$ to the $(VPME)_\e$ system \eqref{eq:vpme} with respective initial data $g_{0,\e}$, on a uniform time interval $[0, T_\ast]$, that are bounded in $C \left ( [0, T_\ast] ; \mc{H}^{2m -1 }_{r_0} \right )$ uniformly in $\e$ (once again by \cite{GPI-WP} these are the unique bounded density solutions with the given initial datum), and such that the spatial densities $\rho[g_\e]$ are uniformly bounded in $L^2([0,T_\ast] ; H^{2m}(\bt^d))$. Since $2m > 4 + d/2 + \lfloor d/2 \rfloor$ is certainly greater than $d/2$, $H^{2m}(\bt^d)$ embeds continuously into $L^\infty(\bt^d)$, whence it follows that $\rho[g_\e]$ are uniformly bounded in $L^1([0,T_\ast] ; L^\infty(\bt^d))$. Thus \ref{hyp:regular-moments} is satisfied.

The analogue of \cite[Theorems 2, 3]{HKR} then states that there exists a unique solution $g \in C \left ( [0, T_\ast] ; \mc{H}^{2m -1 }_{r_0} \right )$ of KIsE \eqref{eq:KE-iso}
such that $g_\e$ converges as $\e$ tends to zero to $g$ in the sense of $L^\infty \left ( [0, T_\ast ] ; L^2\cap L^\infty (\bt^d \times \br^d) \right )$.
This is a stronger notion of convergence than that required by \ref{hyp:regular-convergence} and thus this hypothesis is satisfied.
\qed

\begin{paragraph}{Acknowledgments}
The first author was supported by the Additional Funding Programme for Mathematical Sciences, delivered by EPSRC (EP/V521917/1) and the Heilbronn Institute for Mathematical Research, and also thanks the Forschungsinstitut f{\"u}r Mathematik at ETH Z{\"u}rich for its hospitality during the preparation of this work. The second author acknowledges the support of the SNSF Starting Grant \emph{Challenges and Breakthroughs in the Mathematics of Plasmas}, $TMSGI2\_226018$. The authors are very grateful to the anonymous Referees for their careful reading, thoughtful comments, and many insightful suggestions, which led to a substantial improvement of the paper. The authors also warmly thank Mr.~Florian Spicher for his assistance in the creation of the figures.
\end{paragraph}

\appendix

\section{Auxiliary Lemma}

\begin{lemma}
\label{lem:app-inverse}
Consider the function $b: [0, + \infty) \to [0, + \infty)$ defined by
\be
b(y) = \frac{y}{1 + \log(1+y)} .
\ee
Then $b$ has a well-defined continuous strictly increasing inverse $b^{-1}$. Moreover, for all $u \in [0,+\infty)$,
\be \label{est:app-binv}
b^{-1}(u) \leq 2 u (1 + \log(1+u)) .
\ee
\end{lemma}
\begin{proof}
Since $b$ is a continuously differentiable function of $y$, we can check its monotonicity by direct computation of the derivative:
\be
b'(y) = \frac{1 + (1+z)\log{(1+z)}}{(1+z)(1+\log{(1+z)})^2} > 0 \quad \text{for all} \; z\geq 0.
\ee
Hence, since $b$ is continuous and strictly increasing, $b^{-1}$ is well-defined, continuous, and strictly increasing.

To prove \eqref{est:app-binv}, we will show that 
\be
\label{eq:to prove}
u \leq b(2 u (1 + \log(1+u))).
\ee
Applying $b^{-1}$ will then imply the bound, since $b^{-1}$ is increasing.

With this in mind, we compute
\be \label{eq:app-b-comp}
b(2 u (1 + \log(1+u))) = u \cdot \frac{2 (1 + \log(1+u))}{1 + \log (1 + 2 u (1 + \log(1+u)))} .
\ee
We observe that $1 + \log(1+u) \leq 1 + u$, and hence
\be \label{est:app-denom1}
1 + \log (1 + 2 u (1 + \log(1+u))) \leq 1 + \log(1 + 2u(1+u)) .
\ee
Since $u \geq 0$, 
\be
1 + 2u(1 + u) \leq 1 + 4u + 2u^2 \leq 2(1+u)^2 .
\ee
We substitute this inequality into \eqref{est:app-denom1} to find 
\begin{align} 
1 + \log (1 + 2 u (1 + \log(1+u))) \leq 1 + \log 2 +  2\log(1+u) \label{est:app-denom2} \leq 2 (1 + \log{(1+u)}).
\end{align}
Using \eqref{est:app-denom2} and \eqref{eq:app-b-comp} we obtain \eqref{eq:to prove},
which completes the proof.

\end{proof}

\section{Instability in one dimension for VPME} \label{app:Instability}

This appendix shows, in a self-contained way, the instability mechanism for the quasineutral regime in one space dimension, for the Vlasov--Poisson model with massless electrons.
The proof is based on the construction developed by Han-Kwan and Hauray for the screened Vlasov--Poisson dynamics \cite{HKH}.

We apply the same strategy: the screened Penrose condition yields a growing mode for the linearized equation, around which we construct a high-order expansion in the small amplitude parameter $\delta$. 
We then control the difference between the exact solution and the truncated expansion in $L^1$ on the logarithmic time scale at which the leading mode reaches a macroscopic size.

The role of the Poisson--Boltzmann term is seen after the usual rescaling on a fixed torus: the elliptic relation can be
rewritten as a screened elliptic equation plus a nonlinear remainder,
\be
(1-\partial_{xx})\phi = \rho[g]-1 - \mathcal N(\phi),
\qquad
\mathcal N(\phi)=e^\phi-1-\phi.
\ee
The function $\mathcal N$ vanishes to second order at the origin. In the expansion, this creates (besides the Vlasov defect)
an elliptic defect, and it requires a nonlinear estimate for $\mathcal N$ to close the bootstrap on $\|\phi\|_{L^\infty}$.
Once this point is handled, the comparison between the exact solution and the truncated expansion proceeds as in \cite{HKH},
and the tiling-rescaling step transfers the instability back to physical variables. The physical time scale produced by the
argument is $t \sim \e |\log \e|$, hence $t \le \e^\alpha$ for every fixed $\alpha<1$ when $\e$ is small.

\subsection{Polynomial instability for $(\mathrm{VPME})_\varepsilon$ in one dimension}\label{app:poly-instability-vpme}

As in \cite{HKH} we present the argument in the case $d=1$. The case $d\ge 2$ can be reduced to the one-dimensional construction by tensorizing in the transverse velocity variables: one may consider perturbations of the form
$$
f(x,v)=g(x_1,v_1)\chi(v_2,\dots,v_d),
$$
with $\chi\geq 0$ smooth, integrable and normalized to have unit mass. Then the associated spatial density depends only on $x_1$, the potential may be taken to depend only on $x_1$, and the electric field has only its first component non-zero. In this way the one-dimensional instability mechanism can be extended to higher dimensions.

We work on $\T:=\R/\Z$ with $x\in\T$, $v\in\R$, and we consider
\begin{equation}\label{eq:vpme-app}
\left\{
\begin{aligned}
&\partial_t f_\varepsilon + v\,\partial_x f_\varepsilon + E_\varepsilon\,\partial_v f_\varepsilon = 0,\\
&E_\varepsilon = -\partial_x U_\varepsilon,\\
&\varepsilon^2 \partial_{xx}U_\varepsilon = e^{U_\varepsilon} - \rho_\varepsilon,\qquad
\rho_\varepsilon(t,x)=\int_{\R} f_\varepsilon(t,x,v)\,dv,\\
&f_\varepsilon|_{t=0}=f_{0,\varepsilon}\ge 0,\qquad \int_{\T\times\R} f_{0,\varepsilon}\,dx\,dv=1.
\end{aligned}
\right.
\end{equation}
Let $\mu=\mu(v)>0$ be smooth with $\int_\R \mu(v)\,dv=1$, so that $(f_\varepsilon,U_\varepsilon)\equiv(\mu,0)$
is a homogeneous equilibrium.

We assume the following hypotheses, corresponding to \cite[Eq.~(7.7)]{HKH} with $\alpha=1$:
\begin{enumerate}
\item[(H1)] (\emph{$\delta$-condition}) \quad $\displaystyle \sup_{v\in\R}\frac{|\mu'(v)|}{(1+|v|)\mu(v)}<\infty.$
\item[(H2)] (\emph{screened Penrose instability, $\alpha=1$}) there exists a local minimum $\bar v$ of $\mu$ such that
\begin{equation}\label{eq:penrose-a1}
\int_{\R}\frac{\mu(v)-\mu(\bar v)}{(v-\bar v)^2}\,dv>1.
\end{equation}
\end{enumerate}

\medskip\noindent
\textbf{Norms.}
For $m\in\N$ we use the mixed $L^1$-Sobolev norms
\be
\|g\|_{W^{m,1}_{x,v}}:=\sum_{a+b\le m}\|\partial_x^a\partial_v^b g\|_{L^1(\T\times\R)},
\qquad
\|h\|_{W^{m,1}_{x}}:=\sum_{a\le m}\|\partial_x^a h\|_{L^1(\T)}.
\ee
For negative indices we use the standard dual definition.

\begin{theorem}[Polynomial instability for nonlinear VPME]\label{thm:poly-instability-vpme}
Assume \emph{(H1)-(H2)}. Then for every $N\in\N$ and every $s\in\N$ there exist a sequence
$\varepsilon_k\to 0$ and initial data $f_{0,\varepsilon_k}\ge 0$ with unit mass such that
\begin{equation}\label{eq:init-close}
\|f_{0,\varepsilon_k}-\mu\|_{W^{s,1}_{x,v}(\T\times\R)}\le \varepsilon_k^N,
\end{equation}
and denoting by $f_{\varepsilon_k}$ the corresponding solutions of \eqref{eq:vpme-app}, we have:
for every $\alpha\in[0,1)$,
\begin{equation}\label{eq:macro-instability}
\liminf_{k\to\infty}\ \sup_{t\in[0,\varepsilon_k^\alpha]}\ \|\rho_{\varepsilon_k}(t)-1\|_{L^1_x(\T)} >0,
\qquad
\liminf_{k\to\infty}\ \sup_{t\in[0,\varepsilon_k^\alpha]}\ \varepsilon_k\|E_{\varepsilon_k}(t)\|_{L^1_x(\T)} >0.
\end{equation}
Moreover, for every $r\in\Z$ and every $\alpha\in[0,1)$,
\be
\liminf_{k\to\infty}\ \sup_{t\in[0,\varepsilon_k^\alpha]}\ \|f_{\varepsilon_k}(t)-\mu\|_{W^{r,1}_{x,v}(\T\times\R)} >0.
\ee
In particular, the case $r = -1$ gives
\be
\liminf_{k\to\infty}\ \sup_{t\in[0,\varepsilon_k^\alpha]}\ W_1 \left ( f_{\varepsilon_k}(t), \mu \right ) >0.
\ee
\end{theorem}

\medskip\noindent
\begin{remark}[Grenier's instability scheme]\label{rem:grenier}
The approximation we use is a finite expansion in a small amplitude parameter $\delta$,
built from an unstable eigenmode of the linearized dynamics and corrected order by order so that the reminder is
of size $\delta^{N+1}$ (with the corresponding growth in time). The control of the error up to times of order $|\log\delta|$
is part of a general instability scheme that was introduced by Grenier \cite{Grenier00}. A closely related approach, developed
for viscous boundary layers, can be found in Desjardins and Grenier \cite{DG03}. In the kinetic framework we follow \cite{HKH}, adapting the hierarchy to the Poisson--Boltzmann nonlinearity through
$\mathcal N(z)=e^z-1-z$.
\end{remark}

\subsection*{Preliminary estimates for the screened operator and the nonlinearity}

We introduce the notation
\be
\mathcal A:=1-\partial_{xx},
\qquad
\mathcal N(z):=e^z-1-z.
\ee
For $M>0$ we denote by $\T_M$ the torus $\T_M:=\R/(M\Z)$.

The next estimate is the analogue of the bound used in the screened setting for $\mathcal A^{-1}$.
On $\T_M$ the Fourier symbol of $\mathcal A$ is $1+(2\pi n/M)^2\ge 1$, hence the same type of estimate holds.

\begin{lemma}\label{lem:Ainv-like-311}
Let $s \in \N$ and $F\in W^{s,1}(\T_M)$. Then there exists $C_{s,M}$ such that
\be
\|\partial_x^{s+1}\mathcal A^{-1}F\|_{L^\infty(\T_M)}\leq C_{s,M}\|F\|_{W^{s,1}(\T_M)}.
\ee
\end{lemma}

We also need a nonlinear estimate for $\mathcal N$, in the regime where $\|\phi\|_{L^\infty}$ is small.
This is the point where the Poisson--Boltzmann term enters the stability argument.

\begin{lemma}[Nonlinear estimate for $\mathcal N$]\label{lem:N-nonlinear}
Fix $s\in\N$. There exists $C_{s,M}$ such that if $\phi\in W^{s,1}(\T_M)\cap L^\infty(\T_M)$ and $\|\phi\|_{L^\infty}\le 1$, then
\be
\|\mathcal N(\phi)\|_{W^{s,1}(\T_M)}\ \le\ C_{s,M}\,\|\phi\|_{L^\infty(\T_M)}\,\|\phi\|_{W^{s,1}(\T_M)}.
\ee
Moreover, if $\|\phi\|_\infty+\|\psi\|_\infty\le 1$, then
\be
\|\mathcal N(\phi)-\mathcal N(\psi)\|_{W^{s,1}(\T_M)}\ \le\ C_{s,M}(\|\phi\|_{L^\infty(\T_M)}+\|\psi\|_{L^\infty(\T_M)})\,\|\phi-\psi\|_{W^{s,1}(\T_M)}.
\ee
\end{lemma}

\begin{proof}
Since $\mathcal N(z)=\sum_{\ell\ge 2} z^\ell/\ell!$ is analytic with $\mathcal N(0)=\mathcal N'(0)=0$,
one obtains the first estimate by combining the expansion with standard product bounds in $W^{s,1}(\T_M)$.

For the second estimate we write
\be
\mathcal N(\phi)-\mathcal N(\psi)=(\phi-\psi)\int_0^1 \big(e^{\tau\phi+(1-\tau)\psi}-1\big)\,d\tau,
\ee
and we use again product bounds. When $\|\phi\|_\infty+\|\psi\|_\infty$ is small we have
$\|e^{\tau\phi+(1-\tau)\psi}-1\|_{L^\infty}\lesssim \|\phi\|_\infty+\|\psi\|_\infty$, uniformly in $\tau\in[0,1]$,
which leads to the stated estimate.
\end{proof}

\begin{proof}[Proof of Theorem~\ref{thm:poly-instability-vpme}]
We follow the strategy of \cite{HKH}, and we indicate explicitly where the Poisson--Boltzmann nonlinearity enters.

We work first on the fixed torus $\T_M$, in fast variables $(t/\varepsilon,x/\varepsilon)$, where the unstable mode is built.
Then we rescale back and tile in space to produce solutions on $\T$.

\smallskip\noindent
\emph{Oscillatory sequence, tiling, and rescaling.}
Fix $M\ge 1$ and consider $\varepsilon=\varepsilon_k:=\frac{1}{kM}$.
As in the screened case, we first work on the small torus $\T_{\varepsilon M}$ and then tile $k$ times.
We introduce the rescaled unknowns
\be
\tilde f_\varepsilon(t,x,v)=g\Big(\frac{t}{\varepsilon},\frac{x}{\varepsilon},v\Big),\qquad
\tilde U_\varepsilon(t,x)=\phi\Big(\frac{t}{\varepsilon},\frac{x}{\varepsilon}\Big),\qquad
E=-\partial_x\phi,
\ee
posed on the fixed torus $\T_M$. 

Set $s=t/\varepsilon$ and $y=x/\varepsilon$; we write $g(s,y,v)$ and $\phi(s,y)$ on $\T_M$.

The Vlasov equation keeps the same form, while the elliptic equation becomes
\be
\partial_{xx}\phi = e^\phi-\rho[g],\qquad \rho[g]=\int_{\R} g\,dv.
\ee

\smallskip\noindent
\emph{Screened decomposition.}
With $\mathcal A=1-\partial_{xx}$ and $\mathcal N(z)=e^z-1-z$, we rewrite the elliptic relation as
\begin{equation}\label{eq:screened-form}
\mathcal A\phi=\rho[g]-1-\mathcal N(\phi).
\end{equation}
When $\mathcal N(\phi)$ is omitted, one recovers the screened coupling with $\alpha=1$.

\smallskip\noindent
\emph{Linear instability.}
Linearizing at $(\mu,0)$, the nonlinear term disappears since $\mathcal N'(0)=0$.
Hence the linearized system coincides with the screened linearized system.
Under \eqref{eq:penrose-a1}, there exist an unstable eigenvalue $\lambda$ with $\sigma:=\Re\lambda>0$
and a real growing mode $h_1(t,x,v)$ whose density is of the form $C e^{\sigma t} \cos \left (2\pi k x + t \Im \lambda \right )$ for some non-zero integer $k \in \Z \setminus \{ 0\}$ and constant $C$.

In particular, the associated density fluctuation has zero spatial mean and an oscillatory profile with fixed frequency on $\T_M$.
After scaling back to the physical variable $x$, this profile becomes $\varepsilon$-oscillatory, and the tiling step
replicates it on $\T$ without changing the size of $\|\rho_\varepsilon-1\|_{L^1_x}$ on the time interval considered.

\smallskip\noindent
\emph{High-order approximation and the two defects.}
Fix $N\ge 2$ and $\delta>0$ small. We seek an approximate solution as a finite expansion in $\delta$,
\be
g_N^{\app}=\mu+\sum_{j=1}^N\delta^j h_j,\qquad
\phi_N^{\app}=\sum_{j=1}^N\delta^j \phi_j,\qquad
E_N^{\app}=-\partial_x\phi_N^{\app},
\ee
with $h_1$ as above and $h_j(0)=0$ for $j\ge 2$.
We define the Vlasov defect by
\begin{equation}\label{eq:def-RN}
R_N^{\app}
:=\partial_t g_N^{\app}+v\partial_x g_N^{\app}+E_N^{\app}\partial_v g_N^{\app}.
\end{equation}

In the present model we also introduce an elliptic defect, measuring how well $\phi_N^{\app}$ solves \eqref{eq:screened-form}:
\begin{equation}\label{eq:def-elliptic-defect}
\mathcal R_N
:= \big(\rho[g_N^{\app}]-1-\mathcal N(\phi_N^{\app})\big)-\mathcal A\phi_N^{\app}.
\end{equation}

The quantity $\mathcal R_N$ measures the mismatch in the elliptic relation \eqref{eq:screened-form} produced by truncating the expansion
of the Poisson--Boltzmann term at order $N$.
It is convenient to isolate it, because in the stability step it appears as a source term in the equation for
$\eta=\phi-\phi_N^{\app}$, together with $\rho[w]$ and the nonlinear difference $\mathcal N(\phi)-\mathcal N(\phi_N^{\app})$.

Equivalently,
\begin{equation}\label{eq:approx-poisson-with-defect}
\mathcal A\phi_N^{\app}=\rho[g_N^{\app}]-1-\mathcal N(\phi_N^{\app})-\mathcal R_N.
\end{equation}

\smallskip\noindent
\emph{Poisson--Boltzmann correction in the hierarchy.}
Expanding $\mathcal N(\sum_{j\ge 1}\delta^j\phi_j)$ in powers of $\delta$ defines homogeneous polynomials
$\mathcal N_j(\phi_1,\dots,\phi_{j-1})$ ($j\ge 2$) by
\be
\mathcal N\Big(\sum_{j\ge 1}\delta^j\phi_j\Big)=\sum_{q\ge 2}\delta^q\,\mathcal N_q(\phi_1,\dots,\phi_{q-1}).
\ee
Matching powers of $\delta$ in \eqref{eq:screened-form} gives the elliptic recursion
\be
\mathcal A\phi_1=\rho[h_1],\qquad \mathcal A\phi_j=\rho[h_j]-\mathcal N_j(\phi_1,\dots,\phi_{j-1})\quad (j\ge 2).
\ee
Writing $E_j:=-\partial_x\phi_j$ in the form $E_j=E_j^{\lin}[h_j]+E_j^{\NL}$, with
\be
E_j^{\lin}[h_j]:=-\partial_x\mathcal A^{-1}\rho[h_j],
\qquad
E_j^{\NL}:=\partial_x\mathcal A^{-1}\mathcal N_j(\phi_1,\dots,\phi_{j-1}),
\ee
the Vlasov hierarchy keeps the same structure, with an additional forcing term driven by $E_j^{\NL}$:
\be
\partial_t h_j + L_1 h_j + \sum_{\ell=1}^{j-1}E_\ell\,\partial_v h_{j-\ell}
= -E_j^{\NL}\,\mu'(v),\qquad h_j(0)=0,\quad j\ge 2.
\ee
Here $L_1$ denotes the screened linearized operator (at $\alpha=1$).

\smallskip\noindent
\emph{Bounds on the defects.}
Since the expansion of $\mathcal N$ starts at order $2$, the nonlinear fields $E_j^{\NL}$ involve only lower-order profiles
and they fit in the same induction that controls $h_j$.
Using Lemma~\ref{lem:Ainv-like-311} and Lemma~\ref{lem:N-nonlinear}, one obtains bounds of the form
\begin{equation}\label{eq:defect-bounds}
\|R_N^{\app}(t)\|_{L^1_{x,v}}\lesssim_N \delta^{N+1}e^{(N+1)\sigma t},
\qquad
\|\mathcal R_N(t)\|_{L^1_x}\lesssim_N \delta^{N+1}e^{(N+1)\sigma t},
\end{equation}
as long as $\delta e^{\sigma t}\le \theta_0\ll 1$, for a fixed $\theta_0$.

\smallskip\noindent
\emph{$L^1$ stability and control of the field.}
Let $(g,\phi)$ be the exact solution with $g(0)=g_N^{\app}(0)$.
Set $w=g-g_N^{\app}$, $\eta=\phi-\phi_N^{\app}$, $\mathcal E=E-E_N^{\app}$.
The $L^1$ estimate gives
\be
\frac{d}{dt}\|w(t)\|_{L^1_{x,v}}
\le \|\mathcal E(t)\|_{L^\infty_x}\,\|\partial_v g_N^{\app}(t)\|_{L^1_{x,v}}
+ \|R_N^{\app}(t)\|_{L^1_{x,v}}.
\ee
Subtracting the exact elliptic relation \eqref{eq:screened-form} from \eqref{eq:approx-poisson-with-defect} yields
\begin{equation}\label{eq:eta-eq}
\mathcal A\eta=\rho[w]-\big(\mathcal N(\phi)-\mathcal N(\phi_N^{\app})\big)+\mathcal R_N.
\end{equation}
Applying $\partial_x\mathcal A^{-1}$ and using Lemma~\ref{lem:Ainv-like-311} and the boundedness of $\mathcal A^{-1}$ on $L^1(\T_M)$, we obtain
\be
\|\mathcal E\|_{L^\infty_x}\lesssim
\|w\|_{L^1_{x,v}}+\|\mathcal N(\phi)-\mathcal N(\phi_N^{\app})\|_{L^1_x}+\|\mathcal R_N\|_{L^1_x}.
\ee
Similarly, applying $\mathcal A^{-1}$ to \eqref{eq:eta-eq} yields
\be
\|\eta\|_{L^1_x}\lesssim
\|w\|_{L^1_{x,v}}+\|\mathcal N(\phi)-\mathcal N(\phi_N^{\app})\|_{L^1_x}+\|\mathcal R_N\|_{L^1_x}.
\ee

Assume the bootstrap $\|\phi\|_\infty+\|\phi_N^{\app}\|_\infty\le 2\theta_0$, with $\theta_0$ small.
Then Lemma~\ref{lem:N-nonlinear} gives
$\|\mathcal N(\phi)-\mathcal N(\phi_N^{\app})\|_{L^1}\lesssim \theta_0\|\eta\|_{L^1}$,
and for $\theta_0$ small this contribution is absorbed. Using \eqref{eq:defect-bounds} we arrive at
\be
\|\mathcal E(t)\|_{L^\infty_x}\lesssim \|w(t)\|_{L^1_{x,v}}+\delta^{N+1}e^{(N+1)\sigma t}.
\ee
A Gr\"onwall argument then yields the same type of estimate as in the screened case, up to the time
\be
t_\delta:=\sigma^{-1}\log(\theta_0/\delta),
\qquad\text{so that}\qquad \delta e^{\sigma t_\delta}=\theta_0.
\ee

\smallskip\noindent
\emph{Instability at logarithmic time and scaling back.}
At $t=t_\delta$ the leading mode dominates and produces a nontrivial deviation of the density from $1$,
as in the screened case. For the field, we use \eqref{eq:screened-form} and the fact that
$\|\mathcal N(\phi)\|_{L^1}=O(\theta_0^2)$ in the bootstrap regime.
Scaling back to the original variables and tiling on $\T$ give \eqref{eq:macro-instability}.
Choosing $\delta=\varepsilon^P$ with $P$ large gives \eqref{eq:init-close}.
Finally, by (H1) and the smallness of $\delta$, one ensures $f_{0,\varepsilon}\ge 0$.

Since $\varepsilon t_\delta\sim \varepsilon|\log\varepsilon|$, we have $\varepsilon t_\delta\le \varepsilon^\alpha$
for every fixed $\alpha<1$ and $\varepsilon$ sufficiently small.
The amplification in $W^{r,1}$ follows by the same argument used for the screened model; we do not repeat it here.
\end{proof}

\bibliography{VPME-quasi}

\begin{thebibliography}{10}

\bibitem{Ambrosio2008}
L.~Ambrosio.
\newblock Transport equation and cauchy problem for non-smooth vector fields.
\newblock In B.~Dacorogna and P.~Marcellini, editors, {\em Calculus of
  Variations and Nonlinear Partial Differential Equations}, volume 1927 of {\em
  Lecture Notes in Mathematics}, pages 1--41. Springer Berlin Heidelberg, 2008.

\bibitem{AmbrosioColomboFigalli}
L.~Ambrosio, M.~Colombo, and A.~Figalli.
\newblock On the {L}agrangian structure of transport equations: The
  {V}lasov--{P}oisson system.
\newblock {\em Duke Mathematical Journal}, 166(18):3505--3568, 2017.

\bibitem{Arsenev}
A.~Arsenev.
\newblock Existence in the large of a weak solution to the {Vlasov} system of
  equations.
\newblock {\em Zh. Vychisl. Mat. i Mat. Fiz.}, 15:136--147, 1975.

\bibitem{Baradat}
A.~Baradat.
\newblock Nonlinear instability in {Vlasov} type equations around rough
  velocity profiles.
\newblock {\em Annales de l'Institut Henri Poincar{\'{e}} C, Analyse non
  lin{\'{e}}aire}, 37(3):489 -- 547, 2020.

\bibitem{Bardos}
C.~Bardos.
\newblock {About a Variant of the 1d Vlasov equation, dubbed
  ``Vlasov-Dirac-Benney equation''}.
\newblock In {\em S{\'{e}}minaire Laurent Schwartz - {\'{E}}quations aux
  d{\'{e}}riv{\'{e}}es partielles et applications. Ann{\'{e}}e 2012-2013.},
  S{\'{e}}min. {\'{E}}qu. D{\'{e}}riv. Partielles, pages 1--21. {\'{E}}cole
  Polytechnique, Centre de Math{\'{e}}matiques, Palaiseau, 2014.

\bibitem{Bardos-Besse}
C.~Bardos and N.~Besse.
\newblock {The Cauchy problem for the Vlasov-Dirac-Benney equation and related
  issues in fluid mechanics and semi-classical limits}.
\newblock {\em Kinet. Relat. Models}, 6(4):893--917, 2013.

\bibitem{Bardos-Besse2015}
C.~Bardos and N.~Besse.
\newblock Hamiltonian structure, fluid representation and stability for the
  {V}lasov-{D}irac-{B}enney equation.
\newblock In {\em Hamiltonian partial differential equations and applications},
  volume~75 of {\em Fields Inst. Commun.}, pages 1--30. Fields Inst. Res. Math.
  Sci., Toronto, ON, 2015.

\bibitem{Bardos-BesseSC}
C.~Bardos and N.~Besse.
\newblock Semi-classical limit of an infinite dimensional system of nonlinear
  {S}chr\"{o}dinger equations.
\newblock {\em Bull. Inst. Math. Acad. Sin. (N.S.)}, 11(1):43--61, 2016.

\bibitem{BGNS18}
C.~Bardos, F.~Golse, T.~T. Nguyen, and R.~Sentis.
\newblock The {M}axwell-{B}oltzmann approximation for ion kinetic modeling.
\newblock {\em Phys. D}, 376/377:94--107, 2018.

\bibitem{Bardos-Nouri}
C.~Bardos and A.~Nouri.
\newblock {A Vlasov equation with Dirac potential used in fusion plasmas}.
\newblock {\em J. Math. Phys.}, 53(11):115621, 2012.

\bibitem{Batt-Rein}
J.~Batt and G.~Rein.
\newblock {Global classical solutions of the periodic Vlasov-Poisson system in
  three dimensions}.
\newblock {\em C. R. Acad. Sci. Paris S{\'{e}}r. I Math.}, 313(6):411--416,
  1991.

\bibitem{BMM2016}
J.~Bedrossian, N.~Masmoudi, and C.~Mouhot.
\newblock Landau damping: paraproducts and {G}evrey regularity.
\newblock {\em Ann. PDE}, 2(1):Art. 4, 71, 2016.

\bibitem{BedrossianMasmoudiMouhot}
J.~Bedrossian, N.~Masmoudi, and C.~Mouhot.
\newblock Landau damping in finite regularity for unconfined systems with
  screened interactions.
\newblock {\em Comm. Pure Appl. Math.}, 71(3):537--576, 2018.

\bibitem{BMM2022}
J.~Bedrossian, N.~Masmoudi, and C.~Mouhot.
\newblock Linearized wave-damping structure of {V}lasov-{P}oisson in {$\Bbb
  R^3$}.
\newblock {\em SIAM J. Math. Anal.}, 54(4):4379--4406, 2022.

\bibitem{BCGIR}
D.~Benedetto, E.~Caglioti, A.~Gagnebin, M.~Iacobelli, and S.~Rossi.
\newblock Scattering problem for {Vlasov}-type equations on the d-dimensional
  torus with {Gevrey} data.
\newblock {\em Ann. Inst. H. Poincar{\'{e}} C Anal. Non Lin{\'{e}}aire}.
\newblock (To appear).

\bibitem{BenyiOh}
A.~B\'{e}nyi and T.~Oh.
\newblock The {S}obolev inequality on the torus revisited.
\newblock {\em Publ. Math. Debrecen}, 83(3):359--374, 2013.

\bibitem{BPLT1991}
G.~Bonhomme, T.~Pierre, G.~Leclert, and J.~Trulsen.
\newblock Ion phase space vortices in ion beam-plasma systems and their
  relation with the ion acoustic instability: numerical and experimental
  results.
\newblock {\em Plasma Physics and Controlled Fusion}, 33(5):507--520, may 1991.

\bibitem{Bouchut}
F.~Bouchut.
\newblock {Global weak solution of the Vlasov-Poisson system for small
  electrons mass}.
\newblock {\em Comm. Partial Differential Equations}, 16(8-9):1337--1365, 1991.

\bibitem{Brenier1989}
Y.~Brenier.
\newblock Une formulation de type {Vlassov--Poisson} pour les {\'{e}}quations
  d'{E}uler des fluides parfaits incompressibles.
\newblock [Rapport de recherche] RR-1070, INRIA, 1989.

\bibitem{Brenier2000}
Y.~Brenier.
\newblock Convergence of the {V}lasov-{P}oisson system to the incompressible
  {E}uler equations.
\newblock {\em Comm. Partial Differential Equations}, 25(3-4):737--754, 2000.

\bibitem{BG}
Y.~Brenier and E.~Grenier.
\newblock {Limite singuli{\`{e}}re du syst{\`{e}}me de Vlasov-Poisson dans le
  r{\'{e}}gime de quasi neutralit{\'{e}} : le cas ind{\'{e}}pendant du temps}.
\newblock {\em C. R. Acad. Sci. Paris S{\'{e}}r. I Math.}, 318(2):121--124,
  1994.

\bibitem{Carles-Nouri}
R.~Carles and A.~Nouri.
\newblock Monokinetic solutions to a singular {V}lasov equation from a
  semiclassical perspective.
\newblock {\em Asymptot. Anal.}, 102(1-2):99--117, 2017.

\bibitem{CesbronIac}
L.~Cesbron and M.~Iacobelli.
\newblock Global well-posedness of {Vlasov-Poisson}-type systems in bounded
  domains.
\newblock {\em Analysis and PDE}, 16(10):2465--2494, 2023.

\bibitem{ChatLukNguyen}
S.~Chaturvedi, J.~Luk, and T.~Nguyen.
\newblock The {V}lasov-{P}oisson-{L}andau system in the weakly collisional
  regime.
\newblock {\em J. Amer. Math. Soc.}, 36(4):1103--1189, 2023.

\bibitem{Chen}
F.~F. Chen.
\newblock {\em Introduction to Plasma Physics and Controlled Fusion}.
\newblock Springer International Publishing, 3 edition, 2016.

\bibitem{ChenChen}
Z.~Chen and J.~Chen.
\newblock Moments propagation for weak solutions of the {Vlasov-Poisson} system
  in the three-dimensional torus.
\newblock {\em J. Math. Anal. Appl.}, 42(1):728--737, 2019.

\bibitem{DG03}
B.~Desjardins and E.~Grenier.
\newblock Linear instability implies nonlinear instability for various types of
  viscous boundary layers.
\newblock {\em Ann. Inst. H. Poincar{\'{e}} Anal. Non Lin{\'{e}}aire},
  20(1):87--106, 2003.

\bibitem{Ferriere}
G.~Ferriere.
\newblock Convergence rate in {Wasserstein} distance and semiclassical limit
  for the defocusing logarithmic {S}chr{\"{o}}dinger equation.
\newblock {\em Analysis and PDE}, 14(2):617--666, 2021.

\bibitem{FlynnGuo}
P.~Flynn and Y.~Guo.
\newblock The massless electron limit of the {V}lasov-{P}oisson-{L}andau
  system.
\newblock {\em Communications in Mathematical Physics}, 405(2):27, 2024.

\bibitem{Gagnebin}
A.~Gagnebin.
\newblock Backward problem for the {1D} ionic {V}lasov-{P}oisson equation.
\newblock {\em Kinetic and Related Models}, 17(2):312--330, 2024.

\bibitem{GagnebinIacobelli}
A.~Gagnebin and M.~Iacobelli.
\newblock Landau damping on the torus for the {V}lasov-{P}oisson system with
  massless electrons.
\newblock {\em Journal of Differential Equations}, 376:154--203, 2023.

\bibitem{Golse-SR2003}
F.~Golse and L.~Saint-Raymond.
\newblock The {V}lasov-{P}oisson system with strong magnetic field in
  quasineutral regime.
\newblock {\em Math. Models Methods Appl. Sci.}, 13(5):661--714, 2003.

\bibitem{Grenier95}
E.~Grenier.
\newblock {Defect measures of the Vlasov-Poisson system in the quasineutral
  regime}.
\newblock {\em Comm. Partial Differential Equations}, 20(7-8):1189--1215, 1995.

\bibitem{Grenier96}
E.~Grenier.
\newblock {Oscillations in quasineutral plasmas}.
\newblock {\em Comm. Partial Differential Equations}, 21(3-4):363--394, 1996.

\bibitem{Grenier99}
E.~Grenier.
\newblock Limite quasineutre en dimension 1.
\newblock In {\em Journ\'{e}es ``\'{E}quations aux {D}\'{e}riv\'{e}es
  {P}artielles'' ({S}aint-{J}ean-de-{M}onts, 1999)}, pages Exp. No. II, 8.
  Univ. Nantes, Nantes, 1999.

\bibitem{Grenier00}
E.~Grenier.
\newblock On the nonlinear instability of {Euler} and {Prandtl} equations.
\newblock {\em Comm. Pure Appl. Math.}, 53(9):1067--1091, 2000.

\bibitem{GrenierNguyenRodnianski}
E.~Grenier, T.~T. Nguyen, and I.~Rodnianski.
\newblock Landau damping for analytic and {G}evrey data.
\newblock {\em Math. Res. Lett.}, 28(6):1679--1702, 2021.

\bibitem{GNR2022}
E.~Grenier, T.~T. Nguyen, and I.~Rodnianski.
\newblock Plasma echoes near stable {P}enrose data.
\newblock {\em SIAM J. Math. Anal.}, 54(1):940--953, 2022.

\bibitem{GPI1}
M.~Griffin-Pickering and M.~Iacobelli.
\newblock A mean field approach to the quasi-neutral limit for the
  {Vlasov--Poisson} equation.
\newblock {\em SIAM J. Math. Anal.}, 50(5):5502--5536, 2018.

\bibitem{GPI-MFQN}
M.~Griffin-Pickering and M.~Iacobelli.
\newblock Singular limits for plasmas with thermalised electrons.
\newblock {\em Journal de Math{\'{e}}matiques Pures et Appliqu{\'{e}}es},
  135:199 -- 255, 2020.

\bibitem{GPI-WP-R3}
M.~Griffin-Pickering and M.~Iacobelli.
\newblock Global strong solutions in {${\mathbb{R}}^3$} for ionic
  {V}lasov-{P}oisson systems.
\newblock {\em Kinet. Relat. Models}, 14(4):571--597, 2021.

\bibitem{GPI-WP}
M.~Griffin-Pickering and M.~Iacobelli.
\newblock Global well-posedness for the {V}lasov-{P}oisson system with massless
  electrons in the 3-dimensional torus.
\newblock {\em Communications in Partial Differential Equations},
  46(10):1892--1939, 2021.

\bibitem{GPI-WPproceedings2}
M.~Griffin-Pickering and M.~Iacobelli.
\newblock Recent developments on quasineutral limits for {V}lasov-type
  equations.
\newblock In {\em Recent advances in kinetic equations and applications},
  volume~48 of {\em Springer INdAM Ser.}, pages 211--231. Springer, Cham,
  [2021] \copyright 2021.

\bibitem{GPI-WPproceedings}
M.~Griffin-Pickering and M.~Iacobelli.
\newblock Recent developments on the well-posedness theory for {V}lasov-type
  equations.
\newblock In {\em From particle systems to partial differential equations},
  volume 352 of {\em Springer Proc. Math. Stat.}, pages 301--319. Springer,
  Cham, [2021] \copyright 2021.

\bibitem{Gurevich-Pitaevsky75}
A.~V. Gurevich and L.~P. Pitaevsky.
\newblock Non-linear dynamics of a rarefied ionized gas.
\newblock {\em Progress in Aerospace Sciences}, 16(3):227 -- 272, 1975.

\bibitem{Han-Kwan2011}
D.~Han-Kwan.
\newblock Quasineutral limit of the {Vlasov--Poisson} system with massless
  electrons.
\newblock {\em Comm. Partial Differential Equations}, 36(8):1385--1425, 2011.

\bibitem{HanKwanHDR}
D.~Han-Kwan.
\newblock Stabilit{\'{e}}, limites singuli{\`{e}}res et conditions de
  contr{\^{o}}le g{\'{e}}om{\'{e}}trique en th{\'{e}}orie cin{\'{e}}tique.
\newblock HDR Thesis, 2017.

\bibitem{HKH}
D.~Han-Kwan and M.~Hauray.
\newblock Stability issues in the quasineutral limit of the one-dimensional
  {Vlasov-Poisson} equation.
\newblock {\em Comm. Math. Phys.}, 334(2):1101--1152, 2015.

\bibitem{IHK2}
D.~Han-Kwan and M.~Iacobelli.
\newblock Quasineutral limit for {Vlasov-Poisson} via {Wasserstein} stability
  estimates in higher dimension.
\newblock {\em J. Differential Equations}, 263(1):1--25, 2017.

\bibitem{IHK1}
D.~Han-Kwan and M.~Iacobelli.
\newblock {The quasineutral limit of the Vlasov-Poisson equation in Wasserstein
  metric}.
\newblock {\em Commun. Math. Sci.}, 15(2):481--509, 2017.

\bibitem{Han-Kwan-Nguyen}
D.~Han-Kwan and T.~T. Nguyen.
\newblock Ill-posedness of the hydrostatic {E}uler and singular {V}lasov
  equations.
\newblock {\em Arch. Ration. Mech. Anal.}, 221(3):1317--1344, 2016.

\bibitem{HKNguyenRousset2021}
D.~Han-Kwan, T.~T. Nguyen, and F.~Rousset.
\newblock Asymptotic stability of equilibria for screened {V}lasov-{P}oisson
  systems via pointwise dispersive estimates.
\newblock {\em Ann. PDE}, 7(2):Paper No. 18, 37, 2021.

\bibitem{HKNguyenRousset-WS}
D.~Han-Kwan, T.~T. Nguyen, and F.~Rousset.
\newblock On the linearized {V}lasov-{P}oisson system on the whole space around
  stable homogeneous equilibria.
\newblock {\em Comm. Math. Phys.}, 387(3):1405--1440, 2021.

\bibitem{HKR}
D.~Han-Kwan and F.~Rousset.
\newblock {Quasineutral limit for Vlasov-Poisson with Penrose stable data}.
\newblock {\em Ann. Sci. {\'{E}}c. Norm. Sup{\'{e}}r. (4)}, 49(6):1445--1495,
  2016.

\bibitem{Hauray}
M.~Hauray.
\newblock Mean field limit for the one dimensional {Vlasov-Poisson~equation}.
\newblock {\em S\'eminaire Laurent Schwartz {\textemdash} EDP et applications},
  2012-2013.
\newblock Talk no. 21.

\bibitem{HoferWinter}
R.~H{\"{o}}fer and R.~Winter.
\newblock A fast point charge interacting with the screened {V}lasov-{P}oisson
  system.
\newblock {\em Analysis and PDE}, 17(7):2451--2507, 2024.

\bibitem{Hormander}
L.~H{\"{o}}rmander.
\newblock {\em The Analysis of Linear Partial Differential Operators I}.
\newblock Springer-Verlag, second edition, 1990.

\bibitem{HuangNguyenXu05}
L.~Huang, Q.-H. Nguyen, and Y.~Xu.
\newblock Nonlinear {Landau} damping for the 2d {Vlasov-Poisson} system with
  massless electrons around {Penrose}-stable equilibrium.
\newblock {\em SIAM Journal on Mathematical Analysis}, 57(2):1939--1963, 2025.

\bibitem{HuangNguyenXu06}
L.~Huang, Q.-H. Nguyen, and Y.~Xu.
\newblock Sharp estimates for screened {Vlasov-Poisson} system around
  {Penrose}-stable equilibria in $\mathbb{R}^d $, $ d\geq3$.
\newblock {\em Kinetic and Related Models}, 18(2):186--251, 2025.

\bibitem{Iac22}
M.~Iacobelli.
\newblock A new perspective on {W}asserstein distances for kinetic problems.
\newblock {\em Arch. Ration. Mech. Anal.}, 244(1):27--50, 2022.

\bibitem{IRW}
M.~Iacobelli, S.~Rossi, and K.~Widmayer.
\newblock On the stability of vacuum in the screened {Vlasov}-{Poisson}
  equation.
\newblock {\em J. Lond. Math. Soc.}, 113(1):e70426, 2026.

\bibitem{Jabin-Nouri}
P.~Jabin and A.~Nouri.
\newblock Analytic solutions to a strongly nonlinear {Vlasov} equation.
\newblock {\em C.R. Acad. Sci. Paris, S{\'{e}}r. 1}, 349:541--546, 2011.

\bibitem{Lions-Perthame}
P.~L. Lions and B.~Perthame.
\newblock {Propagation of moments and regularity for the 3-dimensional
  Vlasov-Poisson system}.
\newblock {\em Invent. Math.}, 105(2):415--430, 1991.

\bibitem{Loeper}
G.~Loeper.
\newblock {Uniqueness of the solution to the Vlasov-Poisson system with bounded
  density}.
\newblock {\em J. Math. Pures Appl. (9)}, 86(1):68--79, 2006.

\bibitem{BM}
A.~Majda and A.~Bertozzi.
\newblock {\em {Vorticity and Incompressible Flow}}, volume~27 of {\em
  Cambridge Texts in Applied Mathematics}.
\newblock Cambridge University Press, 2002.

\bibitem{Masmoudi2001}
N.~Masmoudi.
\newblock From {Vlasov--Poisson} system to the incompressible {Euler} system.
\newblock {\em Comm. Partial Differential Equations}, 26(9-10), 2001.

\bibitem{Mason71}
R.~J. Mason.
\newblock Computer simulation of ion-acoustic shocks. the diaphragm problem.
\newblock {\em The Physics of Fluids}, 14(9):1943--1958, 1971.

\bibitem{Medvedev2011}
Y.~V. Medvedev.
\newblock Ion front in an expanding collisionless plasma.
\newblock {\em Plasma Physics and Controlled Fusion}, 53(12):125007, nov 2011.

\bibitem{MouhotVillani}
C.~Mouhot and C.~Villani.
\newblock On {L}andau damping.
\newblock {\em Acta Math.}, 207(1):29--201, 2011.

\bibitem{TrinhNguyen}
T.~T. Nguyen.
\newblock Derivative estimates for screened {V}lasov-{P}oisson system around
  {P}enrose-stable equilibria.
\newblock {\em Kinet. Relat. Models}, 13(6):1193--1218, 2020.

\bibitem{nguyen2023landau}
T.~T. Nguyen.
\newblock Landau damping and the survival threshold, 2023.
\newblock Preprint, available at arXiv:2305.08672.

\bibitem{Pallard}
C.~Pallard.
\newblock Moment propagation for weak solutions to the {Vlasov-Poisson} system.
\newblock {\em Comm. Partial Differential Equations}, 37(7):1273--1285, 2012.

\bibitem{Penrose}
O.~Penrose.
\newblock {Electrostatic Instabilities of a Uniform Non-Maxwellian Plasma}.
\newblock {\em Phys. Fluids}, 3(2):258--265, 1960.

\bibitem{Pfaffelmoser}
K.~Pfaffelmoser.
\newblock {Global classical solutions of the Vlasov-Poisson system in three
  dimensions for general initial data}.
\newblock {\em J. Differential Equations}, 95(2):281--303, 1992.

\bibitem{PuelM2AN}
M.~Puel.
\newblock Convergence of the {S}chr\"{o}dinger-{P}oisson system to the {E}uler
  equations under the influence of a large magnetic field.
\newblock {\em M2AN Math. Model. Numer. Anal.}, 36(6):1071--1090 (2003), 2002.

\bibitem{PuelCPDE}
M.~Puel.
\newblock Convergence of the {S}chr\"{o}dinger-{P}oisson system to the
  incompressible {E}uler equations.
\newblock {\em Comm. Partial Differential Equations}, 27(11-12):2311--2331,
  2002.

\bibitem{PuelSR}
M.~Puel and L.~Saint-Raymond.
\newblock Quasineutral limit for the relativistic {V}lasov-{M}axwell system.
\newblock {\em Asymptot. Anal.}, 40(3-4):303--352, 2004.

\bibitem{RobertsBerkPRL}
K.~V. Roberts and H.~L. Berk.
\newblock Nonlinear evolution of a two-stream instability.
\newblock {\em Phys. Rev. Lett.}, 19:297, 1967.

\bibitem{SCM}
P.~Sakanaka, C.~Chu, and T.~Marshall.
\newblock Formation of ion-acoustic collisionless shocks.
\newblock {\em The Physics of Fluids}, 14(611), 1971.

\bibitem{Schaeffer}
J.~Schaeffer.
\newblock {Global existence of smooth solutions to the Vlasov-Poisson system in
  three dimensions}.
\newblock {\em Comm. Partial Differential Equations}, 16(8-9):1313--1335, 1991.

\bibitem{SchaefferScreened}
J.~Schaeffer.
\newblock The screened {P}oisson-{V}lasov system with infinite charge and
  vanishing screening.
\newblock {\em Comm. Math. Phys.}, 333(1):97--116, 2015.

\bibitem{Strichartz}
R.~S. Strichartz.
\newblock Improved {S}obolev inequalities.
\newblock {\em Trans. Amer. Math. Soc.}, 279(1):397--409, 1983.

\bibitem{Titchmarsh}
E.~Titchmarsh.
\newblock {\em Eigenfunction Expansions Associated with Second-Order
  Differential Equations, Part II}.
\newblock Oxford University Press, 1958.

\bibitem{Ukai-Okabe}
S.~Ukai and T.~Okabe.
\newblock On classical solutions in the large in time of two-dimensional
  {Vlasov's} equation.
\newblock {\em Osaka J. Math.}, 15(2):245--261, 1978.

\bibitem{Vlasov}
A.~A. Vlasov.
\newblock On the vibration properties of the electron gas.
\newblock {\em Zh. Eksper. Teor. Fiz.}, 8(3):291, 1938.

\bibitem{Wei2025}
D.~Wei.
\newblock Nonlinear stability of the one dimensional screened
  {Vlasov}-{Poisson} equation.
\newblock {\em Commun. Pure Appl. Anal.}, 2025.

\bibitem{Yudovich}
V.~Yudovich.
\newblock Non-stationary flows of an ideal incompressible fluid.
\newblock {\em \v{Z}. Vy\v{c}isl. Mat i Mat. Fiz.}, 3:1032--1066, 1963.

\bibitem{Zakharov}
V.~E. Zakharov.
\newblock Benney equations and quasiclassical approximation in the inverse
  problem method.
\newblock {\em Funktsional. Anal. i Prilozhen.}, 14(2):15--24, 1980.

\end{thebibliography}
\bibliographystyle{abbrv}

\end{document}